\documentclass{amsart}
\usepackage{amsfonts,amscd,amsthm,amsgen,amsmath,amssymb}
\usepackage{float}
\usepackage[all]{xy}
\usepackage{epsfig,color}
\usepackage{tikz}
\usetikzlibrary{calc}
\usepackage[vcentermath]{youngtab}

\newtheorem{theorem}{Theorem}[section]
\newtheorem{lemma}[theorem]{Lemma}
\newtheorem{corollary}[theorem]{Corollary}
\newtheorem{proposition}[theorem]{Proposition}

\theoremstyle{definition}
\newtheorem{definition}[theorem]{Definition}
\newtheorem{example}[theorem]{Example}

\theoremstyle{remark}
\newtheorem{remark}[theorem]{Remark}

\numberwithin{equation}{section}

\begin{document}

\setlength\parskip{0.5em plus 0.1em minus 0.2em}

\title[Floer homotopy type and eta invariants of Seifert $3$-manifolds]{Floer homotopy type and eta invariants of Seifert $3$-manifolds fibering over $\mathbb{RP}^2$}

\author{David Baraglia}
\address{School of Mathematical Sciences, Adelaide University, Adelaide SA 5005, Australia}
\email{david.baraglia@adelaide.edu.au}

\author{Pedram Hekmati}
\address{Department of Mathematics, The University of Auckland, Auckland, 1010, New Zealand}
\email{p.hekmati@auckland.ac.nz}


\date{\today}

\begin{abstract}

We compute the Floer homology and Seiberg--Witten Floer homotopy type of Seifert rational homology $3$-spheres which fiber over $\mathbb{RP}^2$. We show that they are all $L$-spaces and their Floer homotopy type is a suspension of $S^0$. Additionally, we compute the Ozsv\'ath--Szab\'o $d$-invariants, or equivalently the Seiberg--Witten $\delta$-invariants for such $3$-manifolds. This is done by computing the eta invariant of spin$^c$-Dirac operators associated to spin$^c$-connections covering the adiabatic connection, a certain metric connection distinct from the Levi--Civita connection. It turns out that this eta invariant involves a contribution given by the eta invariant of an orbifold pin$^c$-connection on the orbifold base of the Seifert fibration, which we also compute.

\end{abstract}

\maketitle


\section{Introduction}

A Seifert $3$-manifold which is a rational homology $3$-sphere must fiber over $S^2$ or $\mathbb{RP}^2$. For the Seifert $3$-manifolds which fiber over $S^2$, their Floer homology has been studied extensively. Ozsv\'ath--Szab\'o showed that the Heegaard Floer homology of such Seifert manifolds is isomorphic to lattice homology, which can be computed algorithmically \cite{os}, \cite{nem2}. Dai--Sasahira--Stoffregen promoted this to a computation of the Seiberg--Witten Floer homotopy type \cite{dss}.

In this paper we will compute the Seiberg--Witten Floer homology and homotopy type of the Seifert rational homology $3$-spheres which fiber over $\mathbb{RP}^2$. It turns out that these are all $L$-spaces and their homotopy type is a suspension of $S^0$, so the problem of computing their Floer homology is reduced to that of computing their Ozsv\'ath--Szab\'o $d$-invariants, or equivalently their Seiberg--Witten $\delta$-invariants, $\delta(Y,\mathfrak{s}) = d(Y,\mathfrak{s})/2$. To do this, we compute the eta invariants of the spin$^c$-Dirac operators on such $3$-manifolds. To take advantage of the Seifert geometry $Y \to \Sigma$, we follow the approach of Mrowka--Ozsv\'ath--Yu \cite{moy} in which the Seiberg--Witten equations are studied using spin$^c$-connections which cover a metric connection which is not the Levi--Civita connection. For the computation of the eta invariant of the Dirac operator, we follow a similar approach to that used by Nicolaescu \cite{nic, nic2} in the case that the base is orientable, however a new feature emerges. The non-orientability on the base orbifold $\Sigma$ contributes an additional term to the eta invariant of the Dirac operator on $Y$. This additional term is the eta invariant $\eta^{pin^c}(\Sigma , \mathfrak{s}^\Sigma)$ of a corresponding orbifold pin$^c$-Dirac operator on $\Sigma$. Using indirect methods (the Fr\o yshov inequality), we are able to compute $\eta^{pin^c}(\Sigma , \mathfrak{s}^\Sigma)$ and in turn compute the delta invariants of $Y$.

Our first main result is that Seifert rational homology $3$-spheres with $\mathbb{RP}^2$ base are Seiberg--Witten $L$-spaces, by which we mean that their Seiberg--Witten Floer homology has no reducibles. In fact a stronger result is true which is that the Seiberg--Witten Floer homotopy type is a suspension of $S^0$. That such Seifert manifolds are Heegaard Floer homology $L$-spaces has been shown by Boyer--Gordon--Watson \cite[Proposition 5]{bgw}. We use $SWF(Y , \mathfrak{s})$ to denote the Seiberg--Witten Floer homotopy type of $(Y , \mathfrak{s})$, $HSW^*(Y , \mathfrak{s}) = \widetilde{H}_{S^1}^*( SWF(Y , \mathfrak{s}) )$ the Seiberg--Witten Floer cohomology of $(Y, \mathfrak{s})$ and $HSW^*_{red}(Y , \mathfrak{s})$ the reducible part of $HSW^*(Y , \mathfrak{s})$. We do not specify which coefficient group we are using to define Seiberg--Witten Floer cohomology because our results are not sensitive to this choice.

\begin{theorem}\label{thm:swfY}
Let $Y$ be a Seifert rational homology $3$-sphere with $\mathbb{RP}^2$ base. Then $Y$ is an $L$-space, that is, $HSW^*_{red}(Y , \mathfrak{s}) = 0$ for every spin$^c$-structure $\mathfrak{s}$. Moreover the Seiberg--Witten Floer homotopy type of $(Y,\mathfrak{s})$ is a suspension of $S^0$:
\[
SWF(Y , \mathfrak{s}) \cong \Sigma^{\delta(Y,\mathfrak{s})\mathbb{C}} S^0.
\]
\end{theorem}

Our next result gives a formula for the delta invariants of Seifert rational homology $3$-spheres with $\mathbb{RP}^2$ base. To state the result we need to describe the set of spin$^c$-structures on such $3$-manifolds. Such Seifert $3$-manifolds are classified by an Euler class $b$ and the Seifert invariants of the singular fibres $(a_1,b_1), \dots , (a_n,b_n)$, (see Section \ref{sec:seifert}). We write $Y = S(b ; (a_1,b_1) , \dots , (a_n,b_n))$ for such a $3$-manifold. There are $4a_1 \cdots a_n$ spin$^c$-structures on $Y$, which can be described by a tuple $(E_0 , \tau , m)$ as we now briefly describe (more details can be found in Section \ref{sec:spincY}).

Let $\Sigma$ denote the base of $Y$, which is an orbifold whose underlying topological space is $\mathbb{RP}^2$. $\Sigma$ has orbifold points of orders $a_1, \dots , a_n$. Associated to $\Sigma$ is an oriented double cover orbifold $\widetilde{\Sigma}$ which has $2n$ orbifold points of orders $a_1,a_1,a_2,a_2 , \dots , a_n,a_n$. Pulling back $Y$ under $\widetilde{\Sigma} \to \Sigma$ yields a Seifert fibration $\pi : \widetilde{Y} \to \widetilde{\Sigma}$ over $\widetilde{\Sigma}$. Let $p : \widetilde{Y} \to Y$ be the covering map. Then for any spin$^c$-structure $\mathfrak{s}$ on $Y$ the pullback $p^*(\mathfrak{s})$ necessarily has the form $p^*(\mathfrak{s}) \cong E \otimes \mathfrak{s}_{can}$, where $\mathfrak{s}_{can}$ is the canonical spin$^c$-structure on $\widetilde{Y}$ determined by the Seifert geometry and $E$ is a line bundle. Then $E \cong \pi^*(E_0)$ for some orbifold line bundle $E_0$ whose Seifert invariants are $( e ; \gamma_1 , \delta_1 , \dots , \gamma_n , \delta_n)$, where $e = deg(E_0)$ and $\gamma_i, \delta_i \in \mathbb{Z}_{a_i}$. It turns out that we may choose $E_0$ such that there is an integer $m \in \{-1,0\}$ for which
\begin{equation}\label{equ:cond1}
e = -\chi(\Sigma) - ml,
\end{equation}
where $\chi(\Sigma) = 1 - \sum_{i=1}^n (a_i - 1)/a_i$, $l = b - \sum_{i=1}^{n} b_i/a_i$ and
\begin{equation}\label{equ:cond2}
\gamma_ i + \delta_i = -1 + m b_i \; ({\rm mod} \; a_i) \text{ for } i = 1, \dots , n.
\end{equation}
Notice that for each choice of $m$ there are exactly $a_1 \cdots a_n$ choices for $(\gamma_1 , \delta_1 , \dots , \gamma_n , \delta_n)$, since we can choose the $\gamma_i$ arbitrarily and then solve for the $\delta_i$. Thus, there are $2a_1 \cdots a_n$ pairs $(E_0 , m)$ satisfying conditions (\ref{equ:cond1}), (\ref{equ:cond2}).

Given $(E_0 , m)$ satisfying (\ref{equ:cond1}), (\ref{equ:cond2}), there is an extra piece of data $\tau$, which we call an {\em equivariant structure of weight $m$} which allows us to descend $E \otimes \mathfrak{s}_{can}$ to a spin$^c$-structure on $Y$ (see Definition \ref{def:eqstr}). For each $(E_0 , m)$ there are exactly two equivariant structures on $E_0$ of weight $m$. If $\tau$ is one, then the other is $-\tau$. So in total there are $4a_1 \cdots a_n$ choices for $(E_0 , \tau , m)$, each of which yields a corresponding spin$^c$-structure $\mathfrak{s}_{(E_0 , \tau , m)}$ on $Y$.

The integer $m$ has the following interpretation. To each spin$^c$-structure $\mathfrak{s} = \mathfrak{s}_{(E_0 , \tau , m)}$ on $Y$, there is up to gauge equivalence a unique spin$^c$-connection $B$ which is a reducible solution to the Seiberg--Witten equations for $(Y , \mathfrak{s})$ (with respect to a suitably defined metric on $Y$ and a suitably chosen connection on $TY$, see Section \ref{sec:spincY} for details). The holonomy of $B$ around a fibre of $Y$ is given by $(-1)^m = \pm 1$, thus the spin$^c$-structures with $m=0$ are said to have {\em trivial fibre holonomy} and the spin$^c$-structures with $m=-1$ are said to have {\em non-trivial fibre holonomy}.

The spin$^c$-structures on $Y$ with trivial fibre holonomy ($m=0$), can be interpreted as pullbacks under $\pi : Y \to \Sigma$ of orbifold pin$^c$-structures on $\Sigma$. To each spin$^c$-structure $\mathfrak{s}_{(E_0 , \tau , 0)}$ with trivial fibre holonomy, there is a corresponding orbifold pin$^c$-structure $\mathfrak{s}^\Sigma_{(E_0 , \tau)}$ on $\Sigma$ and this correspondence defines a bijection between orbifold pin$^c$-structures on $\Sigma$ and spin$^c$-structures on $Y$ with trivial fibre holonomy. The spin$^c$-connection $B$ corresponding to $\mathfrak{s}_{(E_0 , \tau , m)}$ descends to a pin$^c$-connection $B_0$ on $\Sigma$. Associated to $B_0$ is a corresponding pin$^c$-Dirac operator $D_{B_0}$ and we let $\eta^{pin^c}(\Sigma , \mathfrak{s}^\Sigma_{(E_0 , \tau)})$ denote the eta invariant of $D_{B_0}$.

Before stating our next theorem there is one more ingredient that we need to introduce. Let $a,b$ be coprime integers with $a > 0$. Then lens space $L(a,b)$ may be defined as the quotient of $S^3$, the unit sphere in $\mathbb{C}^2$ by the $\mathbb{Z}_a$-action $(x , y) \mapsto ( \omega x  ,\omega^b y)$, $\omega = e^{2\pi i /a}$. As explained in Section \ref{sec:deltalens}, there is a bijection $u \mapsto \mathfrak{s}_u$ between $\mathbb{Z}_a$ and the set of spin$^c$-structures on $L(a,b)$. The delta invariants of $L(a,b)$ are given by:
\[
\delta( L(a,b) , \mathfrak{s}_n ) = \lambda(b,a;n) - \frac{1}{2}s(b,a),
\]
where $s(b,a)$ is the Dedekind sum
\[
s(b,a) = -\frac{1}{4a} \sum_{j=1}^{a-1} \frac{ (1+\omega^{-j})(1+\omega^{-bj})}{(1-\omega^{-j})(1-\omega^{-bj}) }, \quad \omega = e^{2\pi i /a}
\]
and $\lambda(b,a;n)$ is defined by
\[
\lambda(b,a;n) = \frac{1}{a} \sum_{j=1}^{a-1} \frac{ \omega^{nj}}{(1-\omega^{-j})(1-\omega^{-bj}) }.
\]

Recall that the degree of $Y \to \Sigma$, which we denote by $l$ is given by
\[
l = b - \sum_{i=1}^n \frac{b_i}{a_i}.
\]

\begin{theorem}\label{thm:dY}
Let $[E_0] = (e ; \gamma_1 , \delta_1 , \dots , \gamma_n , \delta_n)$ and $m \in \{-1,0\}$ satisfy (\ref{equ:cond1}), (\ref{equ:cond2}) and let $\tau$ be an equivariant structure on $E_0$ of weight $m$.
\begin{itemize}
\item[(1)]{If the fibre holonomy is trivial ($m=0$), then
\[
\delta(Y,\mathfrak{s}_{(E_0 , \tau , 0)}) =  \frac{l}{8} - \frac{1}{2} \eta^{pin^c}(\Sigma , \mathfrak{s}_{(E_0 , \tau)}^{\Sigma}) -\frac{1}{2} \sum_{i=1}^n \left( \delta(L(a_i,b_i) , \mathfrak{s}_{\gamma_i}) + \delta( L(a_i,b_i) , \mathfrak{s}_{\delta_i}) \right).
\]
}
\item[(2)]{If the fibre holonomy is non-trivial ($m=-1$), then
\[
\delta(Y , \mathfrak{s}_{(E_0 , \tau , -1)}) = -\sum_{i=1}^{n} \delta( L(a_i , b_i ) , \mathfrak{s}_{\gamma_i}).
\]
}
\end{itemize}

\end{theorem}

In the case of non-trivial fibre holonomy, Theorem \ref{thm:dY} completely determines the delta invariants of $Y$. In the case of trivial fibre holonomy, the theorem determines the delta invariants in terms of $\eta^{pin^c}(\Sigma , \mathfrak{s}_{(E_0 , \tau)}^{\Sigma})$, the pin$^c$-eta invariant of the corresponding pin$^c$-structure on $\Sigma$. The next theorem gives the value of the eta invariant, up to a sign ambiguity.

\begin{theorem}\label{thm:epc}
For any $(E_0 , \tau , 0)$, we have
\[
\eta^{pin^c}(\Sigma , \mathfrak{s}_{(E_0 , \tau)}^{\Sigma}) = \pm \frac{1}{2}.
\]
Moreover, we have that 
\[
\eta^{pin^c}(\Sigma , \mathfrak{s}_{(E_0 , -\tau)}^{\Sigma}) = -\eta^{pin^c}(\Sigma , \mathfrak{s}_{(E_0 , \tau)}^{\Sigma}).
\]
\end{theorem}

\begin{remark}
Recall that for any pair $(E_0 , m)$ satisfying (\ref{equ:cond1}), (\ref{equ:cond2}), there are exactly two equivariant structures on $E_0$ of weight $m$ and they come in a $\pm$ pair $\tau , -\tau$. Theorem \ref{thm:epc} says that the pin$^c$-eta invariant equals $+1/2$ for one of $\tau, -\tau$ and equals $-1/2$ for the other. Thus, one could use the value of $\eta^{pin^c}$ as a means of distinguishing between the two orbifold pin$^c$-structures $\mathfrak{s}^\Sigma_{(E_0 , \tau)}, \mathfrak{s}^\Sigma_{(E_0 , -\tau)}$ and also as a means of distinguishing between the two spin$^c$-structures $\mathfrak{s}_{(E_0 , \tau , 0)}$ and $\mathfrak{s}_{(E_0 , -\tau , 0)}$.

The proof of Theorem \ref{thm:etapinc} also gives a means of determining the correct sign of the eta invariant.
\end{remark}

\begin{example}
The Hantzsche--Wendt manifold is the unique flat $3$-manifold which is a rational homology sphere. It is Seifert fibered and is given by $Y = S( 1 ; (2,1) , (2,1) )$. Since $4a_1 \cdots a_n = 16$, there are sixteen spin$^c$-structures. By Theorem \ref{thm:dY}, the delta invariants for the spin$^c$-structures with non-trivial fibre holonomy are given by:
\[
\delta(Y , \mathfrak{s}_{(E_0 , \tau , -1)} ) = -\delta( L(2,1) , \mathfrak{s}_{\gamma_1}) - \delta( L(2,1) , \mathfrak{s}_{\gamma_2}).
\]
Since $\delta( L(2,1) , \mathfrak{s}_0 ) = 1/8$, $\delta( L(2,1) , \mathfrak{s}_1) = -1/8$, the delta invariants with non-trivial fibre holonomy are $\pm 1/4$ (two times each) and $0$ (four times).

For the spin$^c$-structures with trivial fibre holonomy, $\gamma_i + \delta_i = 1 \; ({\rm mod} \; 2)$, so $\delta( L(2,1) , \mathfrak{s}_{\gamma_i}) + \delta( L(2,1) , \mathfrak{s}_{\delta_i}) = 0$, so from Theorems \ref{thm:dY} and \ref{thm:epc} we have
\[
\delta(Y , \mathfrak{s}_{(E_0 , \tau , 0)}) =  - \frac{1}{2} \eta^{pin^c}(\Sigma , \mathfrak{s}_{(E_0 , \tau)}^{\Sigma}) = \pm \frac{1}{4},
\]
with $\pm 1/4$ both occurring four times (in Section \ref{sec:rp22} we will arrive at this result by a different method that does not use Theorem \ref{thm:epc}).

Therefore the delta invariants of $Y$ are $\pm 1/4$ (six times each) and $0$ (four times).
\end{example}

\subsection{Outline of the proof of the main results}

Central to this paper is the approach to studying the Seiberg--Witten equations on Seifert $3$-manifolds given by Mrowka--Ozsv\'ath--Yu \cite{moy}. However, this description is only available when the base of the Seifert fibration is orientable (and therefore admits a complex structure). To get around this, we work equivariantly on the double cover $\widetilde{Y}$ which is the pullback of the Seifert fibration $Y \to \Sigma$ under the orientation double cover $\widetilde{\Sigma} \to \Sigma$. Thus, for example, instead of working directly with spin$^c$-structures on $Y$, we work with equivariant spin$^c$-structures on $\widetilde{Y}$.

On $\widetilde{Y}$, a spinor field $\psi$ can be decomposed into components $\psi = (\alpha , \beta)$, where $\alpha, \beta$ are sections of certain line bundles. Mrowka--Ozsv\'ath--Yu show that solutions to the Seiberg--Witten equations (with respect to a suitable choice of metric and a suitable choice of connection on the tangent bundle) have either $\alpha = 0$ or $\beta = 0$. On the other hand, we show that the covering involution $\sigma \colon \widetilde{Y} \to \widetilde{Y}$ exchanges the roles of $\alpha$ and $\beta$, so that if one of them vanishes then the other must too. This means that the only solution to the Seiberg--Witten equations is the reducible (unique up to gauge equivalence). 

We prove that the reducible is non-degenerate. This implies that the Seiberg--Witten Floer homotopy type $SWF(Y , \mathfrak{s})$ is a suspension of $S^0$, so $SWF(Y , \mathfrak{s}) \cong \Sigma^{-n(Y,\mathfrak{s}) \mathbb{C}} S^0$ for some $n(Y , \mathfrak{s}) \in \mathbb{Q}$. If we were working with spin$^c$-connections covering the Levi--Civita connection, then $n(Y , \mathfrak{s})$ would be equal to the quantity $n(Y , \mathfrak{s} , g)$ defined by Manolescu in \cite{man}, which is a combination of eta invariants for the Dirac operator and signature operator. However, since we are working with spin$^c$-connections that do not cover the Levi--Civita connection, the definition of $n(Y , \mathfrak{s} , g )$ has to be modified. We show in Section \ref{sec:swf} that the correct extension of $n(Y , \mathfrak{s} , g)$ to the case of general spin$^c$-connections is given by
\begin{equation}\label{equ:n}
n(Y,\mathfrak{s} , g ,B ) = \frac{1}{2}\eta_{dir}(g,B) - \frac{1}{2}h(g,B) - \frac{1}{8}\eta_{sig}(Y,g) - \int_Y Tr(\nabla , \nabla^{LC,g}).
\end{equation}
See Section \ref{sec:swf} for precise definitions of the terms involved. Compared with the usual formula in the case that $B$ covers the Levi--Civita connection, there is an additional term -$\int_Y Tr(\nabla , \nabla^{LC,g})$. Here $Tr(\nabla , \nabla^{LC,g})$ is the transgression form from $\nabla$ (the connection on $TY$ induced by $B$) to $\nabla^{LC,g}$ (the Levi--Civita connection), given by 
\[
Tr(\nabla^0 , \nabla^1) = \frac{1}{96\pi^2} tr\left( \omega \wedge F + \frac{1}{2} \omega \wedge d^{\nabla^0} \omega + \frac{1}{3} \omega \wedge \omega \wedge \omega \right),
\]
where $F$ is the curvature of $\nabla$ and $\omega = \nabla^{LC,g} - \nabla$.

Since $SWF(Y , \mathfrak{s}) \cong \Sigma^{-n(Y,\mathfrak{s},g,B) \mathbb{C}} S^0$, it follows that the delta invariant $\delta(Y , \mathfrak{s})$ is given by
\[
\delta(Y , \mathfrak{s}) = -n(Y , \mathfrak{s} , g , B)
\]
where $g$ is a suitably defined metric on $Y$ and $B$ is the spin$^c$-connection which is the unique reducible solution of the Seiberg--Witten equations. Therefore to compute $\delta(Y , \mathfrak{s})$, it suffices to compute the four terms on the right hand side of (\ref{equ:n}). In fact $h(g,B) = 0$, which leaves only three terms. The transgression term and $\eta_{sig}$ are relatively easy to compute and so the main task is to compute $\eta_{dir}(g,B)$. This takes up a substantial part of the paper. We follow the strategy of Nicolaescu \cite{nic}, \cite{nic2} for computing $\eta_{dir}(g,B)$. However the non-orientability of $\Sigma$ enters in a crucial way and we find a contribution to $\eta_{dir}(g,B)$ which comes from the eta invariant of the Dirac operator on $\Sigma$ for an orbifold pin$^c$-structure. The pin$^c$-eta invariant is computed in Section \ref{sec:epc}. This is achieved by applying the Frøyshov inequality to a suitably chosen negative definite cobordism from $Y$ to a  non-singular circle bundle over $\mathbb{RP}^2$. For a spin$^c$-structure determined by an equivariant structure $\tau$, this allows us to obtain $\pm \frac 12$ as an upper bound for the pin$^c$-eta invariant. By applying the  inequality a second time to the flipped structure $-\tau$, we obtain the reverse inequality.

\subsection{Structure of the paper}

In Section \ref{sec:swf} we recall the construction of the Seiberg--Witten Floer homotopy type of a rational homology $3$-sphere and explain how to adapt the construction to spin$^c$-connections that do not cover the Levi--Civita connection. In Section \ref{sec:seifert} we review some basic results and establish notation concerning Seifert $3$-manifolds which fiber over $S^2$ or $\mathbb{RP}^2$. In Section \ref{sec:swfht}, we determine the Seiberg--Witten Floer homotopy type of Seifert rational homology $3$-spheres $Y$ which fiber over $\mathbb{RP}^2$. To do this, we first need a convenient description of spin$^c$-structures on $Y$. This is carried out in Section \ref{sec:spincY}. The Seiberg--Witten equations on $Y$ are studied in Section \ref{equ:swe}, where Theorem \ref{thm:swfY} is deduced. As described earlier, this reduces us to the computation of the delta invariants $\delta(Y , \mathfrak{s})$, which is carried out in Section \ref{sec:dY}. We do this by computing the terms $\eta_{sig}(Y,g), \int_Y Tr(\nabla , \nabla^{LC,g})$ and $\eta_{dir}(g,B)$ appearing on the right hand side of Equation \ref{equ:n}. These are computed in Sections \ref{sec:etasig}, \ref{sec:transgr} and \ref{sec:deltalens} respectively. As a result we obtain Theorem \ref{thm:dY}. In Section \ref{sec:epc}, we determine the pin$^c$-eta invariant using the Fr\o yshov inequality. We also compute the pin$^c$-eta invariant in the non-hyperbolic cases $\mathbb{RP}^2$, $\mathbb{RP}^2(a)$ and $\mathbb{RP}^2(2,2)$ by independent means.

\noindent{\bf Acknowledgments.} The first author was financially supported by an Australian Research Council Future Fellowship, FT230100092.  The second author was supported by the Royal Society of New Zealand Marsden Fund 3733339 and is grateful to Max Planck Institute for Mathematics in
Bonn for its hospitality.

\section{Seiberg--Witten Floer spectra for general spin$^c$-connections}\label{sec:swf}

In this section we briefly recall the construction of the Seiberg--Witten Floer homotopy type of a rational homology sphere, following Manolescu \cite{man}. We explain how to modify the construction for spin$^c$-connections which are not lifts of the Levi--Civita connection. Since only minor modifications are required, we will be brief and refer the reader to \cite{man} for further details.

Let $Y$ be a rational homology $3$-sphere, $g$ a Riemannian metric on $Y$ and $\mathfrak{s}$ a spin$^c$-structure on $Y$ with corresponding spinor bundle $S$. Let $L = det(S)$ be the determinant line of $S$. Let $\nabla$ be a metric compatible connection on $TY$ and let $Conn_{\nabla}(S)$ denote the set of spin$^c$-structures on $S$ which lift $\nabla$. Then $Conn_{\nabla}(Y)$ is an affine space over $i\Omega^1(Y)$. If $B$ is a spin$^c$-connection we let $F_B$ denote the $\mathfrak{u}(1)$ part of the curvature of $B$. Equivalently, $F_B$ is equal to half of the curvature of the connection on $L$ induced by $B$. If $B$ is a spin$^c$-connection, we let $D_B$ denote the corresponding Dirac operator.

We will first consider the case that $\nabla = \nabla^{LC,g}$ is the Levi--Civita connection of $g$ and then we will explain how to modify the construction to the general case. Define the {\em configuration space} of $(Y , \mathfrak{s})$ to be
\[
C(Y) = Conn_\nabla(S) \times \Gamma(Y , S).
\]
Fix a reference connection $B_0$. We assume $B_0$ is chosen to that $F_{B_0} = 0$ (this is always possible as $Y$ is a rational homology sphere). Then any spin$^c$-connection may be written as $B = B_0 + b$ for a unique $i\Omega^1(Y)$ and so we may identify the configuration space with $i\Omega^1(Y) \times \Gamma(Y , S)$. Thus an element of $C(Y)$ will be written as a pair $(b , \phi)$, where $b$ is an imaginary $1$-form and $\phi \in \Gamma(Y,S)$. The {\em Chern--Simons--Dirac functional} $\mathcal{L} : C(Y) \to \mathbb{R}$ (with respect to $B_0$) is defined by
\[
\mathcal{L}(b,\phi) = \frac{1}{2} \left( \int_Y \langle \phi , D_{B_0+b} \phi \rangle dvol_Y - \int_Y b \wedge db \right).
\]

The gauge group $\mathcal{G} = Map(Y , S^1)$ acts by gauge transformations on $C(Y)$, namely $g(b,\phi) = (b - g^{-1}dg , g\phi)$. Define the {\em global Coulomb slice} $V$ of $C(Y)$ (with respect to $B_0$) to be the subspace $Ker(d^*) \times \Gamma(Y , S) \subset C(Y)$. If one defines the reduced gauge group $\mathcal{G}_0$ to be the set of gauge transformations of the form $g = e^{if}$ where $\int_Y f dvol_Y = 0$, then $V$ can be identified with the quotient $C(Y)/\mathcal{G}_0$. This determines a metric $\widetilde{g}$ on $V$ by taking the natural $L^2$-metric on $C(Y)$ and restricting it to the subbundle of the tangent bundle which is orthogonal to the gauge orbits. The Chern--Simons--Dirac functional $\mathcal{L}$ is gauge invariant and hence defines a functional on $V$, which is just $\mathcal{L}|_V$. Fix an integer $k \ge 4$ and let $V_k$ denote the $L^2_k$-Sobolev completion of $V$. The (formal) gradient of $\mathcal{L}|_V$ with respect to $\widetilde{g}$ defines a map $\chi : V_k \to V_{k-1}$ which has the form $\chi = l + c$, where $l$ is the linear operator given by
\[
l(b,\phi) = (*db , D_{B_0}\phi)
\]
and $c$ is a compact non-linear operator (we will not need to know the precise form of $c$).

A trajectory for the downward gradient flow of $\mathcal{L}|_V$ is a differentiable map $x : \mathbb{R} \to V_k$ satisfying
\[
\frac{d}{dt} x(t) = -\chi(x(t)).
\]

The construction of the Seiberg--Witten homotopy type proceeds by taking a finite-dimensional approximation of $V$ and a finite-dimensional approximation of the downward gradient flow. Given real numbers $\mu,\lambda$, let $V^\mu_\lambda$ denote the direct sum of all eigenspaces of $l$ in the range $(\lambda , \mu]$ and let $\widetilde{p}^\mu_\lambda$ be the $L^2$-orthogonal projection from $V$ to $V^\mu_\lambda$. For technical reasons it is useful to replace the projections $\widetilde{p}^\mu_\lambda$ by smoothed out versions $p^\mu_\lambda : V \to V^\mu_\lambda$ \cite[\textsection 4]{man}. The precise details of this smoothing procedure will not be important to us.

On $V^\mu_\lambda$ one can consider trajectories $x(t)$ satisfying an approximate gradient flow equation
\[
\frac{d}{dt} x(t) = -( l + p^\mu_\lambda c)x(t).
\]
Solutions to these equations are called {\em approximate Seiberg--Witten trajectories}. Manolescu showed that if $\lambda, -\mu$ are sufficiently large, then there exists an $R > 0$ such that if an approximate trajectory $x : \mathbb{R} \to L^2_{k+1}(V^\mu_\lambda)$ lies in the ball $\overline{B(2R)}$, then in fact $x(t) \in B(R)$ for all $t$.

Since $V^\mu_\lambda$ is non-compact, the flow lines of the vector field $(l+p^\mu_\lambda c)$ might not exist for all time. To get around this one fixes a compactly supported cut-off function $u^\mu_\lambda$ which is identically $1$ on $B(3R)$ and considers the vector field $u^\mu_\lambda( l + p^\mu_\lambda c)$. Since this vector field is compactly supported, the flow exists for all time and defines a $1$-parameter family of diffeomorphisms $\varphi_t$ of $V^\mu_\lambda$. Let $S^\mu_\lambda$ denote the union of the set of critical points and flow lines of $l + p^\mu_\lambda c$ which lie in $B(R)$. This is an isolated invariant set for the flow $\varphi_t$ (see \cite[\textsection 5]{man}). The $S^1$-action on $V$ by constant gauge transforms induces an $S^1$-action on $V^\mu_\lambda$ and this action preserves the approximate gradient flow. Therefore we may take the $S^1$-equivariant Conley index
\[
I^\mu_\lambda = I_{S^1}(S^\mu_\lambda).
\]
$I^\mu_\lambda$ is an $S^1$-equivariant homotopy type. The $S^1$-equivariant homotopy type of $I^\mu_\lambda$ can vary with $\mu,\lambda$, but only by a suspension. To cancel the dependence on $\mu,\lambda$ one defines the following $S^1$-equivariant spectrum
\[
SWF( Y , \mathfrak{s} , g , B_0) = \Sigma^{-V^0_\lambda(g,B_0)} I^\mu_\lambda(g,B_0).
\]
Here we write $V^0_\lambda(g,B_0)$ and $I^\mu_\lambda(g,B_0)$ instead of $V^0_\lambda, I^\mu_\lambda$ to highlight the dependence on $g$ and $B_0$.

The spectrum $SWF(Y , \mathfrak{s} , g , B_0)$ depends on $g$ and $B_0$ so does not yet define an invariant of $(Y , \mathfrak{s})$. If one varies $(g,B_0)$ along a path $(g(t) , B_0(t))$, then one finds that
\[
SWF( Y , \mathfrak{s} , g(1) , B_0(1) ) \cong \Sigma^{-SF(-D_t) \mathbb{C}} SWF(Y , \mathfrak{s} , g(0) , B_0(0))
\]
where $SF(-D_t) \in \mathbb{Z}$ is the spectral flow of the family $\{-D_t\}$. Here $D_t$ denotes the Dirac operator associated to $g(t)$ and $B_0(t)$.

In order to get an invariant of $(Y , \mathfrak{s})$, one needs to ``split" the spectral flow. That is, we seek a quantity $n(Y , \mathfrak{s}, g , B_0) \in \mathbb{Q}$ such that $SF(-D_t) = n(Y , \mathfrak{s} , g(0) , B_0(0)) - n(Y , \mathfrak{s} , g(1) , B_0(1))$. Assuming for the moment that such a quantity has been defined, we can then set
\[
SWF(Y , \mathfrak{s}) = \Sigma^{-n(Y,\mathfrak{s} , g , B_0) \mathbb{C}} SWF(Y , \mathfrak{s} , g , B_0)
\]
to get an $S^1$-equivariant spectrum whose stable homotopy type depends only on $Y$ and $\mathfrak{s}$. In general, $n(Y , \mathfrak{s} , g , B_0)$ is a rational number, but not an integer, so one has to make sense of the desuspension $\Sigma^{-n(Y,\mathfrak{s},g,B_0)\mathbb{C}}$ by using a suitably defined category of spectra, see \cite[\textsection 6]{man} for details.

Manolescu obtains a splitting of the spectral flow by taking
\begin{equation}\label{equ:nlc}
n(Y , \mathfrak{s} , g , B_0) = \frac{1}{2}\eta_{dir}(g,B_0) - \frac{1}{2}h(g,B_0) - \frac{1}{8}\eta_{sig}(Y,g),
\end{equation}
where $\eta_{dir}(b,B_0)$ denotes the eta invariant of the Dirac operator $D_{g,B_0}$ determined by $g,B_0$, $h = dim(Ker(D_{g,B_0}))$ and $\eta_{sig}(Y,g)$ is the eta invariant of the signature operator on $Y$ with respect to the metric $g$.

We will now consider adapting the construction of $SWF(Y , \mathfrak{s})$ to the case where we use spin$^c$-connections which cover a metric connection $\nabla$ on $Y$ which is not necessarily the Levi--Civita connection. The definition of the configuration space $C(Y) = Conn_\nabla(Y) \times \Gamma(Y , S)$ is identical to the Levi--Civita case, except we use $Conn_\nabla(Y)$, the set of spin$^c$-structures which are lifts of $\nabla$. If $B \in Conn_\nabla(Y)$, then we let $D_B$ denote the associated Dirac operator. Fix a reference spin$^c$-connection $B_0$, where as before we assume that $F_{B_0} = 0$.

The Coulomb slice $V$ and its finite-dimensional approximations $V^\mu_\lambda$ are defined in the same way as in the Levi--Civita case and this leads to a Conley index $I^\mu_\lambda(g,B_0)$ and spectrum $SWF(Y , \mathfrak{s} , g , B_0)$. To eliminate the dependence of $SWF(Y , \mathfrak{s} , g , B_0)$ on $(g,B_0)$, we again want to split the spectral flow of the Dirac operator. The only difference is that now our Dirac operators are defined using spin$^c$-connections which do not necessarily lift the Levi--Civita connection. Thus we seek a quantity $n(Y , \mathfrak{s} , g , B_0) \in \mathbb{Q}$ defined for all spin$^c$-connections, not just those lifting the Levi--Civita connection. The quantity $n$ should satisfy the following two properties:
\begin{itemize}
\item[(1)]{If $B_0$ is a lift of the Levi--Civita connection, then $n(Y , \mathfrak{s} , g , B_0)$ is given by Equation (\ref{equ:nlc}).}
\item[(2)]{For any path $(g(t) , B_0(t))$, the spectral flow of the associated family of Dirac operators is given by $SF(-D_{g(t),B_0(t)}) = n(Y , \mathfrak{s} , g(0) , B_0(0)) - n(Y , \mathfrak{s} , g(1) , B_0(1))$.}
\end{itemize}

Property (2) ensures that the spectrum $SWF(Y , \mathfrak{s}) = \Sigma^{-n(Y ,\mathfrak{s} , g , B_0)} SWF(Y , \mathfrak{s} , g , B_0)$ does not depend on the choice of $(g,B_0)$. Property (1) ensures that our definition of $SWF(Y , \mathfrak{s})$ agrees with Manolescu's spectrum.

Let $M$ be a Riemannian manifold and let $\nabla^0,\nabla^1$ be two metric connections on $M$. The {\em transgression form} from $\nabla^0$ to $\nabla^1$ is the $3$-form
\[
Tr(\nabla^0 , \nabla^1) = \frac{1}{96\pi^2} tr\left( \omega \wedge F + \frac{1}{2} \omega \wedge d^{\nabla^0} \omega + \frac{1}{3} \omega \wedge \omega \wedge \omega \right),
\]
where $F$ is the curvature of $\nabla^0$, $\omega = \nabla^1 - \nabla^0$ and $tr$ denotes the trace $tr : \Omega^*( M , End(TM) ) \to \Omega^*(M)$ of differential form-valued sections of $End(M)$. Clearly $Tr(\nabla , \nabla) = 0$ for any connection $\nabla$.

\begin{definition}\label{def:n}
Given $(Y,\mathfrak{s} , g, B)$, we set
\[
n(Y,\mathfrak{s} , g ,B ) = \frac{1}{2}\eta_{dir}(g,B) - \frac{1}{2}h(g,B) - \frac{1}{8}\eta_{sig}(Y,g) - \int_Y Tr(\nabla , \nabla^{LC,g}) \in \mathbb{R}
\]
where $B$ is a spin$^c$-connection which lifts $\nabla$ and $\nabla^{LC,g}$ is the Levi--Civita connection for $g$.
\end{definition}

Notice that $n(Y , \mathfrak{s} , g , B)$ agrees with Equation (\ref{equ:nlc}) when $B$ covers the Levi--Civita connection.

From Definition \ref{def:n} it is not at all clear that the quantity $n(Y , \mathfrak{s}  ,g , B)$ should be a rational number. However, this turns out to be the case. Moreover, we can show that $n$ splits the spectral flow of the Dirac operator:

\begin{proposition}
Let $g_t$ be a path of metrics on $Y$ and $B_t$ a path of spin$^c$-connections, where $B_t$ covers a metric connection $\nabla^t$ on $Y$. Then
\[
SF(-D_{g_t,B_t}) = n(Y , \mathfrak{s} , g_0 , B_0) - n(Y , \mathfrak{s} , g_1  ,B_1).
\]
Moreover, $n(Y , \mathfrak{s} , g , B) \in \mathbb{Q}$ for every $g,B$.
\end{proposition}
\begin{proof}

We first argue that the spectral flow $SF(-D_{g_t,B_t})$ depends only on the endpoints. Indeed, given two paths with the same endpoints, we can join up the two paths to form a loop of Dirac operators. The difference of spectral flows of the two paths is given by the index of the resulting elliptic operator on $S^1 \times Y$ \cite[Theorem 7.4]{aps3}. But this index will be zero because $Y$ is a rational homology sphere.

Since the spectral flow depends only on the endpoints, we can therefore assume that the path $(g_t,B_t)$ has the following properties:
\begin{itemize}
\item[(1)]{$(g_t,B_t)$ is constant near $0$ and $1$.}
\item[(2)]{Let $X$ be the cylinder $X = [0,1] \times Y$ with metric $\widehat{g} = (dt)^2 + g_t$. Then there is a spin$^c$-connection $\widehat{B}$ on $X$ which is cylindrical near the boundary and $\widehat{B}|_{ \{t \} \times Y} = B_t$ for $t=0,1$.}
\end{itemize}

The APS index theorem on $X$ then gives
\begin{equation}\label{equ:aps1}
SF( -D_{g_t,B_t}) = -\frac{1}{2} \eta(D_1) + \frac{1}{2}h(D_1) + \frac{1}{2}\eta(D_0) - \frac{1}{2} h(D_0) + \frac{1}{24} \int_{X} p_1(\widehat{\nabla}),
\end{equation}
where $D_0 = D_{g_0,B_0}$, $D_1 = D_{g_1,B_1}$ $h(D_i) = dim(Ker(D_i))$ and $p_1(\widehat{\nabla}) = -\dfrac{1}{8\pi^2}tr( F_{\widehat{\nabla}}^2)$ is the first Pontryagin form of $\widehat{\nabla}$, the connection on $TX$ induced by $\widehat{B}$.

Let $\widehat{\nabla}^{LC}$ denote the Levi--Civita connection on $X$ for $\widehat{g}$. Using 
\[
p_1(\widehat{\nabla}^{LC}) - p_1(\widehat{\nabla}) = -24 d Tr( \widehat{\nabla} , \widehat{\nabla}^{LC}),
\] 
we have
\begin{align}\label{equ:aps2}
\int_X p_1(\widehat{\nabla}^{LC}) - \int_X p_1(\widehat{\nabla}) &= -24 \int_X d Tr( \widehat{\nabla} , \widehat{\nabla}^{LC}) \\
& = -24 \int_Y Tr( \nabla^1 , \nabla^{LC , g_1}) + 24 \int_Y Tr(\nabla^0 , \nabla^{LC , g_0}). \nonumber
\end{align}

The APS signature theorem on $X$ gives
\begin{equation}\label{equ:aps3}
0 = \frac{1}{3} \int_{X} p_1(\widehat{\nabla}^{LC} ) - \eta_{sig}(g_1) + \eta_{sig}(g_0).
\end{equation}
Combining Equations (\ref{equ:aps1}), (\ref{equ:aps2}), (\ref{equ:aps3}), we get
\begin{align*}
SF(-D_{g_t,B_t}) &= -\frac{1}{2}\eta(D_1) + \frac{1}{2}h(D_1) + \frac{1}{8}\eta_{sig}(g_1) + \frac{1}{2}\eta(D_0) - \frac{1}{2}h(D_0) - \frac{1}{8}\eta_{sig}(g_0) \\
& \quad \quad - \frac{1}{24} \int_{X} \left( p_1(\widehat{\nabla}^{LC}) - p_1(\widehat{\nabla}) \right) \\
&= -\frac{1}{2}\eta(D_1) + \frac{1}{2}h(D_1) + \frac{1}{8}\eta_{sig}(g_1) + \frac{1}{2}\eta(D_0) - \frac{1}{2}h(D_0) - \frac{1}{8}\eta_{sig}(g_0) \\
& \quad \quad + \int_Y Tr( \nabla^1 , \nabla^{LC , g_1}) - \int_Y Tr(\nabla^0 , \nabla^{LC , g_0}) \\
& = n(Y, \mathfrak{s} , g_0 , B_0) - n(Y , \mathfrak{s} , g_1 , B_1).
\end{align*}

It remains to prove rationality. Since $n(Y, \mathfrak{s} , g_0 , B_0) - n(Y , \mathfrak{s} , g_1 , B_1) = SF(-D_{g_t,B_t}) \in \mathbb{Z}$, it suffices to show that $n(Y,\mathfrak{s} , g,B) \in \mathbb{Q}$ for a single choice of $(g,B)$. But in the case that $B$ is a lift of the Levi--Civita connection, rationality of $n(Y,\mathfrak{s},g,B)$ was shown in \cite[\textsection 6]{man}.
\end{proof}

\section{Seifert $3$-manifolds with $\mathbb{RP}^2$ base}\label{sec:seifert}

In this section we establish some general results concerning the topology of $3$-manifolds which are Seifert fibrations over $\mathbb{RP}^2$. It will be necessary to also consider Seifert $3$-manifolds which fibre over $S^2$ and so we consider these first.

Let $M$ be an oriented Seifert $3$-manifold which fibers over $S^2$. So $M$ is a circle bundle over a $2$-dimensional oriented orbifold $S$ whose underlying space is the $2$-sphere. Such Seifert manifolds are classified by a tuple $(b ; (a_1,b_1) , \dots , (a_n,b_n))$, where $b,a_1,b_1, \dots , a_n,b_n$ are integers, $a_i > 0$, $0 < b_i < a_i$ and $a_i,b_i$ are coprime for each $i$. We use the notation $M = M(b ; (a_1, b_1) , \dots , (a_n,b_n))$. As a smooth orbifold, $S$ is classified by the orders $a_1, \dots , a_n$ of the orbifold points. Let $x_1, \dots , x_n \in S$ be the orbifold points. Choose an orbifold metric on $S$. This makes $S$ into a complex orbifold. Our conventions for Seifert manifolds are such that $M = S(N)$ is the unit circle bundle of the orbifold line bundle $N$ associated to the divisor $bx - b_1 x_1 - \cdots - b_n x_n$, where $x$ is any non-singular point of $S$. We can drop the conditions that $0 < b_i < a_i$, but then we have isomorphisms (as Seifert fibrations over $S$) $M(b ; (a_1,b_1) , \dots , (a_n,b_n)) \cong M(b+1 ; (a_1, b_1) , \dots , (a_i , b_i + a_i) , \dots , (a_n,b_n))$. 

Now we consider Seifert fibrations over $\mathbb{RP}^2$. Let $Y$ be an oriented Seifert $3$-manifold which fibers over $\mathbb{RP}^2$, so $Y$ is a circle bundle $\pi : Y \to \Sigma$ over a $2$-dimensional orbifold whose underlying space is $\mathbb{RP}^2$. By \cite[Theorem 2, \textsection 5.2]{or}, such Seifert manifolds are classified by a tuple $(b ; (a_1, b_1) , \dots , (a_n , b_n) )$, where $b,a_1,b_1,$ $ \dots ,$ $a_n,b_n$ are integers, $a_i > 0$, $0 < b_i < a_i$ and $a_i,b_i$ are coprime for each $i$. We denote the corresponding Seifert $3$-manifold by $Y = S( b ; (a_1,b_1) , \dots , (a_n,b_n))$. In the notation of \cite{or}, our manifold $S(b ; (a_1,b_1) , \dots (a_n,b_n))$ has Seifert invariants $\{ -b ; (n_2 , 1) ; (a_1 , b_1) , \dots , (a_n , b_n) \}$. Note that $Y$ is always a rational homology $3$-sphere.

Since the underlying topological space $|\Sigma|$ of $\Sigma$ is $\mathbb{RP}^2$, its oriented double cover is $S^2$. Orbifolds charts on $\Sigma$ can be lifted to the double cover and this gives $S^2$ an orbifold structure which we denote by $\widetilde{\Sigma}$. Let $p : \widetilde{\Sigma} \to \Sigma$ denote the covering map and $\sigma : \widetilde{\Sigma} \to \widetilde{\Sigma}$ the covering involution. As a smooth orbifold, $\Sigma$ is classified by the orders $a_1, \dots , a_n$ of the orbifold points and we sometimes write $\Sigma \cong \mathbb{RP}^2(a_1 , \dots , a_n)$. Similarly, $\widetilde{\Sigma} \cong S^2(a_1,a_1, \dots , a_n , a_n)$. Let $\pi : \widetilde{Y} \to \widetilde{\Sigma}$ denote the pullback of $Y$ under $p : \widetilde{\Sigma} \to \Sigma$. Then $\widetilde{Y}$ is an oriented Seifert manifold over $\widetilde{\Sigma}$ and it is easily seen that $\widetilde{Y} \cong M( 2b ; (a_1 , b_1) , (a_1 , b_1) , \dots , (a_n , b_n) , (a_n , b_n))$. We will use the same notation $p,\sigma$ to denote the covering map $p : \widetilde{Y} \to Y$ and covering involution $\sigma : \widetilde{Y} \to \widetilde{Y}$.

\begin{proposition}
$H_1(Y ; \mathbb{Z})$ is finite of order $4a_1 \cdots a_n$. Hence $Y$ has exactly $4a_1 \cdots a_n$ spin$^c$-structures.
\end{proposition}
\begin{proof}
A presentation of $\pi_1(Y)$ is given by \cite[\textsection 5.3]{or}:
\[
\pi_1(Y) \cong \langle v , q_1 , \dots , q_n , h \; | \; vhv^{-1}h, [q_j , h], q_j^{a_j} h^{b_j}, q_1 \cdots q_n v^2 h^b \rangle
\]
and therefore
\[
H_1(Y ; \mathbb{Z}) \cong \frac{\mathbb{Z}[ v , q_1 , \dots , q_n , h]}{\langle 2h , a_j q_j + b_j h , q_1 + \cdots + q_n + 2v + bh \rangle} \cong \frac{\mathbb{Z}^{n+2}}{\varphi (\mathbb{Z}^{n+2})}
\]
where $\varphi : \mathbb{Z}^{n+2} \to \mathbb{Z}^{n+2}$ has matrix
\[
A = \left[ \begin{matrix} a_1 & & & b_1 & 0 \\ & \ddots & & \vdots & \vdots \\ & & a_n & b_n & 0 \\ 1 & \cdots & 1 & b & 2 \\ 0 & \cdots & 0 & 2 & 0  \end{matrix} \right].
\]
It is easily seen that $|det(A)| = 4a_1 \cdots a_n$ and hence $H_1(Y ; \mathbb{Z})$ is finite of order $4a_1 \cdots a_n$. Since $H^2(Y ; \mathbb{Z}) \cong H_1(Y ; \mathbb{Z})$, there are $4a_1 \cdots a_n$ spin$^c$-structures on $Y$.

\end{proof}

\section{Seiberg--Witten Floer homotopy type of $Y$}\label{sec:swfht}

\subsection{Spin$^c$-structures on $Y$ and pin$^c$-structures on $\Sigma$}\label{sec:spincY}

In order to study the Seiberg--Witten equations on $Y$ we wish to find a convenient description of spin$^c$-structures on $Y$. However this is complicated by the fact that $Y$ fibers over the non-orientable orbifold $\Sigma$. We find it is easier to first pass to the double cover $\widetilde{Y}$, where spin$^c$-structure can easily be defined in terms of orbifold line bundles on $\widetilde{\Sigma}$ and then descend to $Y$.

Let $g_\Sigma$ be an orbifold metric on $\Sigma$. We will always take $g_\Sigma$ to be a metric of constant curvature. Let $g_{\widetilde{\Sigma}}$ be the pullback of $g_{\Sigma}$ to $\widetilde{\Sigma}$. Fix an orientation on $\widetilde{\Sigma}$. The orientation on $\widetilde{\Sigma}$ together with the orbifold metric $g_{\widetilde{\Sigma}}$ determines a complex structure which makes $\widetilde{\Sigma}$ into a complex orbifold and the covering involution $\sigma \colon \widetilde{\Sigma} \to \widetilde{\Sigma}$ is then an antiholomorphic involution.

The Seifert fibration $Y \to \Sigma$ can be regarded as the $S^1$-bundle associated to a principal $O(2)$-orbifold bundle $P \to \Sigma$. Let $N = P \times_{O(2)} \mathbb{R}^2$ be the associated $\mathbb{R}^2$-bundle. Then $Y = S(N)$ is the unit circle bundle of $N$. Since $Y$ is oriented, the orientation class of $N$ coincides with the orientation class of $T\Sigma$. Let $\widetilde{N}$ denote the pullback of $N$ to $\widetilde{\Sigma}$. The orientation on $\widetilde{\Sigma}$ then induces an orientation on $\widetilde{N}$ and hence a reduction of structure to $SO(2) \cong U(1)$. Thus $\widetilde{N}$ can be regarded as a complex orbifold line bundle. Then $\widetilde{Y} = S(\widetilde{N})$ carries a natural circle action given by complex multiplication on the fibres of $\widetilde{N}$. Define the degree $l = deg(N)$ of $N$ to be
\[
l = b - \sum_{i=1}^n \frac{b_i}{a_i}.
\]
Set $\widetilde{l} = 2l$. Then $\widetilde{l} = deg(\widetilde{N})$ is the orbifold degree of $\widetilde{N}$.

Let $\partial$ denote the vector field on $\widetilde{Y}$ which is the tangent at $\theta = 0$ of the circle action $( e^{i \theta} , y ) \mapsto e^{i\theta}y$. The flow generated by $\partial$ has period $2\pi$. Let $i\eta$ be a connection form for the orbifold circle bundle $\widetilde{Y} \to \widetilde{\Sigma}$. That is, $\eta$ is a real $1$-form on $\widetilde{Y}$ which is invariant under the circle action and satisfies $\iota_{\partial} \eta = 1$. We can choose $i\eta$ to have constant curvature, that is, $d\eta = 2 \xi vol_{\widetilde{\Sigma}}$, where
\begin{equation}\label{equ:xi}
\xi = -\frac{ \pi \widetilde{l} }{Vol(\widetilde{\Sigma})} = - \frac{\pi l}{Vol(\Sigma)}.
\end{equation}

Notice that since $\sigma$ reverses fibre orientation, we have $\sigma_*(\partial) = -\partial$. Furthermore, since $\sigma$ reverses orientation on $\widetilde{\Sigma}$, we have $\sigma^*(vol_{\widetilde{\Sigma}}) = -vol_{\widetilde{\Sigma}}$. Replacing $\eta$ by $(\eta - \sigma^*(\eta))/2$, we can assume that the connection $i\eta$ is chosen so that $\sigma^*(\eta) = -\eta$ and $d\eta = 2 \xi vol_{\widetilde{\Sigma}}$.

Let $r > 0$ and define a metric $g_{\widetilde{Y}}$ on $\widetilde{Y}$ by
\[
g_{\widetilde{Y}} = r^2 \eta^2 + \pi^*(g_{\widetilde{\Sigma}}).
\]
With this metric the non-singular fibres of $\widetilde{Y} \to \widetilde{\Sigma}$ have length $2\pi r$ and thus $Vol(\widetilde{Y}) = 2\pi r Vol(\widetilde{\Sigma})$, where $Vol(\widetilde{Y})$ is the volume of $\widetilde{Y}$ with respect to $g_{\widetilde{\Sigma}}$. Since $\sigma^*(\eta) = -\eta$ and $\sigma^*( g_{\widetilde{\Sigma}}) = g_{\widetilde{\Sigma}}$, it follows that $g_{\widetilde{Y}}$ is $\sigma$-invariant and thus descends to a metric $g_Y$ on $Y$. Throughout this paper we always take $Y$ to be equipped with a metric of this form.

The kernel of $\eta$ can be identified with the pullback $\pi^*(T\widetilde{\Sigma})$ of the tangent bundle of $\widetilde{\Sigma}$ and thus we get orthogonal splittings
\[
T\widetilde{Y} = \mathbb{R}\partial \oplus \pi^*(T\widetilde{\Sigma}), \quad T^*\widetilde{Y} = \mathbb{R}\eta \oplus \pi^*(T^*\widetilde{\Sigma}).
\]
Under the action of $\sigma$, the vertical tangent bundle $V_{\widetilde{Y}} = \mathbb{R}\partial$ (a trivial line bundle on $\widetilde{Y}$) descends to a flat line bundle $V_Y$ on $Y$, the vertical tangent bundle of $\pi \colon Y \to \Sigma$. We thus get similar orthogonal splittings
\[
TY = V_Y \oplus \pi^*(T\Sigma), \quad T^*Y = V_Y^* \oplus \pi^*(T^*\Sigma).
\]
We let $\nabla^{LC,g_Y}$, $\nabla^{LC,g_{\widetilde{Y}}}$ denote the Levi--Civita connections on $Y$ and $\widetilde{Y}$ with respect to the metrics $g_Y, g_{\widetilde{Y}}$. 

Following \cite{moy}, we introduce the {\em adiabatic connection} $\nabla^Y$ on $TY$ given by $\nabla^Y = \nabla^{V_Y} \oplus \pi^*(\nabla^{\Sigma})$, where $\nabla^{V_Y}$ is the flat line bundle on the vertical tangent bundle $V_Y$ and $\nabla^\Sigma$ is the Levi--Civita connection on $\Sigma$. Similarly we define $\nabla^{\widetilde{Y}} = d \oplus \pi^*(\nabla^{\widetilde{\Sigma}})$ (notice that $\nabla^{\widetilde{Y}}$ is the pullback of $\nabla^Y$ to $\widetilde{Y}$). When it is clear whether we are referring to $Y$ or $\widetilde{Y}$ we will write the Levi--Civita connection as $\nabla^{LC}$ and the adiabatic connection as $\nabla$.

Recall that $\widetilde{Y} = S(\widetilde{N})$ is the unit circle bundle of $\widetilde{N}$. Therefore $\widetilde{Y}$ is the boundary of the complex orbifold $\widetilde{W} = D(\widetilde{N})$, the unit disc bundle of $\widetilde{N}$. Since $\widetilde{W}$ is a complex orbifold it carries a canonical orbifold spin$^c$-structure, which then restricts to a spin$^c$-structure on $\widetilde{Y}$ which we call the {\em canonical spin$^c$-structure} on $\widetilde{Y}$ and denote it by $\mathfrak{s}_{can}$. The spinor bundle for $\mathfrak{s}_{can}$ is given by $S_{can} = \mathbb{C} \oplus K_{\widetilde{\Sigma}}^{-1}$ and Clifford multiplication $\rho : T^*\widetilde{Y} \to End(S_{can})$ is given by
\begin{equation}\label{equ:Cliff}
\rho( v ) = \left[ \begin{matrix} 0 & -\sqrt{2} v^{1,0} \\ \sqrt{2} h_{\widetilde{\Sigma}}^{-1} v^{0,1} & 0 \end{matrix} \right]
\end{equation}
for $v \in \pi^*( T^* \widetilde{\Sigma})$ and
\[
\rho( \eta ) = \left[ \begin{matrix} \frac{i}{r} & 0 \\ 0 & -\frac{i}{r} \end{matrix} \right].
\]
In the above formula for $\rho(v)$, $v^{1,0} = (v - iIv)/2$, $v^{0,1} = (v+iIv)/2$, so that $v = v^{1,0} + v^{0,1}$ with $v^{1,0}$ valued in $K_{\widetilde{\Sigma}}$, $v^{0,1}$ valued in $\overline{K}_{\widetilde{\Sigma}}$ and $h_{\widetilde{\Sigma}}$ denotes the Hermitian metric on $T\widetilde{\Sigma}$ determined by $g_{\widetilde{\Sigma}}$, regarded as a section of $\overline{K}_{\widetilde{\Sigma}} K_{\widetilde{\Sigma}}$.

In what follows, we seek to descend spin$^c$-structures on $\widetilde{Y}$ to spin$^c$-structures on $Y$ and descend orbifold spin$^c$-structures on $\widetilde{\Sigma}$ to orbifold pin$^c$-structures on $\Sigma$. As a first step towards this aim, we observe that the action of $\sigma$ on $K_{\widetilde{\Sigma}}$ and on $\widetilde{N}$ can be captured by sections of $\sigma^*(K_{\widetilde{\Sigma}}^*) K_{\widetilde{\Sigma}}^*$ and $\sigma^*(\widetilde{N})\widetilde{N}$ as follows. Since $\sigma$ is an antiholomorphic isometry of $\widetilde{\Sigma}$, the derivative of $\sigma$ defines an isomorphism $\sigma^*(\overline{T\widetilde{\Sigma}}) \cong T\widetilde{\Sigma}$. On the other hand, the Hermitian metric $h_{\widetilde{\Sigma}}$ defines an isomorphism $\overline{T\widetilde{\Sigma}} \cong T\widetilde{\Sigma}^*$. By composing these two isomorphisms we get an isomorphism $\sigma^*(T\widetilde{\Sigma}^*) \cong T\widetilde{\Sigma}$, or equivalently a non-vanishing section $\mu_{\widetilde{\Sigma}}$ of $\sigma^*( K_{\widetilde{\Sigma}}^* ) K_{\widetilde{\Sigma}}^*$. Since $\sigma$ is an isometry, one finds that $\mu_{\widetilde{\Sigma}}$ is $\sigma$-invariant, that is, $\sigma^*( \mu_{\widetilde{\Sigma}}) = \mu_{\widetilde{\Sigma}}$. In a similar manner, we have an isomorphism $\sigma^*( \widetilde{N}^* ) \cong \overline{\widetilde{N}}$, which may be identified with a $\sigma$-invariant section $\mu_{\widetilde{N}}$ of $\sigma^*(\widetilde{N}) \otimes \widetilde{N}$.

Let $E$ be a complex line bundle on $\widetilde{Y}$. Then we obtain a corresponding spin$^c$-structure $\mathfrak{s}_E$ on $\widetilde{Y}$ given by $\mathfrak{s}_E = E \otimes \mathfrak{s}_{can}$. This map defines a bijection between line bundles and spin$^c$-structures on $\widetilde{Y}$. The spinor bundle for $\mathfrak{s}_E$ is given by 
\[
S_E = E \otimes S_{can} = E \oplus (E \otimes K_{\widetilde{\Sigma}}^{-1}).
\]
We have that $c(\mathfrak{s}_E) = 2c_1(E) - c_1( K_{\widetilde{\Sigma}})$.

We would like to find a description of spin$^c$-structures on $Y$ similar to the above description of spin$^c$-structures on $\widetilde{Y}$. Unfortunately since $\Sigma$ is not oriented, there is no sensible notion of a canonical spin$^c$-structure on $Y$. To get around this problem we work equivariantly on $\widetilde{Y}$. Suppose that $\mathfrak{s}$ is a spin$^c$-structure on $Y$. The pullback $p^*(\mathfrak{s})$ is a spin$^c$-structure on $\widetilde{Y}$ and hence must be isomorphic to $\mathfrak{s}_E$ for some line bundle $E$ on $\widetilde{Y}$. Furthermore, since $\mathfrak{s}_E$ is a pullback, we must have that $\sigma^*(\mathfrak{s}_E) \cong \mathfrak{s}_E$. Since $\sigma^*(\mathfrak{s}_{can}) \cong K_{\widetilde{\Sigma}} \otimes \mathfrak{s}_{can}$, we must have $\sigma^*(E) \cong E K_{\widetilde{\Sigma}}^{-1}$. Recall that every line bundle $E$ on $\widetilde{Y}$ is isomorphic to a pullback $E \cong \pi^*(E_0)$ of an orbifold line bundle $E_0$ on $\widetilde{\Sigma}$ and moreover $E_0$ is unique up to tensoring by powers of $\widetilde{N}$ \cite[Theorem 2.0.19]{moy}. Therefore we must have $\sigma^*(E_0) \cong E_0 K_{\widetilde{\Sigma}} \widetilde{N}^m$ for some $m$. Additionally, for $\mathfrak{s}_{E}$ to descend to $Y$ we need the isomorphism $\sigma^*(E_0) \cong E_0 K_{\widetilde{\Sigma}} \widetilde{N}^m$ to induce a $\mathbb{Z}_2$-action on $S_{E}$. This leads us to the following: 

\begin{definition}\label{def:eqstr}
Let $E_0$ be an orbifold line bundle on $\widetilde{\Sigma}$. An {\em equivariant structure on $E_0$ of weight $m$} is a section $\tau$ of $E_0 \sigma^*(E_0^*)  K_{\widetilde{\Sigma}}^{*} \widetilde{N}^m$ such that $\tau \sigma^*(\tau) = - \mu_{\widetilde{\Sigma}} \mu_{\widetilde{N}}^m$. Let $E_0,E'_0$ be two orbifold line bundles on $\widetilde{\Sigma}$ and let $\tau$ be an equivariant structure on $E_0$ of weight $m$ and $\tau'$ an equivariant structure on $E'_0$ of weight $m'$. We say that $(E_0 , \tau , m)$ and $(E'_0 , \tau' , m')$ are {\em equivalent} if $m' - m = 2r$ is even and there is a line bundle isomorphism $E'_0 \cong E_0 \widetilde{N}^{-r}$ under which $\tau'$ is sent to $\tau \mu_{\widetilde{N}}^{r}$.
\end{definition}

According to the above definition, if $\tau$ is an equivariant structure on $E_0$ of weight $m$, then $\tau \mu_{\widetilde{N}}$ is an equivariant structure on $E_0 \otimes \widetilde{N}^{-1}$ of weight $m+2$ and $(E_0 , \tau , m ) \sim (E_0 \widetilde{N}^{-1} , \tau \mu_{\widetilde{N}} , m+2 )$. Let $\mathcal{T}(\widetilde{\Sigma})$ be the set of equivalence classes of triples $(E_0 , \tau , m)$ where $\tau$ is an equivariant structure on $E_0$ of weight $m$. We can then write $\mathcal{T}(\widetilde{\Sigma}) = \mathcal{T}_0(\widetilde{\Sigma}) \cup \mathcal{T}_1(\widetilde{\Sigma})$, where $\mathcal{T}_i(\widetilde{\Sigma})$ is the set of equivalence classes of the form $(E_0 , \tau , m)$ with $m = i \; ({\rm mod} \; 2)$. Using the equivalences $(E_0 , \tau , m ) \sim (E_0 \widetilde{N}^{-1} , \tau \mu_{\widetilde{N}} , m+2 )$, every element of $\mathcal{T}_i(\widetilde{\Sigma})$ can be represented by a triple of the form $(E_0 , \tau , -i)$.

Let $U \to \Sigma$ be a complex orbifold line bundle on $\Sigma$ and let $\widetilde{U}$ be the pullback of $U$ to $\widetilde{\Sigma}$. Then $\widetilde{U}$ is an orbifold line bundle on $\widetilde{\Sigma}$ and $\sigma$ lifts to an involution on $\widetilde{U}$, or equivalently, there exists a non-vanishing section $\tau_U$ of $\sigma^*(\widetilde{U}^*) \widetilde{U}$ with the property that $\tau_U \sigma^*(\tau_U) = 1$. Conversely, if $\widetilde{U}$ is an orbifold line bundle on $\widetilde{\Sigma}$ with a section $\tau_U$ of $\sigma^*(\widetilde{U}^*)$ satisfying $\tau_U \sigma^*(\tau_U) = 1$, then $\tau_U$ gives $\widetilde{U}$ the structure of an equivariant orbifold line bundle on $\widetilde{\Sigma}$. Now if $(E_0 , \tau , m)$ is an equivariant structure on $E_0$ of weight $m$ and $(\widetilde{U} , \tau_U)$ is an equivariant line bundle on $\widetilde{\Sigma}$, then $\tau \tau_U$ is an equivariant structure on $E_0 \widetilde{U}$ of weight $m$. In this way, we see that for each fixed $m$, the set of equivalence classes of tuples $(E_0 , \tau , m)$ is (if non-empty) a torsor for the group of equivariant orbifold line bundles on $\widetilde{\Sigma}$, that is, a torsor for the group $Pic^t(\Sigma)$ of orbifold line bundles on $\Sigma$.

\begin{proposition}
Both $\mathcal{T}_0(\widetilde{\Sigma})$ and $\mathcal{T}_1(\widetilde{\Sigma})$ are non-empty, hence they are torsors for $Pic^t(\Sigma)$. Furthermore, $Pic^t(\Sigma)$ is a finite group of order $2a_1 \cdots a_n$. Hence $| \mathcal{T}_0(\widetilde{\Sigma}) | = |\mathcal{T}_1(\widetilde{\Sigma}) | = 2a_1 \cdots a_n$ and $|\mathcal{T}(\widetilde{\Sigma})| = 4a_1 \cdots a_n$.
\end{proposition}
\begin{proof}
Suppose $E_0$ is an orbifold line bundle on $\widetilde{\Sigma}$. A necessary condition for $E_0$ to admit an equivariant structure of weight $m$ is that $E \sigma^*(E^*) K_{\widetilde{\Sigma}}^* \widetilde{N}^m \cong 1$. We will first show that such line bundles exists for any $m$. By the equivalence $(E_0 \tau , m) \sim (E_0 \widetilde{N}^{-1} , \tau \mu_{\widetilde{N}} , m+2)$, it suffices to consider the cases $m = 0,-1$.

Recall \cite{moy} that $Pic^t(\widetilde{\Sigma}) \cong \{ (e ; \gamma_1 , \delta_1 , \dots , \gamma_n , \delta_n) \in \mathbb{Q} \oplus \bigoplus_{i=1}^n \mathbb{Z}_{a_i}^2 \; | \; e = \sum_i (\gamma_i + \delta_i)/a_i \; ({\rm mod} \; \mathbb{Z}) \}$. Suppose that $[E_0] = (e ; \gamma_1 , \delta_1 , \dots , \gamma_n , \delta_n)$. Then $[\sigma^*(E_0)] = (-e ; -\delta_1 , -\gamma_1 , \dots , -\delta_n , -\gamma_n)$. We also have $[K_{\widetilde{\Sigma}}] = ( -\chi(\widetilde{\Sigma}) ; a_1-1 , a_1 - 1 , \dots , a_n - 1 , a_n - 1)$ and $[\widetilde{N}] = ( \widetilde{l} ; -b_1 , -b_1 , \dots , -b_n , -b_n)$.

Consider the case $m=0$. We want to solve $E_0 \sigma^*(E_0^*) \cong K_{\widetilde{\Sigma}}$. This holds if and only if $e = -(1/2)\chi(\widetilde{\Sigma}) = -\chi(\Sigma)$ and for each $i$, $\gamma_i + \delta_i = -1 \; ({\rm mod} \; a_i)$. There are exactly $a_1 \cdots a_n$ solutions given by choosing $\gamma_1, \dots , \gamma_n$ arbitrarily and setting $\delta_i = -1 - \gamma_i$ for each $i$. Note that $( -\chi(\Sigma) ; \gamma_1 , -1-\gamma_1 , \dots , \gamma_n , -1-\gamma_n)$ defines an isomorphism class in $Pic^t(\widetilde{\Sigma})$ because
\[
e = -\chi(\Sigma) = -1 + \sum_{i=1}^n \left( 1 - \frac{1}{a_i} \right) = \sum_{i=1}^{n} \frac{(\gamma_i + \delta_i)}{a_i} \; ({\rm mod} \; \mathbb{Z}).
\]
So there are $a_1 \cdots a_n$ isomorphisms classes of line bundles $E_0$ satisfying $E_0 \sigma^*(E_0^*) \cong K_{\widetilde{\Sigma}}^*$.

Similarly for $m=-1$ we want to solve $E_0 \sigma^*(E_0^*) \cong K_{\widetilde{\Sigma}} \widetilde{N}$. If $[E_0] = (e ; \gamma_1 , \delta_1 , \dots , \gamma_n , \delta_n)$ then this is equivalent to $e = -\chi(\Sigma) + l$ and $\gamma_i + \delta_i = -1 - b_i \; ({\rm mod} \; a_i)$ for each $i$. Again there are $a_1 \cdots a_n$ solutions given by choosing $\gamma_1, \dots , \gamma_n$ arbitrarily and setting $\delta_i = -1 - b_i - \gamma_i$.

Now suppose that $E_0$ is an orbifold line bundle on $\widetilde{\Sigma}$ such that $E_0 \sigma^*(E_0^*) K_{\widetilde{\Sigma}}^* \widetilde{N}^m$ is trivial. We will prove that $E_0$ admits an equivariant structure of weight $m$. Since $E_0 \sigma^*(E_0^*) K_{\widetilde{\Sigma}}^* \widetilde{N}^m$ is trivial, there exists non-vanishing section $\tau_0$ of $E_0 \sigma^*(E_0^*) K_{\widetilde{\Sigma}}^* \widetilde{N}^m$. Then $\tau_0 \sigma^*(\tau_0) = - f \mu_{\widetilde{\Sigma}} \mu_{\widetilde{N}}^m$ for some function $f :  \widetilde{\Sigma} \to \mathbb{C}^*$. Since $\tau_0 \sigma^*(\tau_0)$, $\mu_{\widetilde{\Sigma}}$ and $\mu_{\widetilde{N}}$ are $\sigma$-invariant, we must have $f \circ \sigma = f$. The underlying topological space of $\widetilde{\Sigma}$ is the $2$-sphere, which is simply-connected, so we can write $f = e^{2\pi i h}$ for some function $h : \widetilde{\Sigma} \to \mathbb{C}$. Since $\sigma^*(f) = f$, we must have that $\sigma^*(h) = h + k$ for some $k \in \mathbb{Z}$. But $h = \sigma^*(\sigma^*(h)) = h + 2k$, so $k=0$ and therefore $\sigma^*(h) = h$. Set $\tau = \tau_0 e^{-\pi i h}$. Then $\tau \sigma^*(\tau) = \tau_0 \sigma^*(\tau_0) e^{-2\pi i h} = -\mu_{\widetilde{\Sigma}} \mu_{\widetilde{N}}^m$ and hence $\tau$ is an equivariant structure on $E_0$ of weight $m$.

It remains to show that $Pic^t(\Sigma)$ has order $2a_1 \cdots a_n$. Recall that $Pic^t(\Sigma) = H^2_{orb}( \Sigma ; \mathbb{Z})$ is degree $2$ orbifold cohomology of $\Sigma$. Observe that $H^*_{orb}(\Sigma ; \mathbb{Z}) \cong H^*_{O(2)}( \widetilde{Y} ; \mathbb{Z})$ and similarly $H^*_{orb}(\widetilde{\Sigma} ; \mathbb{Z}) \cong H^*_{S^1}(\widetilde{Y} ; \mathbb{Z})$. From $S^1 \to O(2) \to \mathbb{Z}_2$ we have a spectral sequence $E_r^{p,q}$ converging to $H^*_{O(2)}(\widetilde{Y} ; \mathbb{Z}) \cong H^*_{orb}(\Sigma ; \mathbb{Z})$, where $E^{p,q}_2 = H^p( \mathbb{Z}_2 ; H^q_{S^1}(Y ; \mathbb{Z}) ) \cong H^p(\mathbb{Z}_2 ; H^q_{orb}(\widetilde{\Sigma} ; \mathbb{Z}))$. Since $H^1_{orb}(\widetilde{\Sigma} ; \mathbb{Z}) = 0$, we obtain from the spectral sequence a short exact sequence
\[
0 \to \mathbb{Z}_2 \to Pic^t(\Sigma) \to Pic^t(\widetilde{\Sigma})^{\sigma} \to 0.
\]
So $|Pic^t(\Sigma)| = 2 |Pic^t(\widetilde{\Sigma})^{\sigma}|$. We will show that $| Pic^t(\widetilde{\Sigma})^{\sigma} | = a_1 \cdots a_n$. Suppose $[E_0] = (e ; \gamma_1 , \delta_1 , \dots , \gamma_n , \delta_n)$. Then $[E_0] = [\sigma^*(E_0)]$ if and only if $e = 0$ and $\gamma_i + \delta_i = 0$ for all $i$. There are $a_1 \cdots a_n$ solutions.
\end{proof}

In what follows we will give a bijection between $\mathcal{T}_0(\widetilde{\Sigma})$ and the set of orbifold pin$^c$-structures on $\Sigma$ and between $\mathcal{T}(\widetilde{\Sigma})$ and the set of spin$^c$-structures on $Y$.

Let $(E_0 , \tau , 0 ) \in \mathcal{T}_0(\widetilde{\Sigma})$. Thus $E_0$ is an orbifold line bundle on $\widetilde{\Sigma}$ and $\tau$ is a section of $E_0 \sigma^*(E_0^*) K_{\widetilde{\Sigma}}^*$ such that $\tau \sigma^*(\tau) = -h_{\widetilde{\Sigma}}$. Let $\mathfrak{s}_{can}^{\widetilde{\Sigma}}$ denote the canonical orbifold spin$^c$-structure on $\widetilde{\Sigma}$ and set $\mathfrak{s}^{\widetilde{\Sigma}}_{E_0} = E_0 \otimes \mathfrak{s}_{can}^{\widetilde{\Sigma}}$. The spinor bundle for $\mathfrak{s}_{E_0}^{\widetilde{\Sigma}}$ is given by $S_{E_0} = E_0 \oplus (E_0 K_{\widetilde{\Sigma}}^{-1})$ with Clifford mutliplication defined identically to Equation (\ref{equ:Cliff}). To obtain a pin$^c$-structure on $\widetilde{\Sigma}$, we need to lift $\sigma$ to an involution on $S^{\widetilde{\Sigma}}_{E_0}$ which makes Clifford multiplication $\sigma$-equivariant. We define
\[
\sigma_{\tau} : \Gamma( \widetilde{\Sigma} , S^{\widetilde{\Sigma}}_{E_0} ) \to \Gamma( \widetilde{\Sigma} , S^{\widetilde{\Sigma}}_{E_0})
\]
by
\[
\sigma_{\tau} \left[ \begin{matrix} \alpha \\ \beta \end{matrix} \right] = \left[ \begin{matrix} -\tau \mu_{\widetilde{\Sigma}}^{-1} \sigma^*(\beta) \\ \tau \sigma^*(\alpha) \end{matrix} \right].
\]
Then it is easily checked that $\sigma_{\tau}^2 = id$ and $\sigma_{\tau} \circ \rho(v) = \rho( v' ) \circ \sigma_{\tau}$ for all $1$-forms $v$, where $v'$ denotes the pullback of $v$ under $\sigma$ as a $1$-form. Thus $\sigma_{\tau}$ makes $\mathfrak{s}^{\widetilde{\Sigma}}_{E_0}$ into an equivariant pin$^c$-structure on $\widetilde{\Sigma}$, which then descends to a pin$^c$-structure on $\Sigma$ which we will denote as $\mathfrak{s}^\Sigma_{(E_0,\tau)}$. Let $\mathcal{P}(\Sigma)$ denote the set of isomorphism classes of orbifold pin$^c$-structures on $\Sigma$. We have just defined a map
\begin{equation}\label{equ:pinc}
\psi : \mathcal{T}_0(\widetilde{\Sigma}) \to \mathcal{P}(\Sigma), \quad (E_0 , \tau , 0) \mapsto \mathfrak{s}^{\Sigma}_{(E_0 , \tau_0)}.
\end{equation}

\begin{proposition}
The map $\psi$ given by (\ref{equ:pinc}) defines a bijection from $\mathcal{T}_0(\widetilde{\Sigma})$ to $\mathcal{P}(\Sigma)$.
\end{proposition}
\begin{proof}
This follows immediately since both $\mathcal{T}_0(\widetilde{\Sigma})$ and $\mathcal{P}(\Sigma)$ are torsors for $Pic^t(\Sigma)$ and the map $(E_0 , \tau , 0) \mapsto \mathfrak{s}^{\Sigma}_{(E_0 , \tau_0)}$ respects the torsor structures.
\end{proof}

Now we turn to spin$^c$-structures on $Y$. Let $E_0$ be an orbifold line bundle on $\widetilde{\Sigma}$ and let $\tau$ be an equivariant structure on $E_0$ of weight $m$. Set $E = \pi^*(E_0)$. Since $\widetilde{Y} = S(\widetilde{Y})$ is the unit circle bundle of $\widetilde{Y}$, we have a tautological section $s$ of $\pi^*(\widetilde{N})$ defined by $s(y) = y$. Unravelling the definition of $\mu_{\widetilde{N}}$, one sees that $\mu_{\widetilde{N}}^{-1} s = \sigma^*(s)^{-1}$, or $\mu_{\widetilde{N}} = s \sigma^*(s)$. Define an involution
\[
\sigma_\tau : \Gamma( \widetilde{Y} , S_E) \to \Gamma( \widetilde{Y} , S_E)
\]
by
\begin{equation}\label{equ:sigmatau}
\sigma_{\tau} \left[ \begin{matrix} \alpha \\ \beta \end{matrix} \right] = \left[ \begin{matrix} -\tau s^{-m} \mu_{\widetilde{\Sigma}}^{-1} \sigma^*(\beta) \\ \tau s^{-m} \sigma^*(\alpha) \end{matrix} \right].
\end{equation}

Notice that since $\tau \sigma^*(\tau) = \mu_{\widetilde{\Sigma}} \mu_{\widetilde{N}}^m$, it follows that $(\tau s^{-m}) \sigma^*(\tau s^{-m}) = \mu_{\widetilde{\Sigma}}$ and that $\sigma_\tau^2 = id$. It is easily seen that $\sigma_\tau$ respects Clifford multiplication and thus makes $\mathfrak{s}_E$ into an equivariant spin$^c$-structure, which then descends to a spin$^c$-structure on $Y$ which we denote by $\mathfrak{s}_{(E_0 , \tau , m)}$, or simply by $\mathfrak{s}_{(E_0 , \tau)}$ in the case $m=0$.

Suppose $(E_0 , \tau , m)$, $(E_0' , \tau' , m')$ are equivalent. So $m' - m = 2r$ and there is an isomorphism $\varphi : E'_0 \to E_0 \widetilde{N}^{-r}$ which sends $\tau'$ to $\tau \mu_{\widetilde{N}}$. Set $E' = \pi^*(E)$. Then the isomorphism $\Phi : S_{E'} \to S_{E}$ given by $\Phi(\alpha , \beta) = ( \varphi(\alpha) s^r , \varphi(\beta) s^r)$ satisfies $\Phi \circ \sigma_{\tau'} = \sigma_{\tau} \circ \Phi$. Hence equivalent triples define isomorphic spin$^c$-structures on $Y$. So the map $(E_0 , \tau , m) \to \mathfrak{s}_{(E_0 , \tau , m)}$ defines a map
\[
\psi_Y : \mathcal{T}(\widetilde{\Sigma}) \to \mathcal{S}(Y), \quad (E_0 , \tau , m) \mapsto \mathfrak{s}_{(E_0 , \tau , m)},
\]
where $\mathcal{S}(Y)$ denotes the set of isomorphism classes of spin$^c$-structures on $Y$. We will eventually show that $\psi_Y$ is a bijection.

Let $\mathfrak{s}$ be a spin$^c$-structure on $Y$. Recall that a spin$^c$-connection for $\mathfrak{s}$ is equivalent to specifying a metric connection on $TY$ and a unitary connection on the determinant line bundle associated to $\mathfrak{s}$. Unless stated otherwise we will take our spin$^c$-connection to be a lift of the adiabatic connection $\nabla^Y$. Under this assumption, spin$^c$-connections are in bijection with unitary connections on the determinant line bundle. If $B$ denotes such a unitary connection, we let $F_B$ denote the $\mathfrak{u}(1)$ part of the curvature of $B$. Equivalently, $F_B$ is half the curvature of the induced connection on the determinant line bundle. By a {\em reducible} for $\mathfrak{s}$, we mean a spin$^c$-connection $B$ such that $F_B = 0$. Since $b_1(Y) = 0$, each spin$^c$-structure admits a reducible which is unique up to gauge equivalence. Similarly, on $\widetilde{Y}$ we will only consider spin$^c$-connections which are lifts of the adiabatic connection $\nabla^{\widetilde{Y}}$. If $B$ is such a connection, we again write $F_B$ for one half of the curvature of the induced connection on the determinant line bundle. By a reducible on $\widetilde{Y}$ we mean a spin$^c$-connection $B$ such that $F_B = 0$. Clearly a reducible on $Y$ pulls back to a reducible on $\widetilde{Y}$.

In the case of the canonical spin$^c$-structure $\mathfrak{s}_{can}$ on $\widetilde{Y}$, the determinant line bundle is $K_{\widetilde{\Sigma}}^{-1} \cong T^{1,0} \widetilde{\Sigma}$ which may be equipped with the pullback of the Levi--Civita connection for $\widetilde{\Sigma}$. Thus we have a canonically determined spin$^c$-connection $B_{can}$, the {\em canonical spin$^c$-connection}, with the property that $F_{B_{can}} = -\frac{1}{2} F_{K_{\widetilde{\Sigma}}} = -\pi i \chi(\Sigma) vol_{\widetilde{\Sigma}}/Vol(\Sigma)$. Then if $E$ is any Hermitian line bundle on $\widetilde{Y}$, we have a bijection between unitary connections on $E$ and spin$^c$-connections for $\mathfrak{s}_E$ given by $A \mapsto B = A + B_{can}$. Since $F_B = F_A - (1/2)F_{K_{\widetilde{\Sigma}}}$, we have that $F_B = 0$ if and only if $F_A = (1/2) F_{K_{\widetilde{\Sigma}}}$.

Let $E_0$ be an orbifold line bundle on $\widetilde{\Sigma}$ and let $\tau$ be an equivariant structure on $E_0$ of weight $m$. Set $E = \pi^*(E_0)$. Recall that $\tau$ defines an involution $\sigma_\tau$ on $S_E$. Let $A_0$ denote a unitary connection on $E_0$ and set $B_0 = \pi^*(A_0) + B_{can}$. The $1$-form $i\eta$ defines a connection on $\widetilde{N}$. To be more precise, there is a connection $A_{\widetilde{N}}$ on $\widetilde{N}$ such that, letting $s : \mathbb{C} \to \pi^*(\widetilde{N})$ denote the tautological section of $\widetilde{N}$ over $Y$, we have that $s^*( \pi^*(A_{\widetilde{N}})) = d + i\eta$.

We assume that the connection $A_0$ on $E_0$ has constant curvature, so $F_{A_0} = -2\pi i e \, vol_{\widetilde{\Sigma}}/Vol(\widetilde{\Sigma})$, where $e = deg(E_0)$ is the orbifold degree of $E_0$. We further assume that $A_0$ is chosen compatibly with $\tau$ in the sense that $\tau$ is a constant section of $E_0 \sigma^*(E_0^*) K_{\widetilde{\Sigma}}^* \widetilde{N}^m$, where the connection on $E_0 \sigma^*(E_0^*) K_{\widetilde{\Sigma}}^* \widetilde{N}^m$ is induced by $A_0$, $A_{\widetilde{\Sigma}}$ (the Levi--Civita connection for $\widetilde{\Sigma}$) and $A_{\widetilde{N}}$. To see that this can be done, let $A'_0$ be any connection on $E_0$ with constant curvature. Set $R = E_0 \sigma^*(E_0^*) K_{\widetilde{\Sigma}}^* \widetilde{N}^m$ and let $\nabla^{R,A'_0}$ denote the connection on $R$ induced by $A'_0, A_{\widetilde{\Sigma}}$ and $A_{\widetilde{N}}$. We have $\nabla^{R,A'_0} \tau = \alpha \tau$ for some imaginary $1$-form $\alpha$. Since $\tau \sigma^*(\tau) = \mu_{\widetilde{\Sigma}} \mu_{\widetilde{N}}^m$ is a constant section of $\sigma^*(K_{\widetilde{\Sigma}}^*)K_{\widetilde{\Sigma}}^* \sigma^*(\widetilde{N}^m)\widetilde{N}^m$, it follows that $\alpha + \sigma^*(\alpha) = 0$. Since $\nabla^{R,A'_0}$ has constant curvature, but $deg(R) = 0$ (because it has a non-vanishing section $\tau$), we have that $\nabla^{R,A'_0}$ is a flat connection and hence $d \alpha = 0$. Now set $A_0 = A'_0 - (1/2)\alpha$. Then $A_0$ is also constant curvature because $d\alpha = 0$ and $\nabla^{R,A_0}\tau = (\alpha - (1/2)\alpha + (1/2)\sigma^*(\alpha) )\tau = 0$.

Recall that $\sigma_\tau$ is given by (\ref{equ:sigmatau}). Since under the isomorphism $s \colon \mathbb{C} \to \pi^*(\widetilde{N})$ we have $s^*(A_{\widetilde{N}}) = d + i\eta$, it follows that $\sigma_\tau \circ \nabla_{B_0} = \nabla_{B_m} \circ \sigma_t$, where for a spin$^c$-connection $B$, $\nabla_B$ denotes the corresponding covariant derivative and where $B_m = B_0 + im\eta$.

Let $\lambda \in \mathbb{R}$ and set $A = \pi^*(A_0) + i\lambda \eta$, $B = A + B_{can}$. Then since $\sigma^*(\eta) = -\eta$, we have $\sigma_\tau \circ \nabla_B = (\nabla_{B_0} + im\eta - i\lambda \eta) \circ \sigma_\tau$. Therefore we will have $\sigma_\tau \circ \nabla_B = \nabla_B \circ \sigma_\tau$ provided that $\lambda = m-\lambda$, or $\lambda = m/2$.

Choosing $\lambda = m/2$ we see that the connection $B = A + B_{can}$, $A = \pi^*(A_0) + i\lambda \eta$ commutes with $\sigma_t$ and thus induces a spin$^c$-connection for the spin$^c$-structure $\mathfrak{s}_{(E_0 , \tau , m)}$ on $Y$. By assumption $A_0$ has constant curvature and therefore
\begin{align*}
F_B &= F_{A_0} + i\lambda d\eta - \frac{1}{2} F_{K_{\widetilde{\Sigma}}} \\
&= -\pi i \left( e + ml + \chi(\Sigma) \right) \frac{vol_{\widetilde{\Sigma}}}{Vol(\Sigma)}
\end{align*}

But since $E_0 \sigma^*(E_0^*) K_{\widetilde{\Sigma}}^* \widetilde{N}^m \cong \mathbb{C}$, taking degrees, we have $2e + \chi(\widetilde{\Sigma}) + m\widetilde{l} = 2(e + \chi(\Sigma) + ml) = 0$. Thus $F_B = 0$. So we have constructed for each spin$^c$-structure the corresponding reducible $B$ (recall that $B$ is unique up to gauge equivalence).

\begin{proposition}\label{prop:psiY}
The map $\psi_Y \colon \mathcal{T}(\widetilde{\Sigma}) \to \mathcal{S}(Y)$, $\psi(E_0 , \tau , m) = \mathfrak{s}_{(E_0 , \tau , m)}$ is a bijection.
\end{proposition}
\begin{proof}
Since $|\mathcal{T}(\widetilde{\Sigma})| = |\mathcal{S}(Y)| = 4a_1 \cdots a_n$, it suffices to show that $\psi_Y$ is injective. We first show that $\psi_Y$ is injective when restricted to $\mathcal{T}_0(\widetilde{\Sigma})$ and $\mathcal{T}_1(\widetilde{\Sigma})$. In fact, since $\mathcal{T}_i(\widetilde{\Sigma})$ are torsors for $Pic^t(\Sigma)$ and $\mathcal{S}(Y)$ is a torsor for $H^2(Y ; \mathbb{Z})$, this amounts to showing that the pullback map $\pi^* \colon Pic^t(\Sigma) \to H^2(Y ; \mathbb{Z})$ is injective. For this we use the Gysin sequence for the orbifold circle bundle $\pi : Y \to \Sigma$
\[
\cdots \to H^0_{orb}(\Sigma ; \mathbb{Z}_{orn} ) \to H^2_{orb}(\Sigma ; \mathbb{Z}) \buildrel \pi^* \over \longrightarrow H^2(Y ; \mathbb{Z}) \to H^1_{orb}(\Sigma ; \mathbb{Z}_{orn}) \to \cdots
\]
where $\mathbb{Z}_{orn}$ denotes the orientation local system. Since $H^0_{orb}(\Sigma ; \mathbb{Z}_{orn}) = 0$, injectivity of $\pi^*$ follows.

Now to prove injectivity of $\psi_Y$ it suffices to show that $\psi_Y( \mathcal{T}_0(\widetilde{\Sigma}))$ and $\psi_Y(\mathcal{T}_1(\widetilde{\Sigma}))$ are disjoint. Recall that for each $(E_0 , \tau , m)$, we constructed a spin$^c$-connection $B$ on $\mathfrak{s}_{(E_0 , \tau , m)}$ with $F_B = 0$. The connection $B$ has the form $B = A + B_{can}$, $A = \pi^*(A_0) + i\lambda \eta$ for some connection $A_0$ on $E_0$ and where $\lambda = m/2$. The holonomy of $B$ around a non-singular fibre of $Y \to \Sigma$ is given by $e^{ \pm 2\pi i \lambda}$, where the sign factor $\pm $ depends on which direction we go around the fibre. But since $\lambda = m/2$, the holonomy is equal to $(-1)^m$ in either direction. Suppose that $\psi_Y( (E_0 , \tau_0 , m) ) = \psi_Y( (E'_0 , \tau' , m'))$. Then the corresponding reducibles $B,B'$ must be gauge equivalent and hence must have the same holonomies. So $(-1)^m = (-1)^{m'}$ and hence $m = m' \; ({\rm mod} \; 2)$. This proves that $\psi_Y( \mathcal{T}_0(\widetilde{\Sigma}))$ and $\psi_Y(\mathcal{T}_1(\widetilde{\Sigma}))$ are disjoint.
 
\end{proof}

\begin{remark}
As a consequence of Proposition \ref{prop:psiY}, the set $\mathcal{S}(Y)$ of spin$^c$-structures on $Y$ admits a decomposition $\mathcal{S}(Y) = \mathcal{S}_0(Y) \cup \mathcal{S}_1(Y)$ into disjoint subsets $\mathcal{S}_i(Y) = \psi_Y( \mathcal{T}_i(\widetilde{\Sigma}))$, each having size $2a_1 \cdots a_n$. If $\mathfrak{s} \in \mathcal{S}_m(Y)$, and $B$ is a reducible for $\mathfrak{s}$, that is, a spin$^c$-connection for $\mathfrak{s}$ with $F_B = 0$, then the holonomy of $B$ around a non-singular fibre is $(-1)^m$. So $\mathcal{S}_0(Y)$ is the set of spin$^c$-structures for which the reducible has trivial fibre holonomy and $\mathcal{S}_1(Y)$ is the set of spin$^c$-structures with non-trivial fibre holonomy.
\end{remark}


\subsection{The Seiberg--Witten equations on $Y$}\label{equ:swe}

In this section we consider the Seiberg--Witten equations on $Y$. As in Section \ref{sec:spincY}, it is convenient to describe spin$^c$-structures on $Y$ as equivariant spin$^c$-structures on $\widetilde{Y}$. Then solutions to the Seiberg--Witten equations on $Y$ will correspond to $\sigma$-invariant solutions of the Seiberg--Witten equations on $\widetilde{Y}$.

Let $E_0$ be an orbifold line bundle on $\widetilde{\Sigma}$ and let $\tau$ be an equivariant structure on $E_0$ of weight $m$. Let $E = \pi^*(E_0)$ and $\mathfrak{s}_E = E \otimes \mathfrak{s}_{can}$. Let $S_E = E \oplus (EK_{\widetilde{\Sigma}}^{-1})$ be the spinor bundle for $\mathfrak{s}_E$. As seen in Section \ref{sec:spincY}, $\tau$ determines an involution
\[
\sigma_\tau : \Gamma( \widetilde{Y} , S_E) \to \Gamma( \widetilde{Y} , S_E)
\]
which makes $\mathfrak{s}_E$ into an equivariant spin$^c$-structure on $\widetilde{Y}$, inducing a spin$^c$-structure $\mathfrak{s} = \mathfrak{s}_{(E_0 , \tau , m)}$ on $Y$. We let $S$ denote the spinor bundle of $\mathfrak{s}_{(E_0 , \tau , m)}$.

The Seiberg--Witten equations for $(Y , \mathfrak{s})$ with respect to the metric $g_Y$, the adiabatic connection $\nabla^Y$ and with zero perturbation are equations for a pair $( B , \psi )$, where $B \in Conn_{\nabla^Y}(S)$ is a spin$^c$-connection which induces $\nabla^Y$ on $TY$ and $\psi \in \Gamma( Y , S )$ is a spinor field. The Seiberg--Witten equations are:
\begin{align*}
D_B \psi &= 0, \\
F_B &= \tau(\psi),
\end{align*}
where $\tau(\psi) = \rho^{-1}( (\psi \psi^*)_0 )$ denotes the $2$-form corresponding to the trace-free part of $\psi \psi^*$ under Clifford multiplication.

A solution $(B,\psi)$ of the Seiberg--Witten equations for $(Y , \mathfrak{s})$ with respect to $g_Y$ and $\nabla^Y$ pulls back under $p : \widetilde{Y} \to Y$ to a $\sigma$-invariant solution $(\widetilde{B} , \widetilde{\psi})$ of the Seiberg--Witten equations for $(\widetilde{Y} , \mathfrak{s}_E )$ with respect to $g_{\widetilde{Y}}$ and $\nabla^Y$, where the action on solutions is $(\widetilde{B} , \widetilde{\psi}) \mapsto ( \sigma_t^*(\widetilde{B}) , \sigma_t(\widetilde{\psi}))$. Conversely, since $\sigma$ acts freely, any $\sigma$-invariant solution to the Seiberg--Witten equations on $\widetilde{Y}$ descends to a solution to the Seiberg--Witten equations on $Y$.

\begin{proposition}\label{prop:swsoln}
Any solution $(B , \psi)$ to the Seiberg--Witten equations for $(Y , \mathfrak{s})$ with respect to $g_Y$ and $\nabla^Y$ is reducible, that is, $\psi = 0$. Moreover, there is a unique reducible solution up to gauge equivalence. The reducible $(B , 0)$ is non-degenerate in the sense that $Ker(D_B) = 0$.
\end{proposition}
\begin{proof}
Let $(B , \psi)$ be a solution to the Seiberg--Witten equations on $Y$ and $(\widetilde{B} , \widetilde{\psi})$ its pullback to $\widetilde{Y}$. Then since $S_E = E \oplus (E K_{\widetilde{\Sigma}}^{-1})$, we can write $\widetilde{\psi} = (\alpha , \beta)$. As shown in \cite{moy} any solution to the Seiberg--Witten equations on $\widetilde{Y}$ with respect to the metric $g_{\widetilde{Y}}$ and adiabatic connection $\nabla^{\widetilde{Y}}$ has $\alpha = 0$ or $\beta = 0$. On the other hand, since $\widetilde{\psi}$ is $\sigma_t$-invariant, we have (using Equation (\ref{equ:sigmatau})) $\beta = \tau s^{-m} \sigma^*(\alpha)$. So if either $\alpha$ or $\beta$ is zero, then they are both zero. Hence $\psi = 0$ and $(B , \psi) = (B , 0)$ is a reducible solution. Furthermore, since $b_1(Y) = 0$, the reducible is unique up to gauge equivalence.

It remains to show non-degeneracy of the reducible $B$. Let $\widetilde{B} = p^*(B)$. Since $Ker(D_B) = Ker(D_{\widetilde{B}})^{\sigma_\tau}$, it suffices to show that $Ker(D_{\widetilde{B}}) = 0$. As shown in Section \ref{sec:spincY}, we can take $\widetilde{B}$ to have the form $\widetilde{B} = A + B_{can}$, where $A = \pi^*(A_0) + i(m/2)\eta$ and $A_0$ is a connection on $E_0$ with constant curvature and $m \in \{0,-1\}$.

If $m=-1$, then $\widetilde{B}$ has non-trivial fibre holonomy and hence $Ker(D_{\widetilde{B}}) = 0$.

If $m=0$, then $Ker(D_{\widetilde{B}}) \cong H^0(\widetilde{\Sigma} ; E_0 ) \oplus H^1( \widetilde{\Sigma} ; E_0)$. Here, the connection $A_0$ is used to define a holomorphic structure on $E_0$. Let $r = deg|E_0|$, where $|E_0|$ denotes the desingularisation of $E_0$ \cite[Definition 2.0.2]{moy}. Then $dim( Ker(D_{\widetilde{B}}) ) = r+1$ if $r \ge 0$ and $-r-1$ if $r < 0$.

Let $[E_0] = ( e ; \gamma_1 , \delta_1 , \dots , \gamma_n , \delta_n )$. Since $\sigma^*(E_0) \cong E_0 K_{\widetilde{\Sigma}}^{-1}$, we get that 
\[
e = -\chi(\Sigma) = -1 + \sum_i \frac{(a_i - 1)}{a_i}
\] 
and $\gamma_i + \delta_i = -1 \; ({\rm mod} \; a_i)$. If we take $0 \le \gamma_i ,\delta_i \le a_i - 1$, then it follows that $0 \le \gamma_i + \delta_i \le 2a_i - 2$ and thus we must have $\gamma_i + \delta_i = a_i - 1$. Then it follows that 
\[
r = deg|E_0| = e - \sum_i \frac{(\gamma_i + \delta_i)}{a_i} = e - \sum_i \frac{ (a_i-1)}{a_i} = -1.
\]
So $dim( Ker(D_{\widetilde{B}}) ) = r+1 = 0$.

\end{proof}

\begin{corollary}
The Seiberg--Witten Floer homotopy type of $(Y , \mathfrak{s})$ is given by
\[
SWF(Y , \mathfrak{s}) = \Sigma^{-n(Y , \mathfrak{s} , g_Y , B)} S^0,
\]
where $B$ denotes the unique reducible corresponding to $\mathfrak{s}$. In particular, the delta invariant of $(Y , \mathfrak{s})$ is given by
\[
\delta(Y , \mathfrak{s}) = -n(Y , \mathfrak{s} , g_Y , B).
\]
\end{corollary}
\begin{proof}
Since the only solution to the Seiberg--Witten equations for $(Y , \mathfrak{s})$ with respect to $g_Y, \nabla^Y$ is the reducible $(B,0)$, which is non-degenerate. The methods of \cite[\textsection 6-7]{man2} can then be applied, giving $SWF(Y , \mathfrak{s}) = \Sigma^{-n(Y , \mathfrak{s} , g_Y , B)} S^0$ and therefore $\delta(Y , \mathfrak{s}) = -n(Y , \mathfrak{s} , g_Y , B)$.
\end{proof}

\section{Computing the delta invariants of $Y$}\label{sec:dY}

We continue to use the same setup as in previous sections. $\pi \colon Y \to \Sigma$ is a Seifert $3$-manifold with metric $g_Y$ constructed in Section \ref{sec:spincY}.

\subsection{Computation of $\eta_{sig}$}\label{sec:etasig}

\begin{proposition}\label{prop:etasig}
The eta invariant of the signature operator on $Y$ with metric $g_Y$ is given by
\[
\eta_{sig}(Y , g_Y ) = \frac{2}{3} \frac{ l \pi r^2 \chi(\Sigma) }{Vol(\Sigma)} - \frac{2}{3} \frac{ l^3 \pi^2 r^4}{Vol(\Sigma)^2} + \frac{l}{3} + 4 \sum_{i=1}^{n} s(b_i,a_i).
\]
\end{proposition}
\begin{proof}
For Seifert $3$-manifolds with orientable base, the eta invariant of the signature operator was computed by Ouyang \cite{ou}. Applying this to $\widetilde{Y}$, we get
\begin{align*}
\eta_{sig}( \widetilde{Y} , g_{\widetilde{Y}} ) &= \frac{2}{3} \frac{ \widetilde{l} \pi r^2 \chi(\widetilde{\Sigma}) }{Vol(\widetilde{\Sigma})} - \frac{2}{3} \frac{ \widetilde{l}^3 \pi^2 r^4}{Vol(\widetilde{\Sigma})^2} + \frac{\widetilde{l}}{3} - \epsilon + 8 \sum_{i=1}^{n} s(b_i,a_i) \\
& = \frac{4}{3} \frac{ l \pi r^2 \chi(\Sigma) }{Vol(\Sigma)} - \frac{4}{3} \frac{ l^3 \pi^2 r^4}{Vol(\Sigma)^2} + \frac{2l}{3} - \epsilon + 8 \sum_{i=1}^{n} s(b_i,a_i),
\end{align*}
where $\epsilon = 1$ if $l > 0$, $\epsilon = 0$ if $l=0$ and $\epsilon = -1$ if $l < 0$. Now we compare $\eta_{sig}(Y, g_Y)$ with $\eta_{sig}(\widetilde{Y} , g_{\widetilde{Y}})$. The Seifert manifold $Y = S(b ; (a_1,b_1) , \dots , (a_n , b_n))$ is the boundary of a star shaped plumbing $X$, constructed as follows. The plumbing graph $\Gamma$ will be star shaped. Associated to the central vertex, we take a disc bundle over $\mathbb{RP}^2$ whose total space is oriented and whose Euler class is $b$. For all the remaining vertices of $\Gamma$, the disc bundles will be disc bundles over $S^2$. Attached to the central node are $n$ linear plumbing graphs $\Gamma_1, \dots , \Gamma_n$ whose weights $d^i_1, \dots , d^i_{s_i}$ satisfy $a_i/b_i = [d^i_1 , \dots , d^i_{s_i}]$, where $[d_1 , \dots , d_s]$ denotes the negative continued fraction defined by $[d_1] = d_1$, $[d_1 , \dots , d_s] = d_1 - 1/[d_2 , \dots , d_s]$ for $s > 1$. We then plumb these disc bundles together according to $\Gamma$. It is easily seen by the Seifert--van Kampen theorem that $\pi_1(X) \cong \mathbb{Z}_2$. The universal cover $\widetilde{X} \to X$ is a plumbing with boundary $\widetilde{Y}$ and with star shaped plumbing graph $\widetilde{\Gamma}$ given as follows. The central node has weight $2b$ and attached to the central node are the $2n$ linear plumbing graphs $\Gamma_1, \Gamma_1 , \dots , \Gamma_n , \Gamma_n$. Each vertex corresponds to a disc bundle over $S^2$ (including the central vertex) and they are plumbed together according to $\widetilde{\Gamma}$.

From this it follows easily that the signatures of $X$ are $\widetilde{X}$ are related by
\[
\sigma(\widetilde{X}) = 2 \sigma(X) + \epsilon.
\]

Choose a metric $g_X$ on $X$ which is cylindrical near the boundary and which restricts to $g_Y$ on $Y$ and let $g_{\widetilde{X}}$ be the pullback metric on $\widetilde{X}$. Now we apply the APS index theorem for the signature operators on $X$ and $\widetilde{X}$,
\begin{align*}
\sigma(X) &= \frac{1}{3} \int_X p_1( \nabla^{g_X} ) - \eta_{sig}(Y , g_Y), \\
\sigma(\widetilde{X}) &= \frac{1}{3} \int_{\widetilde{X}} p_1( \nabla^{g_{\widetilde{X}}}) - \eta_{sig}(\widetilde{Y} , g_{\widetilde{Y}}),
\end{align*}
where $\nabla^{g_X}, \nabla^{g_{\widetilde{X}}}$ are the Levi--Civita connections for $g_X , g_{\widetilde{X}}$ and $p_1(\nabla^{g_X}), p_1( \nabla^{g_{\widetilde{X}}})$ are the corresponding first Pontryagin forms. Since $g_{\widetilde{X}}$ is the pullback of $g_X$, we have $\int_{\widetilde{X}} p_1( \nabla^{g_{\widetilde{X}}} ) = 2 \int_X p_1( \nabla^{g_X})$ and thus
\begin{align*}
\eta_{sig}(Y , g_Y) &= \frac{1}{2}\sigma(\widetilde{X}) - \sigma(X) + \frac{1}{2} \eta_{sig}(\widetilde{Y} , g_{\widetilde{Y}}) \\
&= \frac{2}{3} \frac{ l \pi r^2 \chi(\Sigma) }{Vol(\Sigma)} - \frac{2}{3} \frac{ l^3 \pi^2 r^4}{Vol(\Sigma)^2} + \frac{l}{3} + 4 \sum_{i=1}^{n} s(b_i,a_i).
\end{align*}

\end{proof}

\subsection{Computation of the transgression term}\label{sec:transgr}

Let $\nabla^Y$ denote the adiabatic connection on $Y$ and $\nabla^{LC,g_Y}$ the Levi--Civita connection. Let
\[
Tr(\nabla^Y , \nabla^{LC , g_Y}) = \frac{1}{96\pi^2} tr\left( \omega_Y \wedge F_Y + \frac{1}{2} \omega_Y \wedge d^{\nabla^Y} \omega_Y + \frac{1}{3} \omega_Y \wedge \omega_Y \wedge \omega_Y \right)
\]
be the transgression form from $\nabla^Y$ to $\nabla^{LC, g_Y}$. Here $F_Y$ is the curvature of $\nabla^Y$ and $\omega_Y = \nabla^{LC , g_Y} - \nabla^Y$. In this section we will compute the transgression term $\int_Y Tr(\nabla^Y , \nabla^{LC , g_Y})$. It will be easier to work on the double cover $p \colon \widetilde{Y} \to Y$. If $\nabla, \nabla^{LC}$ denote the adiabatic and Levi--Civita connections on $\widetilde{Y}$, then it follows that
\[
Tr(\nabla , \nabla^{LC}) = p^*( Tr(\nabla^Y , \nabla^{LC,g_Y}))
\]
and therefore
\[
\int_Y Tr(\nabla^Y , \nabla^{LC , g_Y}) = \frac{1}{2}\int_{\widetilde{Y}} Tr(\nabla , \nabla^{LC}).
\]
So we will focus on computing $\int_{\widetilde{Y}} Tr(\nabla , \nabla^{LC})$. We have
\[
Tr(\nabla , \nabla^{LC}) = \frac{1}{96\pi^2} tr\left( \omega \wedge F + \frac{1}{2} \omega \wedge d^{\nabla} \omega + \frac{1}{3} \omega \wedge \omega \wedge \omega \right)
\]
where $F$ is the curvature of $\nabla$ and $\omega = \nabla^{LC} - \nabla$.

Let $e_1,e_2$ be a local orthonormal frame for $g_{\widetilde{\Sigma}}$. Let $\{ \omega^i \}_j$ denote the connection matrix of the Levi--Civita connection for $g_{\widetilde{\Sigma}}$ in the frame $e_1,e_2$. It is the unique matrix of $1$-forms such that ${\omega^i}_j = -{\omega^j}_i$ (metric compatibility) and $de^k = {\omega^k_j} \wedge e^j$ for $k=1,2$ (torsion-freeness). By skew-symmetry, we can write
\[
\omega = \left[ \begin{matrix} 0 & {\omega^1}_2 \\ {\omega^2}_1 & 0 \end{matrix} \right] = \left[ \begin{matrix} 0 & w \\ -w & 0 \end{matrix} \right]
\]
for a $1$-form $w$. The curvature of $\omega$ is then
\[
R = d\omega + \omega \wedge \omega = d\omega = \left[ \begin{matrix} 0 & \nu \\ -\nu & 0 \end{matrix} \right]
\]
where $\nu = dw$. Then $\nu = \kappa e^1 \wedge e^2$, where $\kappa$ is the Gaussian curvature. Since $g_{\widetilde{\Sigma}}$ is assumed to have constant curvature, we will have $\kappa = 2 \pi \chi(\widetilde{\Sigma})/Vol(\widetilde{\Sigma})$.

On $\widetilde{Y}$ we can take a local orthonormal frame of the form $e_0 , e_1 , e_2$, where $e_1, e_2$ are an orthonormal frame for $T\widetilde{\Sigma}$ and $e^0 = r \eta$. We have $de^0 = 2r\xi vol_{\widetilde{\Sigma}}$, where $\xi$ is given by Equation (\ref{equ:xi}).

The adiabatic connection on $T\widetilde{Y} = \mathbb{R} \oplus T\widetilde{\Sigma}$ is given in the local frame $e_0,e_1,e_2$ by
\[
\nabla = d + \left[ \begin{matrix} 0 & 0 & 0 \\ 0 & 0 & w \\ 0 & -w & 0 \end{matrix} \right].
\]
We claim that the Levi--Civita connection $\nabla^{LC}$ for $g_{\widetilde{Y}}$ is given by
\[
\nabla^{LC} = d + A
\]
where
\[
A = \left[ \begin{matrix} 0 & r\xi e^2 & -r\xi e^1 \\ -r\xi e^2 & 0 & -r\xi e^0 + w \\ r\xi e^1 & r\xi e^0 - w & 0 \end{matrix} \right].
\]
Clearly $A^t = -A$, so $d+A$ preserves $g_{\widetilde{Y}}$. It remains to prove torsion-freeness. The torsion-free condition is
\[
d \left[ \begin{matrix} e^0 \\ e^1 \\ e^2 \end{matrix} \right] + A \wedge \left[ \begin{matrix} e^0 \\ e^1 \\ e^2 \end{matrix} \right] = 0.
\]
To verify this, set $\mu = r \xi$. Then we can write $A = A_0 + \mu B$, where
\[
A_0 = \left[ \begin{matrix} 0 & 0 & 0 \\ 0 & 0 & w \\ 0 & -w & 0 \end{matrix} \right], \quad B = \left[ \begin{matrix} 0 & e^2 & -e^1 \\ -e^2 & 0 & -e^0 \\ e^1 & e^0 & 0 \end{matrix} \right].
\]
We have
\[
d \left[ \begin{matrix} e^0 \\ e^1 \\ e^2 \end{matrix} \right] = \left[ \begin{matrix} 2 \mu e^1 \wedge e^2 \\ -w \wedge e^2 \\ w \wedge e^1 \end{matrix} \right],
\]
\[
A_0 \wedge \left[ \begin{matrix} e^0 \\ e^1 \\ e^2 \end{matrix} \right] = \left[ \begin{matrix} 0 \\ w \wedge e^2 \\ -w \wedge e^1 \end{matrix} \right]
\]
and
\[
B \wedge \left[ \begin{matrix} e^0 \\ e^1 \\ e^2 \end{matrix} \right] = \left[ \begin{matrix} -2 e^1 \wedge e^2 \\ 0 \\ 0 \end{matrix} \right].
\]
Torsion-freeness of $d + A$ follows easily.

Observe that $\nabla^{LC} = \nabla + \mu B$. So the transgression term is given by
\[
\frac{1}{192\pi^2} \int_{\widetilde{Y}} tr \left( \mu B \wedge F  + \frac{1}{2} (\mu B) \wedge d^{\nabla} (\mu B) + \frac{1}{3} (\mu B) \wedge (\mu B) \wedge (\mu B) \right)
\]
where $F$ is the curvature of $\nabla$. Thus
\[
F = \left[ \begin{matrix} 0 & 0 & 0 \\ 0 & 0 & \kappa \\ 0 & -\kappa & 0 \end{matrix} \right] e^1 \wedge e^2.
\]
By straightforward calculation, we get
\[
tr( B \wedge F) = 2 \kappa vol_{\widetilde{Y}},
\]
\begin{align*}
d^{\nabla} B &= dB + A_0 \wedge B + B \wedge A_0 \\
&= \left[ \begin{matrix} 0 & w \wedge e^1 & w \wedge e^2 \\ -w \wedge e^1 & 0 & -2\mu e^1 \wedge e^2 \\ -w \wedge e^2 & 2 \mu e^1 \wedge e^2 & 0 \end{matrix} \right] + \left[ \begin{matrix} 0 & 0 & 0 \\ w \wedge e^1 & w \wedge e^0 & 0 \\ w \wedge e^2 & 0 & w \wedge e^0 \end{matrix} \right] + \left[ \begin{matrix} 0 & e^1 \wedge w & e^2 \wedge w \\ 0 & e^0 \wedge w & 0 \\ 0 & 0 & e^0 \wedge w \end{matrix} \right] \\
&= \left[ \begin{matrix} 0 & 0 & 0 \\ 0 & 0 & -2\mu e^1 \wedge e^2 \\ 0 & 2\mu e^1 \wedge e^2 & 0 \end{matrix} \right].
\end{align*}
Thus
\[
tr( B \wedge d^{\nabla} B ) = -4 \mu vol_{\widetilde{Y}}.
\]
Lastly, one finds
\[
tr( B \wedge B \wedge B ) = -6 vol_{\widetilde{Y}}.
\]
Therefore,
\begin{align*}
& \frac{1}{192\pi^2} \int_{\widetilde{Y}} tr \left( \mu B \wedge F  + \frac{1}{2} (\mu B) \wedge d^{\nabla} (\mu B) + \frac{1}{3} (\mu B) \wedge (\mu B) \wedge (\mu B) \right) \\
& = \frac{1}{192\pi^2} \int_{\widetilde{Y}} \left( 2\mu \kappa -4 \mu^3 \right) vol_{\widetilde{Y}} \\
&= \frac{1}{96\pi^2}(2\mu \kappa - 4 \mu^3 ) Vol(Y) \\
&= \frac{(2\pi r)}{96\pi^2}Vol(\Sigma)( 2\mu \kappa - 4 \mu^3) \\
&= -\frac{1}{12} \left( \frac{ \pi l r^2 \chi(\Sigma)}{Vol(\Sigma)} - \frac{ \pi^2 r^4 l^3}{Vol(\Sigma)^2} \right).
\end{align*}

\subsection{Delta invariants of lens spaces}\label{sec:deltalens}

Let $a,b$ be integers with $a > 0$ and $b$ coprime to $a$. The lens space $L(a,b)$ is the quotient $L(a,b) = S^3/\mathbb{Z}_a$ where $\mathbb{Z}_a = \langle \tau \rangle$ acts on the unit sphere $S^3$ in $\mathbb{C}^2$ by $\tau(z_1 , z_2) = ( \omega z_1 , \omega^b z_2)$, where $\omega = e^{2\pi i/a}$. Let $\mathfrak{s}_{can}$ be the restriction to $S^3$ of the canonical spin$^c$-structure on $\mathbb{C}^2$. Since $\mathbb{Z}_a$ acts holomorphically on $\mathbb{C}^2$, the action lifts to $\mathfrak{s}_{can}$ and thus $\mathfrak{s}_{can}$ descends to a spin$^c$-structure on $L(a,b)$. For any $n \in \mathbb{Z}_a$, we obtain a homomorphism $\phi_n \colon \mathbb{Z}_a \to S^1$ by setting $\phi_n(\tau) = \omega^n$. Then $L_{\phi_n} = S^3 \times_{\mathbb{Z}_a} \mathbb{C}$ defines a line bundle on $L(a,b)$, where $\mathbb{Z}_a$ acts on $\mathbb{C}$ by $\phi_n$. Let $\mathfrak{s}_n = L_{\phi_n} \otimes \mathfrak{s}_{can}$. In this way we have a bijection between $\mathbb{Z}_n$ and spin$^c$-structures on $L(a,b)$. As shown in \cite[\textsection 2]{bar}, the delta invariants of $L(a,b)$ are given by
\[
\delta( L(a,b) , \mathfrak{s}_n ) = \lambda(a,b;n) - \frac{1}{2} s(b,a),
\]
where $s(b,a)$ is the Dedekind sum
\begin{align*}
s(b,a) &= \frac{1}{4a} \sum_{j=1}^{a-1} \cot \left( \frac{jb}{a} \right) \cot \left( \frac{j}{a} \right) \\
& = -\frac{1}{4a} \sum_{j=1}^{a-1} \frac{ (1+\omega^{-j})(1+\omega^{-bj})}{(1-\omega^{-j})(1-\omega^{-bj}) } \\
& = - \frac{1}{a} \sum_{j=1}^{a-1} \frac{1}{(1-\omega^{-j})(1-\omega^{-bj}) } + \frac{a-1}{4a}
\end{align*}
and $\lambda(b,a;n)$ is defined by 
\[
\lambda(b,a;n) = \frac{1}{a} \sum_{j=1}^{a-1} \frac{ \omega^{nj} }{(1-\omega^{-j})(1-\omega^{-jb})}.
\]

The Dedekind sum $s(b,a)$ can also be expressed using the sawtooth function $((x))$ which is given by $((x)) = \{x\} -1/2$ for $x \notin \mathbb{Z}$, where $\{x\}$ is the fractional part of $x$ and $((x)) = 0$ for $x \in \mathbb{Z}$. Then
\[
s(b,a) = \sum_{j=1}^{a-1} \left( \! \left( \frac{bj}{a} \right) \! \right) \left( \! \left( \frac{j}{a} \right) \! \right).
\]
For the purpose of computing eta invariants it is useful to consider a more general type of sum sometimes referred to as a {\em Dedekind--Rademacher sum} \cite{rad}
\[
s(b,a ; x , y) = \sum_{j=0}^{a-1} \left( \! \left( x + \frac{b(j+y)}{a} \right) \! \right) \left( \! \left( \frac{j+y}{a} \right) \! \right),
\]
where $a,b$ are integers, $a > 0$, $b$ is coprime to $a$ and $x,y$ are real numbers. Since $s(b,a;x+1,y) = s(b,a;x,y+1) = s(b,a;x,y)$, we will sometimes regard $x,y$ as elements of $\mathbb{R}/\mathbb{Z}$. Unlike ordinary Dedekind sums, $s(b,a ; x , y)$ depends on $b$ as an integer, not just an element of $\mathbb{Z}_a$. However we have
\[
s(b , a ; x , y ) = s(b - ma , a ; x + my , y )
\]
for any integer $m$ \cite[Theorem 3]{rad}. In particular, $s(b,a ; x , 0 )$ depends only on the residue class of $b$ mod $a$.

The function $\lambda(b,a;n)$ is related to the Dedekind--Rademacher sum by
\begin{equation}\label{equ:slambda}
\lambda(b,a;n) = -s(b,a ; n/a , 0) + \frac{a-1}{4a} - \frac{1}{2} \left\{ \frac{n}{a} \right\} - \frac{1}{2} \left( \! \left( \frac{b'n}{a} \right) \! \right)
\end{equation}
where $b'$ is an integer such that $bb' = 1 \; ({\rm mod} \; a)$ \cite[Equation (2.2.11)]{nem}.

\begin{lemma}\label{lem:slambda}
We have
\[
s\left(b,a; \frac{\gamma + b/2}{a} , \frac{1}{2}\right) + \frac{1}{2}\left( \! \left( \frac{b'\gamma + 1/2}{a} \right) \! \right) = -\lambda(b,a;\gamma).
\]
\end{lemma}
\begin{proof}
We have
\[
s\left(b,a ; \frac{\gamma + b/2}{a} , \frac{1}{2} \right) = \sum_{j=0}^{a-1} \left( \! \left( \frac{ \gamma + b/2 + b(j+1/2) }{a} \right) \! \right) \left( \! \left( \frac{j + 1/2}{a} \right) \! \right).
\]
Replacing $j$ by $-b'j - b'\gamma -1$ in this sum, we get
\[
s\left(b,a ; \frac{\gamma + b/2}{a} , \frac{1}{2} \right) = \sum_{j=0}^{a-1} \left( \! \left( \frac{ -j }{a} \right) \! \right) \left( \! \left( \frac{-b'\gamma - 1/2 - b'j}{a} \right) \! \right) = s\left( b' , a ; \frac{b'\gamma + 1/2}{a} , 0 \right).
\]

For any integer $n$, we have
\[
\left( \! \left( \frac{n + 1/2}{a} \right) \! \right) = \left( \! \left( \frac{n}{a} \right) \! \right) + \frac{1}{2a} - \frac{1}{2} \delta(n/a),
\]
where $\delta(x) = 1$ if $x \in \mathbb{Z}$ and $\delta(x) = 0$ otherwise. Substituting the relation into the definition of $s(b',a ; (b'\gamma + 1/2)/a , 0)$ gives
\begin{align*}
s\left(b' , a ; \frac{b'\gamma + 1/2}{a} , 0 \right) &= \sum_{j=0}^{a-1} \left( \! \left( \frac{ b'(j+\gamma) + 1/2 }{a} \right) \! \right) \left( \! \left( \frac{j}{a} \right) \! \right) \\
&= \sum_{j=0}^{a-1} \left( \left( \! \left( \frac{b'j + b'\gamma}{a} \right) \! \right) + \frac{1}{2a} - \frac{1}{2} \delta\left( \frac{b'j + b'\gamma}{a} \right) \right) \left( \! \left( \frac{j}{a} \right) \! \right) \\
&= s(b',a ; b'\gamma/a , 0) + \frac{1}{2} \left( \! \left( \frac{ \gamma }{a} \right) \! \right)
\end{align*}
where we used that $\sum_{j=0}^{a-1} (( j/a )) = 0$. Next, we have 
\[
s(b',a ; b'\gamma/a , 0 ) = \sum_{j=0}^{a-1} \left( \! \left( \frac{ b'\gamma + b'j}{a} \right) \! \right) \left( \! \left( \frac{j}{a} \right) \! \right).
\]
Replace $j$ by $-bj-\gamma$ in the sum to get
\[
s(b',a ; b'\gamma/a , 0) = \sum_{j=0}^{a-1} \left( \! \left( -\frac{j}{a} \right) \! \right) \left( \! \left( \frac{-bj - \gamma}{a} \right) \! \right) = s(b,a ; \gamma/a , 0).
\]
So we have
\begin{align*}
& s\left(b,a ; \frac{ \gamma + b/2 }{a} , \frac{1}{2} \right) + \frac{1}{2}\left( \! \left( \frac{b'\gamma + 1/2}{a} \right) \! \right) \\
& \quad \quad = s(b,a ; \gamma/a , 0) + \frac{1}{2} \left( \! \left( \frac{\gamma}{a} \right) \! \right) + \frac{1}{2}\left( \! \left( \frac{b'\gamma + 1/2}{a} \right) \! \right) \\
& \quad \quad = s(b,a ; \gamma/a , 0) + \frac{1}{2} \left\{ \frac{\gamma}{a} \right\} - \frac{1}{4} + \frac{1}{4}\delta(\gamma/a) + \frac{1}{2}\left( \! \left( \frac{b'\gamma }{a} \right) \! \right) + \frac{1}{4a} - \frac{1}{4}\delta(\gamma/a) \\
& \quad \quad = s(b,a ; \gamma/a , 0) + \frac{1}{2} \left\{ \frac{\gamma}{a} \right\} + \frac{1}{2}\left( \! \left( \frac{b'\gamma }{a} \right) \! \right) - \frac{(a-1)}{4a} \\
& \quad \quad = -\lambda(b,a;\gamma).
\end{align*}

\end{proof}

\subsection{Computation of $\eta_{dir}$}\label{sec:etadir}

Let $E_0$ be an orbifold line bundle on $\widetilde{\Sigma}$, let $m \in \{0,-1 \}$ and let $\tau$ be an equivariant structure on $E_0$ of weight $m$. Consider a reducible $B = A + B_{can}$ for the spin$^c$-structure $\mathfrak{s}_{(E_0 , \tau , m)}$. As shown in Section \ref{sec:spincY} we can take $A$ of the form $A = \pi^*(A_0) + i \lambda \eta$, where $A_0$ is a constant curvature connection on $E_0$ and $\lambda = m/2$. Set $E = \pi^*(E_0)$ and let $S_E = E \oplus (EK_{\widetilde{\Sigma}}^{-1})$ be the assocated spinor bundle. The Dirac operator $D_B \colon \Gamma( \widetilde{Y} , S_E ) \to \Gamma(\widetilde{Y} , S_E )$ associated to $B$ has the form
\[
D_B = Z + T
\]
where
\[
Z = \left[ \begin{matrix} \frac{i}{r}( \nabla_{\partial} + i\lambda) & 0 \\ 0 & -\frac{i}{r}( \nabla_{\partial} + i\lambda) \end{matrix} \right], \quad T = \left[ \begin{matrix} 0 & \sqrt{2} \, \overline{\partial}_{A_0}^* \\ \sqrt{2} \, \overline{\partial}_{A_0} & 0 \end{matrix} \right].
\]
Here $\nabla_{\partial}$ denotes the covariant derivative in the direction of the vertical vector field $\partial$ (taken with respect to the connections on $E_0$ and $E_0 K_{\widetilde{\Sigma}}^{-1}$ determined by $A_0$ and the Levi--Civita connection on $\widetilde{\Sigma}$). The operator $\overline{\partial}_{A_0}$, is defined using the horizontal distribution determined by $\eta$. More precisely, if $e_1,e_2$ is a local oriented orthonormal frame for $T\widetilde{\Sigma}$ and $\widetilde{e}_1, \widetilde{e}_2$ are the horizontal lifts of $e_1,e_2$, then
\[
\overline{\partial}_{A_0} = \pi^*(e^1 - ie^2) \otimes \frac{1}{2}( \nabla^{A_0}_{\widetilde{e}_1} + i\nabla^{A_0}_{\widetilde{e}_2} ).
\]
The operator $\partial_{A_0}$ is defined similarly and $\overline{\partial}_{A_0}^* = - * \overline{\partial}_{A_0} *$, where $*$ is the Hodge star on $\wedge^* T\widetilde{\Sigma}$. We remark that the operators $Z$ and $T$ anti-commute:
\[
ZT + TZ = 0.
\]

Recall that the equivariant structure $\tau$ defines an involution $\sigma_\tau$ as given by (\ref{equ:sigmatau}). If we let $S_{(E_0 , \tau , m)}$ denote the spinor bundle for $\mathfrak{s}_{(E_0 , \tau , m)}$ then $S_{(E_0 , \tau , m)}$ is obtained by descending $S_E$ to $Y$ using $\sigma_\tau$. Thus we have an equality
\begin{equation}\label{equ:invsec}
\Gamma( Y ,  S_{(E_0 , \tau , m)} ) = \Gamma(\widetilde{Y} , S_E)^{\sigma_\tau}.
\end{equation}
The spin$^c$-connection $B$ commutes with $\sigma_\tau$ and hence descends to a spin$^c$-connection $B^+$ on $Y$ for the spin$^c$-structure $\mathfrak{s}_{(E_0 , \tau , m)}$. Under the identification (\ref{equ:invsec}), the Dirac operator $D_{B^+}$ for $B^+$ is just the restriction of the Dirac operator $D_B$ to $\sigma_\tau$-invariant sections. Note also that $T$ and $Z$ commute with $\sigma_\tau$, so they restrict to operators on $\Gamma(\widetilde{Y} , S_E)^{\sigma_\tau}$.

Notice that if we replace $\tau$ by $-\tau$, we get another equivariant structure on $E_0$ and another spin$^c$-structure $\mathfrak{s}_{(E_0 , -\tau , m)}$ on $Y$. Since $\sigma_{-\tau} = -\sigma_\tau$, we have an equality
\[
\Gamma( Y ,  S_{(E_0 , -\tau , m)} ) = \Gamma(\widetilde{Y} , S_E)^{-\sigma_\tau}
\]
Then just as $B$ descends to a spin$^c$-connection $B^+$ for $\mathfrak{s}_{(E_0 , \tau , m)}$, we also have that $B$ descends to a spin$^c$-connection $B^-$ for $\mathfrak{s}_{(E_0 , -\tau , m)}$. The Dirac operator for $B^-$ is given by taking $D_B$ and restricting it to $\Gamma(\widetilde{Y} , S_E)^{-\sigma_\tau}$.

For $\mu \in \mathbb{R}$, let $V_\mu = \{ \psi \in \Gamma(\widetilde{Y} , S_E) \; | \; D_B \psi = \mu \psi \}$ be the $\mu$-eigenspace of $D_B$. Let $V_\mu^{\pm} = \{ \psi \in V_\mu \; | \; \sigma_\tau \psi = \pm \psi\}$. Let $\eta_{D_B}(s)$ denote the eta function of $D_B$
\[
\eta_{D_B}(s) = \sum_{\mu > 0} \frac{ dim(V_\mu) - dim(V_{-\mu}) }{\mu^s}.
\]
Let us also define eta functions $\eta_{D_B}^{\pm}(s)$ corresponding to the $\pm 1$-eigenspaces of $\sigma_\tau$
\[
\eta_{D_B}^\pm(s) = \sum_{\mu > 0} \frac{ dim( V^\pm_{\mu}) - dim( V^\pm_{-\mu}) }{ \mu^s }.
\]
Notice that $\eta_{D_B}^+(s)$ is the eta function for the Dirac operator $D_{B_Y}$, which corresponds to the spin$^c$-structure $\mathfrak{s}_{(E_0 , \tau , m)}$. Similarly $\eta_{D_B}^-(s)$ is the eta function for the Dirac operator corresponding to the spin$^c$-structure $\mathfrak{s}_{(E_0 , -\tau , m)}$. Therefore we are interested in computing $\eta_{D_B}^\pm(0)$.

Following Nicolaescu \cite[Appendix 3]{nic}, we define
\[
E_\mu = \{ \psi \in V_\mu \; | \; ZT \psi = 0 \}.
\]
We also set $E_\mu^{\pm} = E_\mu \cap V^{\pm}_{\mu}$. Then $V_\mu = E_\mu \oplus E_\mu^{\perp}$, where $E_\mu^{\perp}$ is the $L^2$-orthogonal complement of $E_\mu$ in $V_\mu$. Nicolaescu shows that $ZT$ is injective on $E_\mu^\perp$ and $ZT(E_\mu^\perp) = E_{-\mu}^\perp$. Therefore 
\[
dim(V_\mu) - dim(V_{-\mu}) = dim(E_\mu) - dim(E_{-\mu})
\]
and hence
\[
\eta_{D_B}(s) = \sum_{\mu > 0 } \frac{ e_\mu - e_{-\mu} }{\mu^s}
\]
where $e_\mu = dim(E_\mu)$. Since $Z$ and $T$ commute with $\sigma_{\tau}$, we also have that $ZT$ gives a bijection $ZT : (E_\mu^{\pm})^{\perp} \to (E_{-\mu}^{\pm})^\perp$ and therefore we similarly have
\[
\eta_{D_B}^{\pm}(s) = \sum_{\mu > 0 } \frac{ e^{\pm}_\mu - e^{\pm}_{-\mu} }{\mu^s}
\]
where $e^{\pm}_{\mu} = dim( E_\mu^{\pm})$.

Assume $\mu \neq 0$. Since on $E_\mu$ we have $ZT = TZ = 0$ and $Z+T = D_B = \mu$, it follows that 
\[
E_\mu = F_\mu \oplus B_\mu
\]
where 
\[
F_\mu = \{ e \in E_\mu \; | \; Te = 0\}, \quad B_\mu = \{ e \in E_\mu \; | \; Ze = 0\}.
\]
Since $Z+T = \mu$, we have that $Ze = \mu$ for $e \in F_\mu$ and $Te = \mu e$ for $e \in B_\mu$. Since $Z$ and $T$ commute with $\sigma_\tau$ we similarly have $E_\mu^{\pm} = F_\mu^{\pm} \oplus B_\mu^{\pm}$ where $F_\mu^{\pm} = F_\mu \cap V_{\mu}^{\pm}$ and $B_\mu^{\pm} = B_\mu \cap V_\mu^{\pm}$. Let $f_\mu , f_\mu^{\pm}, b_\mu, b_\mu^{\pm}$ be the dimensions of $F_\mu , F_\mu^{\pm} , B_\mu , B_\mu^{\pm}$. So $e_\mu = f_\mu + b_\mu$, $e_\mu^{\pm} = f_\mu^{\pm} + b_{\mu}^{\pm}$, $f_\mu = f_\mu^+ + f_\mu^{-}$ and $b_\mu = b_\mu^+ + b_\mu^{-}$.

We have that $\eta_{D_B}(s) = \eta_f(s) + \eta_b(s)$, where
\[
\eta_f(s) = \sum_{\mu > 0} \frac{f_\mu - f_{-\mu}}{\mu^s}, \quad \eta_b(s) = \sum_{\mu > 0} \frac{b_\mu - b_{-\mu}}{\mu^s}.
\]
Similarly, $\eta_{D_B}^\pm(s) = \eta^{\pm}_f(s) + \eta^{\pm}_b(s)$, where
\[
\eta_f^{\pm}(s) = \sum_{\mu > 0} \frac{f^{\pm}_\mu - f^{\pm}_{-\mu}}{\mu^s}, \quad \eta_b^{\pm}(s) = \sum_{\mu > 0} \frac{b^{\pm}_\mu - b^{\pm}_{-\mu}}{\mu^s}.
\]
A priori, it is not entirely clear that these sums converge in some half plane $Re(s) > c$ nor that they have meromorphic continuations to $\mathbb{C}$. However we will see that this is the case. Moreover the eta functions will all be holomorphic at $s=0$, so that we get equalities $\eta_{D_B}(0) = \eta_f(0) + \eta_b(0)$ and $\eta_{D_B}^\pm(0) = \eta_f^\pm(0) + \eta_b^\pm(0)$.

Consider the eigenspaces $B_\mu$. If $e \in B_\mu$, then $Ze = 0$. Writing $e = (\alpha , \beta)$, we have that $(\nabla_\partial + i\lambda)\alpha = 0$ and $(\nabla_\partial + i\lambda)\beta = 0$. Consider the restrictions of $\alpha,\beta$ to a non-singular fibre of $\widetilde{Y} \to \widetilde{\Sigma}$. These equations imply that parallel translation around the fibre acts as multiplication by $e^{-2\pi i \lambda}$. Recall that $\lambda = m/2$, where $m = 0$ or $-1$. In the case $m=-1$, parallel translation around the fibre acts as $-1$, so $\alpha = -\alpha$ and thus $\alpha = 0$. Similarly $\beta = 0$. Therefore $B_\mu = 0$ and thus $b_\mu = b_\mu^{\pm} = 0$ in the case $m=-1$.

In the case $m = 0$ (so $\lambda = 0$) parallel translation around the fibre has trivial holonomy and the equations $\nabla_\partial \alpha = 0$, $\nabla_\partial \beta = 0$ are equivalent to saying that $\alpha$ and $\beta$ are pullbacks of sections on $\widetilde{\Sigma}$. More precisely, we can say $\alpha = \pi^*(\alpha_0)$, $\beta = \pi^*(\beta_0)$, where
\[
(\alpha_0 , \beta_0) \in \Gamma( \widetilde{\Sigma} , S_{E_0}^{\widetilde{\Sigma}} ),
\]
where $S_{E_0}^{\widetilde{\Sigma}} = E_0 \oplus (E_0 K_{\widetilde{\Sigma}}^{-1})$. Notice that on $B_\mu$ $D_B = Z + T = T$ and furthermore $T$ can be identified with the $2$-dimensional Dirac operator on $S_{E_0}$. Therefore $\eta_b(s)$ is the eta function of the spin$^c$-Dirac operator $Z$ on $S_{E_0}$. Actually, $\eta_b(s)$ is identically zero. This is because the map $\iota : S_{E_0} \to S_{E_0}$ given by $\iota(\alpha , \beta ) = (\alpha , -\beta)$ has the property that $T \iota = -\iota T$. Hence for each $\mu \neq 0$, we get an isomorphism $\iota \colon B_\mu \to B_{-\mu}$ and so $b_\mu = b_{-\mu}$. In contrast, the eta functions $\eta_{b}^{\pm}(s)$ are not necessarily zero. Notice that $\eta_{b}^{+}(s)$ is the eta function of $T$ restricted to $\Gamma( \widetilde{\Sigma} , S_{E_0}^{\widetilde{\Sigma}})^{\sigma_\tau}$. But $\Gamma( \widetilde{\Sigma} , S_{E_0}^{\widetilde{\Sigma}})^{\sigma_\tau}$ can be identified with the pinor bundle on $\Sigma$ corresponding to the pin$^c$-structure $\mathfrak{s}_{(E_0 , \tau)}^{\Sigma}$. Under this identification $T$ is the pin$^c$-Dirac operator for $\mathfrak{s}_{(E_0 , \tau)}^{\Sigma}$. Thus $\eta_b^+(s)$ is the eta function for the pin$^c$-Dirac operator corresponding to the pin$^c$-structure $\mathfrak{s}_{(E_0 , \tau)}^{\Sigma}$. We denote the eta invariant $\eta_b^+(0)$ by $\eta^{pin^c}(\Sigma , \mathfrak{s}_{(E_0 , \tau)}^\Sigma)$. Similarly, $\eta_b^-(s)$ is the eta function for the pin$^c$-Dirac operator corresponding to the pin$^c$-structure $\mathfrak{s}_{(E_0 , -\tau)}$. So $\eta_b^-(0) = \eta^{pin^c}(\Sigma , \mathfrak{s}_{(E_0 , -\tau)}^\Sigma)$. In all cases, we showed that $\eta_b(s) = 0$ and therefore $\eta_b^-(s) = -\eta_b^+(s)$.

Let $\iota \colon S_E \to S_E$ be defined by $\iota(\alpha , \beta) = (\alpha , -\beta)$. Then $Z \iota = \iota Z$ and $T \iota = -\iota T$. Therefore $\iota$ defines an involution $\iota \colon F_\mu \to F_{\mu}$. Furthermore $\iota \circ \sigma_\tau = -\sigma_\tau \circ \iota$, hence $\iota$ exchanges $F_\mu^+$ and $F_\mu^-$. Therefore, $f_\mu^+ = f_\mu^- = (1/2) f_\mu$ and consequently
\[
\eta_f^+(s) = \eta_f^-(s) = \frac{1}{2}\eta_f(s).
\]

We summarise our findings as follows:

\begin{proposition}\label{prop:etapm}
If $m = 0$, then
\[
\eta_{D_B}^{\pm}(0) = \frac{1}{2}\eta_f(0) \pm \eta^{pin^c}( \Sigma , \mathfrak{s}_{(E_0 , \tau)}^\Sigma).
\]
If $m=-1$, then
\[
\eta_{D_B}^{\pm}(0) = \frac{1}{2}\eta_f(0).
\]
\end{proposition}

Now we turn to the problem of computing $\eta_f(0)$. This was carried out by Nicolaescu in \cite{nic} and \cite{nic2}. Since we use slightly different conventions to Nicolaescu for Seifert manifolds and for Clifford multiplication, we will carefully review the computation.

First recall that $\eta_f(s)$ is given by
\[
\eta_f(s) = \sum_{\mu \neq 0} \frac{ sgn(\mu) f_\mu }{ |\mu|^s}
\]
where $f_\mu = dim(F_\mu)$, $F_\mu = \{ e \in \Gamma(\widetilde{Y} , S_{E} ) \; | \; Te = 0, Ze = \mu e\}$. If we write $e = (\alpha , \beta)$ where $\alpha$ is a section of $E$ and $\beta$ is a section of $E K_{\widetilde{\Sigma}}^{-1}$, then the equations become
\begin{align*}
\overline{\partial}_{A_0} \alpha &= 0, \\
\nabla_{\partial} \alpha + i(\lambda + \mu r)   \alpha &= 0, \\
\overline{\partial}^*_{A_0} \beta & = 0, \\
\nabla_{\partial} \beta + i(\lambda - \mu r ) \beta & = 0. 
\end{align*}

Notice that the equations for $\alpha$ and $\beta$ are decoupled from each other and so we can write $F_\mu = F'_\mu \oplus F''_\mu$, where
\begin{align*}
F'_\mu &= \{ \alpha \in \Gamma(\widetilde{Y} , E) \; | \; \overline{\partial}_{A_0} \alpha = 0, \nabla_{\partial} \alpha + i(\lambda + \mu r)   \alpha = 0 \}, \\
F''_\mu &= \{ \beta \in \Gamma(\widetilde{Y} , EK_{\widetilde{\Sigma}}^{-1} ) \; | \; \overline{\partial}^*_{A_0} \beta = 0, \nabla_{\partial} \beta + i(\lambda + \mu r) \beta = 0 \}.
\end{align*}

Consider first $F'_\mu$. The equation $\nabla_{\partial} \alpha + i(\lambda + \mu r)   \alpha = 0$ implies that parallel translation of $\alpha$ around a non-singular fibre acts as multiplication by $e^{-2\pi i (\lambda + \mu r)}$. Hence there are no non-zero solutions unless $\lambda + \mu r \in \mathbb{Z}$.

If $\lambda + \mu r = k \in \mathbb{Z}$, then the equation $\nabla_{\partial} \alpha + i(\lambda + \mu r)   \alpha = 0$ becomes $\nabla^{A_0}_{\partial} \alpha + ik \alpha = 0$. Here we use the superscript $A_0$ to remind us that $\partial^{A_0}_\partial$ denotes the covariant derivative in the direction $\partial$ with respect to the connection $\pi^*(A_0)$.

Let $s$ denote the tautological section of $\pi^*(\widetilde{N})$. Then $\nabla^{A_{\widetilde{N}}}_{\partial} s = is$. Therefore the equation $\nabla^{A_0}_{\partial} \alpha + ik \alpha = 0$ can be re-written as $\nabla_{\partial}^{A_k}( \alpha s^{k} ) = 0$, where $A_k$ is the connection on $E_0 \otimes \widetilde{N}^{k}$ induced by $A_0$ and $\widetilde{N}$. Set $\alpha_k = \alpha s^k$, so that $\alpha = \alpha_k s^{-k}$. The equation $\nabla_{\partial}^{A_k}( \alpha_k ) = 0$ says that $\alpha_k$ is the pullback to $\widetilde{Y}$ of a section of $E_k = E_0 \otimes \widetilde{N}^k$ on $\widetilde{\Sigma}$. Futhermore, since $s$ is holomorphic, the equation $\overline{\partial}_{A_0} \alpha = 0$ is equivalent to $\overline{\partial}_{A_k}( \alpha_k) = 0$. We have shown that
\[
F'_\mu \cong H^0( \widetilde{\Sigma} , E_0 \widetilde{N}^k), \text{ if } \mu = \frac{k - \lambda}{r}.
\]
Here we are using the connections $A_0, A_{\widetilde{N}}$ to define holomorphic structures on $E_0, \widetilde{N}$. Since the underlying space of $\widetilde{\Sigma}$ is $S^2$, it follows that up to isomorphism there is only one holomorphic structure on any given orbifold line bundle. Hence we may safely write $H^0( \widetilde{\Sigma} , E_0 \widetilde{N}^k)$ without any mention of $A_0$ or $A_{\widetilde{N}}$. For any orbifold line bundle $U$ on $\widetilde{\Sigma}$, we write $h^i(U) = dim( H^i( \widetilde{\Sigma} , U ))$.

Set $f'_\mu = dim(F'_\mu)$, $f''_\mu = dim(F''_\mu)$. We have just shown that
\[
f'_\mu = \begin{cases} h^0(E_0 \widetilde{N}^k) & \text{if } \mu = (k-\lambda)/r \text{ for some } k \in \mathbb{Z}, \\ 0 & \text{otherwise.} \end{cases}
\]

In a completely analogous manner one finds that
\[
f''_\mu = \begin{cases} h^1(E_0 \widetilde{N}^k) & \text{if } \mu = (\lambda-k)/r \text{ for some } k \in \mathbb{Z}, \\ 0 & \text{otherwise.} \end{cases}
\]

To proceed further we consider separately the cases of trivial and non-trivial fibre holonomy (i.e. $m=0$ and $m = -1$).

In the case of trivial fibre holonomy, $m=0$, $\lambda = 0$ and so
\[
\eta_f(s) = r^s \sum_{k=1}^{\infty} \frac{h^0(E_k) - h^0(E_{-k}) + h^1(E_{-k}) - h^1(E_k) }{ k^s },
\]
where $E_k = E_0 \otimes \widetilde{N}^k$. Riemann--Roch gives $h^0(E_k) - h^1(E_k) = deg| E_k | + 1$, where $|E_k|$ denotes the desingularisation of $E_k$ \cite[\textsection 2]{moy}.

Following Nicolaescu \cite{nic2}, we introduce a function $G(k) = deg|E_{-k}| - deg(E_{-k})$. If $[E_0] = (e ; \gamma_1 , \delta_1 , \dots , \gamma_n , \delta_n)$, then $[E_{-k}] = (e - k\widetilde{l} ; \gamma_1 + kb_1 , \delta_1 + kb_1 , \dots , \gamma_n + kb_n , \delta_n + kb_n)$. In general, if $[U] = ( u ; u_1 , v_1 , \dots , u_n , v_n)$, then $deg|U| = u - \sum_{i=1}^{n} ( \{u_i/a_i\} + \{v_i/a_i\})$ and $deg(U) = u$. Hence
\[
G(k) = -\sum_{i=1}^{n} \left( \left\{ \frac{ \gamma_i + kb_i}{a_i} \right\} + \left\{ \frac{\delta_i + kb_i}{a_i} \right\} \right).
\]

From the definition of $G$, we have
\begin{align*}
\eta_f(s) &= r^s \sum_{n=1}^{\infty} \frac{ (deg|E_k| + 1) - (deg|E_{-k}| + 1) }{ k^s} \\
&= r^s \sum_{n=1}^{\infty} \frac{ (G(-k) + deg(E_k) ) - (G(k) + deg(E_{-k}) ) }{k^s} \\
&= r^s \sum_{n=1}^{\infty} \frac{ G(-k) - G(k) + (e + k\widetilde{l}) - (e - k\widetilde{l}) }{k^s} \\
&= r^s \sum_{n=1}^{\infty} \frac{G(k)-G(-k)}{k^s} + 2 \widetilde{l} r^s \sum_{n=1}^{\infty} \frac{1}{k^{s-1}} \\
&= r^s \eta_G(s) + 2 \widetilde{l} r^s \zeta(s-1)
\end{align*}
where
\[
\eta_G(s) = \sum_{k=1}^{\infty} \frac{ G(-k) - G(k) }{k^s}
\]
and $\zeta(s)$ is the Riemann zeta function. Since $\zeta(-1) = -1/12$, this gives
\begin{equation}\label{equ:etafG}
\eta_f(0) = -\frac{\widetilde{l}}{6} + \eta_G(0).
\end{equation}

Write $G(k) = \sum_{i=1}^n G_i(k)$, where
\[
G_i(k) = -\left\{ \frac{ \gamma_i + kb_i}{a_i} \right\} - \left\{ \frac{\delta_i + kb_i}{a_i} \right\}.
\]
Then $G_i(k)$ is periodic with period $a_i$. We have $\eta_G(s) = \sum_{i=1}^n \eta_{G_i}(s)$ where $\eta_{G_i}(s) = \sum_{k = 1}^{\infty} ( G_i(-k) - G_i(k) )/k^s$. By \cite[Lemma 1.5]{nic2}, we have
\[
\sum_{k=1}^{\infty} \frac{G_i(-k) - G_i(k)}{k^s} = \sum_{r = 1}^{a_i} \frac{ G_i(-r) - G_i(r) }{r^s } \zeta(s,r/a_i)
\]
where
\[
\zeta(s,u) = \sum_{n = 0}^{\infty} \frac{1}{(n+u)^s}, \; u \notin \{0,-1,-2, \dots \}
\]
is the Riemann--Hurwitz zeta function. Since $\zeta(0,u) = (1/2) - u$ we get
\[
\eta_{G_i}(0) = \sum_{r=1}^{a_i} ( G_i(-r) - G_i(r) ) \left( \frac{1}{2} - \frac{r}{a_i} \right) = \sum_{r=1}^{a_i} ( G_i(r) - G_i(-r) ) \left( \! \left( \frac{r}{a_i} \right) \! \right),
\]
where the second equality follows because $G_i(a_i) - G_i(-a_i) = 0$. From the definition of $G_i$, we get
\[
\eta_{G_i}(0) = \sum_{r=1}^{a_i} \left( \left\{ \frac{\gamma_i - rb_i}{a_i} \right\} + \left\{ \frac{\delta_i - rb_i}{a_i} \right\} - \left\{ \frac{\gamma_i + rb_i}{a_i} \right\} - \left\{ \frac{\delta_i + rb_i}{a_i} \right\} \right) \left( \! \left( \frac{r}{a_i} \right) \! \right).
\]

Then, using $\{x\} = ((x)) + (1/2) - (1/2)\delta(x)$, where $\delta(x) = 1$ if $x \in \mathbb{Z}$ and $\delta(x) = 0$ if $x \notin \mathbb{Z}$, we get (and noting that $((-x)) = -((x))$)
\begin{align*}
\eta_{
G_i}(0) &= \sum_{r=1}^{a_i} \left( \left( \! \left( \frac{ \gamma_i - rb_i}{a_i} \right) \! \right) + \left( \! \left( \frac{\delta_i - rb_i}{a_i} \right) \! \right) - \left( \! \left( \frac{ \gamma_i + rb_i}{a_i} \right) \! \right) - \left( \! \left( \frac{\delta_i + rb_i}{a_i} \right) \! \right) \right) \left( \! \left( \frac{r}{a_i} \right) \! \right) \\
& \quad \quad -\frac{1}{2} \sum \left( \delta\left( \frac{\gamma_i - rb_i}{a_i} \right) + \delta \left( \frac{\delta_i - rb_i}{a_i} \right) - \delta \left( \frac{\gamma_i + rb_i}{a_i }\right) - \delta \left( \frac{\delta_i +rb_i}{a_i} \right) \right) \left( \! \left( \frac{r}{a_i} \right) \! \right) \\
& = -2\sum_{r=1}^{a_i} \left( \left( \! \left( \frac{ \gamma_i + rb_i}{a_i} \right) \! \right) + \left( \! \left( \frac{\delta_i + rb_i}{a_i} \right) \! \right) \right) \left( \! \left( \frac{r}{a_i} \right) \! \right) \\
& \quad \quad -\left( \! \left( \frac{ b'_i \gamma_i}{a_i} \right) \! \right) -  \left( \! \left( \frac{b'_i \delta}{a_i} \right) \! \right) \\
& = -2 s(b_i , a_i ; \gamma_i/a_i , 0) - 2s(b_i , a_i ; \delta_i/a_i , 0)  -\left( \! \left( \frac{ b'_i \gamma_i}{a_i} \right) \! \right) -  \left( \! \left( \frac{b'_i \delta}{a_i} \right) \! \right).
\end{align*}
Here $b'_i$ denotes the inverse of $b_i$ mod $a_i$.

\begin{proposition}
In the case of trivial fibre holonomy $m = 0$, if $[E_0] = (e ; \gamma_1 , \delta_1 , \dots , \gamma_n , \delta_n)$, then
\[
\eta_f(0) = -\frac{l}{3} + 2\sum_{i=1}^{n} \left( s(b_i,a_i;\gamma_i/a_i,0) + s(b_i,a_i;\delta_i/a_i , 0)) \right).
\]
\end{proposition}
\begin{proof}
Using the computation of $\eta_{G_i}(0)$ given above and Equation (\ref{equ:etafG}), we get:
\begin{equation}\label{equ:etaf1}
\eta_f(0) = -\frac{l}{3} - 2\sum_{i=1}^{n} \left( s(b_i,a_i;\gamma_i/a_i,0) + s(b_i,a_i;\delta_i/a_i , 0) + \frac{1}{2} \left( \! \left( \frac{b'_i \gamma_i}{a_i} \right) \! \right) + \frac{1}{2} \left(\! \left( \frac{b'_i \delta_i}{a_i} \right) \! \right) \right).
\end{equation}

Recall that since $m=0$, we have $\gamma_i + \delta_i = -1 \; ({\rm mod} \; a_i)$. In general, if $a > 0$, $b$ is coprime to $a$ and $\gamma,\delta$ are such that $\gamma + \delta = -1 \; ({\rm mod} \; a)$, then using Equation (\ref{equ:slambda}), we get that 
\begin{align*}
&s(b,a;\gamma/a,0) + s(b,a;\delta/a,0) + \frac{1}{2} \left(\! \left( \frac{b'\gamma}{a} \right) \! \right) + \frac{1}{2} \left( \! \left( \frac{b'\delta}{a} \right) \! \right) \\
& \quad = -\lambda(b,a;\gamma) - \lambda(b,a;\delta) - \frac{1}{2} \left\{ \frac{\gamma}{a} \right\} - \frac{1}{2} \left\{ \frac{\delta}{a} \right\} + \frac{a-1}{2a} \\
& \quad = -\lambda(b;a;\gamma) - \lambda(b,a;\delta) - \frac{(a-1)}{2a} + \frac{a-1}{2a} \\
& \quad = -\lambda(b;a;\gamma) - \lambda(b,a;\delta)
\end{align*}
where we used that $\gamma + \delta = -1 \; ({\rm mod} \; a)$ to deduce that $\{ \gamma/a \} + \{\delta/a\} = (a-1)/a$. Applying this to Equation (\ref{equ:etaf1}), we get that
\[
\eta_f(0) = -\frac{l}{3} + 2\sum_{i=1}^{n} \left( s(b_i,a_i;\gamma_i/a_i,0) + s(b_i,a_i;\delta_i/a_i , 0)) \right).
\]
\end{proof}

Now we assume that the fibre holonomy is non-trivial. We take $m = -1$, $\lambda = -1/2$. From our earlier calculations of $f'_\mu, f''_\mu$, we have
\[
\eta_f(s) = r^s \sum_{k \in \mathbb{Z}} \frac{ sgn(k+1/2) ( h^0(E_k) - h^1(E_k) ) }{ |k - \lambda|^s }.
\]
By Riemann--Roch, $h^0( E_k) - h^1(E_k ) = 1 + deg| E_k |$, so
\begin{align*}
\eta_f(s) &= r^s \sum_{k \in \mathbb{Z}} \frac{ sgn(k+1/2)( 1 + deg| E_k| )}{ |k+1/2|^s } \\
&= r^s \sum_{k = 0}^{\infty} \frac{1 + deg|E_k |}{(k+1/2)^s} - r^s \sum_{k=-1}^{\infty} \frac{1 + deg|E_k|}{|k+1/2|^s} \\
&= r^s \sum_{k = 0}^{\infty} \frac{deg|E_k|}{(k+1/2)^s} - r^s \sum_{k=-1}^{\infty} \frac{deg|E_k|}{|k+1/2|^s} \\
&= r^s \sum_{k \in \mathbb{Z}} \frac{ sgn(k +1/2)( G(-k) + e + k \widetilde{l} ) }{ |k + 1/2|^s } \\
&= r^s \sum_{k \in \mathbb{Z}} \frac{ sgn(k +1/2)( G(-k) +  k\widetilde{l} ) }{ |k + 1/2|^s }.
\end{align*}

Next we have
\begin{align*}
& \sum_{k \in \mathbb{Z}} \frac{ sgn(k+1/2) k }{ |k+1/2|^s } \\
& \quad = 2\sum_{k=0}^{\infty} \frac{k}{(k+1/2)^s} \\
& \quad = 2 \sum_{k=0}^{\infty} \frac{ (k+1/2) }{(k+1/2)^s} - \sum_{k=0}^{\infty} \frac{ 1}{(k+1/2)^s} \\
& \quad = 2\zeta(s-1 , 1/2) - \zeta(s,1/2).
\end{align*}
When $s=0$, this equals
\[
2\zeta(-1,1/2) - \zeta(0,1/2) = 1/12 - 0 = 1/12.
\]
Here we used that $\zeta(-1,u) = -B_2(u)/2$, where $B_2(x) = x^2 - x + 1/6$ is the second Bernoulli polynomial and $\zeta(0,u) = 1/2 - u$. Therefore,
\[
\eta_f(0) = \frac{\widetilde{l}}{12} + \eta_G(0)
\]
where
\[
\eta_G(s) = \sum_{k \in \mathbb{Z}} \frac{ sgn(k+1/2)G(-k) }{ |k+1/2|^s}.
\]

Let $S = (1/2) + \mathbb{Z}$. Every $\mu \in S$ can be written as $\mu = k + 1/2$ for a unique integer $k$. It follows that
\[
\eta_G(s) = \sum_{\mu \in S} \frac{ sgn(\mu)G( -(\mu - 1/2) ) }{ |\mu|^s}.
\]
As before, write $G = \sum_{i=1}^n G_i$ and $\eta_G(s) = \sum_{i=1}^n \eta_{G_i}(s)$.

By \cite[Lemma 1.11]{nic2}, if $f : \mathbb{R} \to \mathbb{R}$ has period $a \in \mathbb{Z}_+$ and we set
\[
\eta(f,s) = \sum_{\mu \in S} \frac{ sgn(\mu) f(\mu - 1/2) }{|\mu|^s},
\]
then
\[
\eta(f,s) = \sum_{r=0}^{a-1} \frac{f(r)}{p^s} \left( \zeta( s , \left\{ \frac{r+1/2}{a} \right\} ) - \zeta( s , 1 - \left\{ \frac{r+1/2}{a} \right\} ) \right).
\]
In particular,
\[
\eta(f,0) = -2\sum_{r=0}^{a-1} f(r) \left( \! \left( \frac{ r + 1/2 }{a} \right) \! \right).
\]
Taking $f(r) = G_i(-r)$ in the above and then summing over $i$ gives
\begin{align*}
\eta_f(0) &= \frac{\widetilde{l}}{12} -2 \sum_{i=1}^n \sum_{r=0}^{a_i - 1} G_i(-r) \left( \! \left( \frac{ r + 1/2 }{a_i} \right) \! \right) \\
&= \frac{\widetilde{l}}{12} + 2 \sum_{i=1}^n \sum_{r=0}^{a_i-1} \left( \left\{ \frac{ \gamma_i - rb_i}{a_i} \right\} + \left\{ \frac{\delta_i - rb_i}{a_i} \right\} \right)\left( \! \left( \frac{ r + 1/2 }{a_i } \right) \! \right).
\end{align*}

To simplify this, observe that
\begin{align*}
& \sum_{r=0}^{a-1} \left\{ \frac{n+rb}{a} \right\} \left( \! \left( \frac{ r + 1/2 }{a} \right) \! \right) \\
& \quad \quad = \sum_{r=0}^{a-1} \left( \left( \! \left( \frac{n+rb}{a} \right) \! \right) + \frac{1}{2} - \frac{1}{2} \delta( (n+rb)/a ) \right) \left( \! \left( \frac{ r + 1/2 }{a } \right) \! \right) \\
&\quad \quad  = \sum_{r=0}^{a-1} \left( \! \left( \frac{n+rb}{a} \right) \! \right) \left( \! \left( \frac{ r + 1/2}{a} \right) \! \right) + \frac{1}{2} \sum_{r=0}^{a-1} \left( \! \left( \frac{ r + 1/2 }{a } \right) \! \right) - \frac{1}{2} \left( \! \left( \frac{ -nb' + 1/2 }{a } \right) \! \right) \\
& \quad \quad = \sum_{r=0}^{a-1} \left( \! \left( \frac{n+rb}{a} \right) \! \right) \left( \! \left( \frac{ r + 1/2}{a} \right) \! \right)  - \frac{1}{2} \left( \! \left( \frac{ -nb' + 1/2 }{a } \right) \! \right) \\
& \quad \quad = s(b,a ; (n-b/2)/a , 1/2) - \frac{1}{2} \left( \! \left( \frac{ -nb' + 1/2 }{a } \right) \! \right)
\end{align*}
where we used that $\sum_{r=0}^{a-1} (\! ( (r+x)/a )\! ) = ((x))$. Therefore,

\begin{align*}
\eta_f(0) & = \frac{\widetilde{l}}{12} + 2 \sum_{i=1}^n \left( s(-b_i,a_i; (\gamma_i + b_i/2)/a_i , 1/2) + s(-b_i , a_i ; (\delta_i + b_i/2)/a_i , 1/2) \right) \\
& \quad \quad - \sum_{i=1}^n \left( \left( \! \left( \frac{ b'_i \gamma_i + 1/2}{a_i} \right)\! \right) + \left( \! \left( \frac{b'_i \delta_i + 1/2}{a_i} \right) \! \right) \right)
\end{align*}

Noting that $s(-b,a ; x , 1/2) = -s(b,a ; x , 1/2)$ and using Lemma \ref{lem:slambda}, this simplifies to
\[
\eta_f(0) = \frac{\widetilde{l}}{12} + 2 \sum_{i=1}^{n} \left( \lambda(b_i,a_i;\gamma_i) + \lambda(b_i,a_i;\delta_i) \right).
\]

We then have:
\begin{proposition}
In the case of non-trivial fibre holonomy $(m=-1)$, if $[E_0] = (e ; \gamma_1 , \delta_1 , \dots , \gamma_n , \delta_n)$, then
\[
\eta_f(0) = \frac{l}{6} + 4 \sum_{i=1}^n \lambda(b_i , a_i ; \gamma_i).
\]
\end{proposition}
\begin{proof}
We have already shown that
\[
\eta_f(0) = \frac{l}{6} + 2 \sum_{i=1}^{n} \left( \lambda(b_i,a_i;\gamma_i) + \lambda(b_i,a_i;\delta_i) \right).
\]
Now since $m=-1$, we have $\gamma_i + \delta_i = -1 - b_i \; ({\rm mod} \; a_i)$. In general, if $a > 0$, $b$ is coprime to $a$ we have that
\[
\lambda(b,a;-1-b-\gamma) = \frac{1}{a} \sum_{r=1}^{a-1} \frac{ \omega^{-(1+b+\gamma)r} }{(1-\omega^{-r})(1-\omega^{-br})}
\]
where $\omega = e^{2\pi i/a}$. Multiply the numerator and denominator by $\omega^{r+br}$ to get
\[
\lambda(b,a;-1-b-\gamma) = \frac{1}{a} \sum_{r=1}^{a-1} \frac{ \omega^{-\gamma r} }{(1-\omega^r)(1-\omega^{br})} = \lambda(b,a;\gamma).
\]
Applying this to $a_i,b_i,\gamma_i,\delta_i$, we have that $\delta_i = -1-b_i - \gamma_i \; ({\rm mod} \; a_i)$ and thus $\lambda(b_i,a_i;\delta_i) = \lambda(b_i,a_i;\gamma_i)$. Hence
\[
\eta_f(0) = \frac{l}{6} + 4 \sum_{i=1}^n \lambda(b_i , a_i ; \gamma_i).
\]
\end{proof}

Combining our calculations of $\eta_f(0)$ with Proposition \ref{prop:etapm}, we get:
\begin{corollary}
Let $[E_0] = (e ; \gamma_1 , \delta_1 , \dots , \gamma_n , \delta_n)$.
\begin{itemize}
\item[(1)]{If the fibre holonomy is trivial ($m=0$), then
\[
\eta_{D_B}(0)^{\pm} = -\frac{l}{6} \pm \eta^{pin^c}(\Sigma , \mathfrak{s}_{(E_0 , \tau)}^{\Sigma}) + \sum_{i=1}^n ( \lambda(b_i,a_i ; \gamma_i) + \lambda(b_i,a_i ; \delta_i) ).
\]
}
\item[(2)]{If the fibre holonomy is non-trivial ($m=-1$), then
\[
\eta_{D_B}(0)^{\pm} = \frac{l}{12} + 2 \sum_{i=1}^n \lambda(b_i , a_i ; \gamma_i).
\]
}
\end{itemize}
\end{corollary}

\subsection{The delta invariants of $Y$}\label{sec:dY2}

Let $E_0$ be an orbifold line bundle on $\widetilde{\Sigma}$, let $m \in \{0,-1\}$ and let $\tau$ be an equivariant structure on $E_0$ of weight $m$. So $(E_0 , \tau , m)$ defines a spin$^c$-structure $\mathfrak{s} = \mathfrak{s}_{(E_0 , \tau , m)}$. Let $B$ be the unique reducible correponding to $(E_0 , \tau , m)$ and $D_B$ the corresponding Dirac operator on $S_E$, $E = \pi^*(E_0)$. As shown in Proposition \ref{prop:swsoln}, $Ker(D_B) = 0$. Hence we have
\[
\delta(Y,\mathfrak{s}) = -n(Y,\mathfrak{s} , g_Y , B ) = -\frac{1}{2} \eta_{D_B}^+(0) + \frac{1}{8}\eta_{sig}(Y , g_Y) + \int_Y Tr( \nabla^Y , \nabla^{LC,g_Y}).
\]

From Proposition \ref{prop:etasig}, we have
\[
\eta_{sig}(Y , g_Y ) = \frac{2}{3} \frac{ l \pi r^2 \chi(\Sigma) }{Vol(\Sigma)} - \frac{2}{3} \frac{ l^3 \pi^2 r^4}{Vol(\Sigma)^2} + \frac{l}{3} + 4 \sum_{i=1}^{n} s(b_i,a_i).
\]

From Section \ref{sec:transgr}, we have
\[
\int_Y Tr(\nabla^Y , \nabla^{LC,g_Y}) = -\frac{1}{12} \left( \frac{ \pi l r^2 \chi(\Sigma)}{Vol(\Sigma)} - \frac{ \pi^2 r^4 l^3}{Vol(\Sigma)^2} \right).
\]
Hence

\[
\delta(Y, \mathfrak{s}) = -\frac{1}{2}\eta_{D_B}^+(0) + \frac{l}{24} + \frac{1}{2} \sum_{i=1}^n s(b_i,a_i).
\]

Combined with the calculation of $\eta_{D_B}^+(0)$ in Section \ref{sec:etadir}, we obtain:

\begin{theorem}\label{thm:deltaY}
Let $[E_0] = (e ; \gamma_1 , \delta_1 , \dots , \gamma_n , \delta_n)$.
\begin{itemize}
\item[(1)]{If the fibre holonomy is trivial ($m=0$), then
\[
\delta(Y,\mathfrak{s}_{(E_0 , \tau , 0)}) =  \frac{l}{8} - \frac{1}{2} \eta^{pin^c}(\Sigma , \mathfrak{s}_{(E_0 , \tau)}^{\Sigma}) +\frac{1}{2} \sum_{i=1}^n \left( -\lambda(b_i,a_i ; \gamma_i) - \lambda(b_i,a_i ; \delta_i)  + s(b_i,a_i) \right).
\]
}
\item[(2)]{If the fibre holonomy is non-trivial ($m=-1$), then
\[
\delta(Y , \mathfrak{s}_{(E_0 , \tau , -1)}) =  \sum_{i=1}^n \left( -\lambda(b_i , a_i ; \gamma_i) + \frac{1}{2}s(b_i,a_i) \right).
\]
}
\end{itemize}

\end{theorem}

\begin{remark}
Since $\delta( L(a,b) , \mathfrak{s}_{\gamma} ) = \lambda(b,a;\gamma) - (1/2)s(b,a)$, we can re-write Theorem \ref{thm:deltaY} in terms of delta invariants of lens spaces. Namely:
\begin{itemize}
\item[(1)]{If the fibre holonomy is trivial, then
\[
\delta(Y,\mathfrak{s}_{(E_0 , \tau , 0)}) =  \frac{l}{8} - \frac{1}{2} \eta^{pin^c}(\Sigma , \mathfrak{s}_{(E_0 , \tau)}^{\Sigma}) -\frac{1}{2} \sum_{i=1}^n \left( \delta(L(a_i,b_i) , \mathfrak{s}_{\gamma_i}) + \delta( L(a_i,b_i) , \mathfrak{s}_{\delta_i}) \right).
\]
}
\item[(2)]{If the fibre holonomy is non-trivial, then
\[
\delta(Y , \mathfrak{s}_{(E_0 , \tau , -1)}) = -\sum_{i=1}^{n} \delta( L(a_i , b_i ) , \mathfrak{s}_{\gamma_i}).
\]
}
\end{itemize}

\end{remark}

\section{Calculating the eta invariant $\eta^{pin^c}$}\label{sec:epc}

In this section we prove that the pin$^c$-eta invariant $\eta^{pin^c}(\Sigma , \mathfrak{s}^{\Sigma})$ always equals $\pm 1/2$. In general, working out the correct sign is a delicate problem since pin$^c$-structures on $\Sigma$ come in pairs $\mathfrak{s}_{(E_0 , \tau)}^\Sigma, \mathfrak{s}_{(E_0 , -\tau)}^\Sigma$ and the eta invariant is sensitive to this choice since $\eta^{pin^c}(\Sigma , \mathfrak{s}_{(E_0 , -\tau)}^\Sigma) = -\eta^{pin^c}(\Sigma , \mathfrak{s}_{(E_0 , \tau)}^\Sigma)$. However this guarantees that $\eta^{pin^c} = 1/2$ for half of the pin$^c$-structures and $\eta^{pin^c} = -1/2$ for the other half. We determine precisely the signs in the spherical cases $\Sigma = \mathbb{RP}^2$, $\Sigma = \mathbb{RP}^2(a)$. We also give an independent proof that $\eta^{pin^c} = \pm 1/2$ in the Euclidean case $\mathbb{RP}^2(2,2)$ and in the process of doing so, obtain a calculation of the delta invariants of the Hantzsche--Wendt manifold.

\subsection{Proof that $\eta^{pin^c}(\Sigma , \mathfrak{s}^\Sigma) = \pm 1/2$}

Because of Theorem \ref{thm:deltaY}, in order to compute $\eta^{pin^c}(\Sigma , \mathfrak{s}^\Sigma)$, it is sufficient to calculate the delta invariants for a single choice of Seifert $3$-manifold $Y \to \Sigma$. Given $\Sigma = \mathbb{RP}^2(a_1, \dots , a_n)$, a convenient choice of $Y$ to take is $Y = S( b ; (a_1 , -1) , (a_2, -1) , \dots , (a_n,-1) )$, where $b$ is any non-zero integer. 

Let $x_1, \dots , x_n$ denote the orbifold points of $\Sigma$, where $x_i$ has order $a_i$. For each $i$, let $x'_i , x''_i \in \widetilde{\Sigma}$ denote the preimages of $x_i$ in $\widetilde{\Sigma}$. Let $\widetilde{W} = D(\widetilde{N})$ be the closed unit disc bundle of $\widetilde{N}$. This is a $4$-dimensional orbifold with boundary $\widetilde{Y}$ and $n$ isolated singularities located at the points of the zero section of $\widetilde{N}$ lying over the orbifold points $x'_1, x''_1 \dots , x'_n, x''_n$. Let $x$ be a non-orbifold point of $\Sigma$ and let $x',x''$ denote the preimages of $x$ in $\widetilde{\Sigma}$. The orbifold line bundle $\widetilde{N}$ may be taken to be the orbifold line bundle associated to the divisor $bx + bx' + \sum_{i=1}^n (x'_i + x''_i)$. It follows that $W$ is a complex orbifold (with boundary) and the singularities corresponding to $x'_i, x''_i$ are of the form $\mathbb{C}^2/\mathbb{Z}_{a_i}$, where $\mathbb{Z}_{a_i}$ acts on $\mathbb{C}^2$ as scalar multiplication by $e^{2\pi i/a_i}$. The singularities can be resolved by blowing up the singular points. Let $\widetilde{X}$ denote the blowup of $\widetilde{W}$ at the singular points. Thus $\widetilde{M}$ is a complex manifold which bounds $\widetilde{Y}$. Equivalently, $\widetilde{M}$ is the plumbing on the star-shaped graph $\widetilde{\Gamma}$ given in Figure \ref{fig:star2}. Let $\Sigma_0, \Sigma'_1, \Sigma''_1, \dots , \Sigma'_n , \Sigma''_n$ denote the zero sections of each disc bundle, where $\Sigma_0$ corresponds to the central node and $\Sigma'_i, \Sigma''_i$ to the nodes with weight $-a_i$. Then $e_0 = [\Sigma_0], e'_i = [\Sigma'_i], e''_i = [\Sigma''_i]$ form a basis for the lattice $\widetilde{L} = H_2(\widetilde{X} ; \mathbb{Z})$ with $e_0^2 = 2b$, $(e'_i)^2 = (e''_i)^2 = -a_i$, $\langle e_0  ,e'_i \rangle = \langle e_0 , e''_i \rangle = 1$, $\langle e'_i , e'_j \rangle = \langle e'_i , e''_j \rangle = 0$ for $i \neq j$. Let $\widetilde{L}^* = H^2(\widetilde{X} ; \mathbb{Z}) \cong H_2(\widetilde{X} , \partial \widetilde{X} ; \mathbb{Z})$ be the dual lattice. A dual basis $D_0, D'_i , D''_i$ is given by taking the relative homology classes of a fibre over each node. One can easily check that 
\begin{equation}\label{equ:e0}
e_0 = 2 b D_0 + \sum_i (D'_i + D''_i).
\end{equation}

Let $\widetilde{\nu}$ denote a closed tubular neighbourhood of $\Sigma_0$ in $\widetilde{X}$. Then $\widetilde{\nu}$ can be identified with $\widetilde{W}(b)$, the unit disc bundle over $S^2$ with Euler class $2b$. Its boundary is $\widetilde{Y}(b)$, the circle bundle over $S^2$ with Euler class $2b$. Let $\widetilde{M}$ be obtained from $\widetilde{X}$ by removing the interior of $\widetilde{\nu}$. Then $\widetilde{M}$ is a complex manifold with boundary. We view it as a cobordism with ingoing boundary $\widetilde{Y}(b)$ and with outgoing boundary $\widetilde{Y}$.

Let $W = D(N)$ be the closed unit disc bundle of $N$. Since $\widetilde{N}$ is the pullback of $N$ to $\widetilde{\Sigma}$, the covering involution $\sigma \colon \widetilde{\Sigma} \to \widetilde{\Sigma}$ lifts to an involution on $\widetilde{W}$ and $W = \widetilde{W}/\sigma$. Since $\sigma$ acts anti-holomorphically on $\widetilde{W}$, it induces an involution (also denoted by $\sigma$) on $\widetilde{X}$ and $X = \widetilde{X}/\sigma$. Equivalently, $X$ is the plumbing on the graph $\Gamma$ shown in Figure \ref{fig:star2}, where the central node is the disc bundle $W(b)$ of Euler class $b$ over $\mathbb{RP}^2$. A basis $e_1 , \dots , e_n$ for the intersection lattice $L = H_2(X  ; \mathbb{Z})$ is given by taking $e_i = [\Sigma_i]$ to be the zero section of the disc bundle over the node of weight $-a_i$ (so the preimage of $\Sigma_i$ in $\widetilde{X}$ is $\Sigma'_i \cup \Sigma''_i$). Notice that there is no homology class corresponding to the central node because $H_2(\mathbb{RP}^2 ; \mathbb{Z}) = 0$.

We choose the tubular neighbourhood $\widetilde{\nu}$ to be $\sigma$-invariant so that $\sigma$ acts on $\widetilde{M}$ and the quotient $M = \widetilde{M}/\sigma$ is obtained from $X$ by removing the interior of $\nu = \widetilde{\nu}/\sigma$. Since $\nu$ can be identified with $W(b)$, we see that $M$ can be regarded as a cobordism with ingoing boundary $Y(b)$ and outgoing boundary $Y$.

To summarise, we have a decomposition
\begin{equation}\label{equ:X}
X = W(b) \cup_{Y(b)} M
\end{equation}
and corresponding decomposition
\begin{equation}\label{equ:Xtilde}
\widetilde{X} = \widetilde{W}(b) \cup_{\widetilde{Y}(b)} \widetilde{M},
\end{equation}
where each term in the decomposition is a double cover of the corresponding term in the decomposition for $X$.

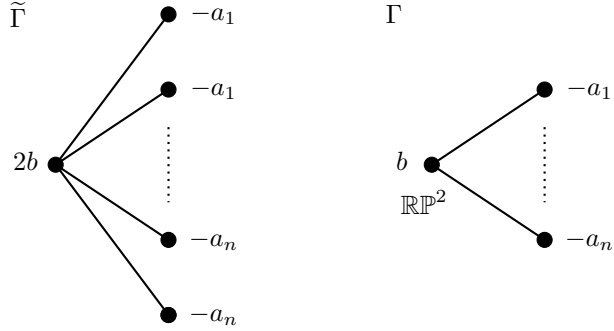
\begin{figure}

\begin{tikzpicture}
\node at (-1,2) {$\widetilde{\Gamma}$} ;
\draw[thick] (-0.5,0) -- (1,2) ;
\draw [fill] (-0.5,0) circle(0.1);
\draw [fill] (1,2) circle(0.1);
\draw[thick] (-0.5,0) -- (1,1) ;
\draw [fill] (1,1) circle(0.1);
\draw[thick] (-0.5,0) -- (1,-2) ;
\draw [fill] (1,-1) circle(0.1);
\draw [fill] (1,-2) circle(0.1);
\draw [dotted, thick] (1,0.5) -- (1,-0.5);
\node at (-0.9,0.05) {$2b$};
\node at (1.6,2) {$-a_1$};
\node at (1.6,1) {$-a_1$};
\node at (1.6,-1) {$-a_n$};
\draw[thick] (-0.5,0) -- (1,-1) ;
\node at (1.6,-2) {$-a_n$};
\draw [fill] (1,-2) circle(0.1);

\node at (5-1,2) {$\Gamma$} ;
\draw [fill] (5-0.5,0) circle(0.1);
\draw[thick] (5-0.5,0) -- (5+1,1) ;
\draw [fill] (5+1,1) circle(0.1);
\draw [fill] (5+1,-1) circle(0.1);
\draw [dotted, thick] (5+1,0.5) -- (5+1,-0.5);
\node at (5-0.9,0.05) {$b$};
\node at (5+1.6,1) {$-a_1$};
\node at (5+1.6,-1) {$-a_n$};
\draw[thick] (5-0.5,0) -- (5+1,-1) ;
\node at (5-0.6,-0.5) {$\mathbb{RP}^2$} ;

\end{tikzpicture}
\caption{Plumbing graphs $\widetilde{\Gamma}$ and $\Gamma$}\label{fig:star2}
\end{figure}

Using the Mayer--Vietoris sequences corresponding to the decompositions (\ref{equ:X}) and (\ref{equ:Xtilde}). One finds $b_1(M) = b_1(\widetilde{M}) = 0$ and that $\Lambda = H_2(M ; \mathbb{Z})$ and $\widetilde{\Lambda} = H_2(\widetilde{M} ; \mathbb{Z})$ are torsion-free. So $\Lambda, \widetilde{\Lambda}$ are the intersection lattices of $M$ and $\widetilde{M}$.

We have a short exact sequence
\[
0 \to \mathbb{Z}_2 \to H_1( Y(b) ; \mathbb{Z}) \to \mathbb{Z}_2 \to 0,
\]
where the $\mathbb{Z}_2$-subgroup is generated by $f$, the homology class of a fibre of $Y(b) \to \mathbb{RP}^2$. The boundary map $\partial \colon H_2( X  ; \mathbb{Z}) \to H_1( Y(b) ; \mathbb{Z})$ in the Mayer--Vietoris sequence corresponding to (\ref{equ:X}) is easily seen to be given by $\partial e_i = f$ for $i = 1, \dots , n$. So we have a short exact sequence
\[
0 \to \Lambda \to L \buildrel \partial \over \longrightarrow \mathbb{Z}_2 \to 0,
\]
and dually
\[
0 \to L^* \to \Lambda^* \to \mathbb{Z}_2 \to 0.
\]

Let $\widetilde{L}_0$ denote the sublattice of $\widetilde{L}$ spanned by $D'_1, D''_1 , \dots , D'_n , D''_n$. From the Mayer--Vietoris sequence for $\widetilde{X}$, one similarly obtains a short exact sequence
\[
0 \to \widetilde{\Lambda} \to \widetilde{L}_0 \buildrel \partial \over \longrightarrow \mathbb{Z}_{2|b|} \to 0
\]
where $\partial$ is defined by $\partial D'_i = \partial D''_i = 1$. Dually we get
\[
0 \to \widetilde{L}^*_0 \to \widetilde{\Lambda}^* \to \mathbb{Z}_{2|b|} \to 0.
\]
Define
\[
\widetilde{\partial} = \sum_i (D'_i + D''_i).
\]
From the definition of $\partial$, we see that $\frac{1}{2b}\widetilde{\partial}$ represents a generator of $\widetilde{\Lambda}^*/\widetilde{L}_0^* \cong \mathbb{Z}_{2|b|}$. Since $\Sigma_0$ is contained in the complement of $\widetilde{M}$ in $\widetilde{X}$, we see that $e_0 \in \widetilde{L} \subseteq \widetilde{L}^* \cong H^2( \widetilde{X} ; \mathbb{Z})$ maps to zero   under $H^2( \widetilde{X} ; \mathbb{Z}) \to H^2(\widetilde{M} ; \mathbb{Z}) \cong \widetilde{\Lambda}^*$. Therefore from Equation (\ref{equ:e0}) we deduce that in $\widetilde{\Lambda}^*$ we have the equality
\[
\widetilde{\partial} = -2b D_0,
\]
or equivalently, $(1/2b)\widetilde{\partial} = -D_0$.

Recall that $D_0 = [F_0]$ is the relative homology class of a fibre of the disc bundle over the central node. Then $F_0$ is an analytic hypersurface in $\widetilde{X}$ and so there is an associated holomorphic line bundle $\mathcal{L}_0$ on $\widetilde{X}$ which has a holomorphic section whose zero divisor is $F_0$. It follows that $c_1(\mathcal{L}_0) = D_0$. Notice also that if $\mathcal{L}_0|_{\widetilde{Y}} \cong \pi^*( \mathcal{O}(x) )$, where $x$ is the point of $\widetilde{\Sigma}$ over which the fibre $F_0$ lies. Similarly, $D'_i = [F'_i]$, $D''_i = [F''_i]$ are the relative homology classes of fibres of the disc bundles over the nodes of weight $-a_i$. Associated to $F'_i, F''_i$ are holomorphic line bundles $\mathcal{L}'_i, \mathcal{L}''_i$ with $c_1(\mathcal{L}'_i) = D'_i$, $c_1(\mathcal{L}''_i) = D''_i$. We also have that $\mathcal{L}'_i|_{\widetilde{Y}} \cong \pi^*( \mathcal{O}(x'_i) )$, $\mathcal{L}''_i|_{\widetilde{Y}} \cong \pi^*( \mathcal{O}(x''_i))$.

\begin{theorem}\label{thm:etapinc}
We have $\eta^{pin^c}(\Sigma , \mathfrak{s}^\Sigma) = \pm 1/2$ for every pin$^c$-structure on $\Sigma$.
\end{theorem}
\begin{proof}
Recall that every pin$^c$-structure on $\Sigma$ is of the form $\mathfrak{s}^\Sigma_{(E_0 , \tau)}$ for some orbifold line bundle $E_0$ on $\widetilde{\Sigma}$ and $\tau$ is an equivariant structure on $E_0$ of weight $m=0$. We have $[E_0] = (e ; \gamma_1 , \delta_1 , \dots , \gamma_n , \delta_n)$ where $e = -\chi(\Sigma) = -1 + \sum_i (a_i-1)/a_i$ and $\gamma_i + \delta_i = -1 \; ({\rm mod} \; a_i)$. It follows that we can uniquely lift $\gamma_i,\delta_i$ to integers satisfying $0 \le \gamma_i , \delta_i \le a_i - 1$, $\gamma_i + \delta_i = a_i - 1$. Notice that $deg|E_0| = -1$ and thus $E_0$ is the line bundle associated to the divisor $-p + \sum_i ( \gamma_i x'_i + \delta_i x''_i)$, where $p$ is any non-orbifold point of $\widetilde{\Sigma}$. Let $E = \pi^*(E_0)$. Then $E$ extends to a line bundle $\widehat{E}$ on $\widetilde{M}$ given by
\[
\widehat{E} = \left. \left( \mathcal{L}_0^{*} \otimes \bigotimes_{i=1}^n \left( (\mathcal{L}_i')^{\gamma_i} \otimes (\mathcal{L}''_i)^{\delta_i} \right) \right)\right|_{\widetilde{M}}.
\]
Then we have
\[
c_1( \widehat{E} ) = -D_0 + \sum_i (\gamma_i D'_i + \delta_i D''_i).
\]

Let $\mathfrak{s}_{\widehat{E}}$ be the spin$^c$-structure on $\widehat{M}$ given by $\mathfrak{s}_{\widehat{E}} = \widehat{E} \otimes \mathfrak{s}_{can}(\widetilde{M})$. Thus $\mathfrak{s}_{\widehat{E}}$ is an extension of $\mathfrak{s}_E$ to $\widetilde{M}$. We have
\begin{align*}
c_1( \sigma^*(\widehat{E}^*) ) + c_1(\widehat{E}) &= -2D_0 + \sum_i (  \gamma_i + \delta_i)(D'_i + D''_i) \\
&= -2D_0 + \sum_i ( a_i - 1 )(D'_i + D''_i) \\
&= -2D_0 + \widetilde{\partial} + \sum_i (a_i - 2)(D'_i + D''_i) \\
&= (-2-2b)D_0 + \sum_i (a_i - 2)(D'_i + D''_i) \\
&= c_1( K_{\widetilde{M}}).
\end{align*}

So $\sigma^*(\widehat{E}) \cong \widehat{E} \otimes K_{\widetilde{M}}^{-1}$ and consequently $\mathfrak{s}_{\widehat{E}}$ can be made into an equivariant spin$^c$-structure on $\widetilde{M}$ which then descends to a spin$^c$-structure on $M$, which we denote by $\mathfrak{s}_M$. To see why this is so, first note that $\sigma^*(\widehat{E}) \cong \widehat{E} \otimes K_{\widetilde{M}}^{-1}$ implies that $\sigma^*(\mathfrak{s}_{\widehat{E}}) \cong \mathfrak{s}_{\widehat{E}}$. Therefore $\sigma$ can be lifted to a map $\widetilde{\sigma} \colon S_{\widehat{E}} \to S_{\widehat{E}}$ on the spinor bundles for $\mathfrak{s}_{\widehat{E}}$ preserving the unitary structure and Clifford multiplication. Since $\widetilde{\sigma}^2$ covers the identity, we have $\widetilde{\sigma}^2 = h$ for some $h \colon \widetilde{M} \to S^1$. But $h \circ \widetilde{\sigma} = \widetilde{\sigma}^3 = \widetilde{\sigma} \circ h = \sigma^*(h) \circ \widetilde{\sigma}$ and therefore $\sigma^*(h) = h$. Since $b_1(\widetilde{M}) = 0$, we can write $h = e^{2\pi i f}$ for some $f \colon \widetilde{M} \to \mathbb{R}$. Then $\sigma^*(h) = h$ implies $\sigma^*(f) = f + k$ for some integer $k$. But then $f = \sigma^* \sigma^* (f) = f + 2k$, so $k = 0$ and $\sigma^*(f) = f$. It then follows that $e^{-\pi i f} \circ \widetilde{\sigma}$ is an involutive lift of $\sigma$, and this makes $\mathfrak{s}_{\widehat{E}}$ into an equivariant spin$^c$-structure.

Since $b_1(\widetilde{M}) = 0$, it follows from the Borel spectral sequence for $H^*_{\mathbb{Z}_2}(\widetilde{M} ; \mathbb{Z}) \cong H^*(M ; \mathbb{Z})$ that the pullback map $H^2( M ; \mathbb{Z}) \to H^2( \widetilde{M} ; \mathbb{Z})$ is a 2:1 map. Thus, up to isomorphism there are two ways to make $\mathfrak{s}_{\widehat{E}}$ into an equivariant spin$^c$-structure. If $\widetilde{\sigma}$ is an involutive lift of $\sigma$ to $S_{\widehat{E}}$, then so is $-\widetilde{\sigma}$ and $\pm \widetilde{\sigma}$ give the two possible equivariant structures. Restricted to $\widetilde{Y}$, these two equivariant structures give the two equivariant structures on $\mathfrak{s}_E$. Therefore, we can choose the equivariant structure on $\mathfrak{s}_{\widehat{E}}$ such that its restriction to $\widetilde{Y}$  coincides with the equivariant structure determined by $\tau$. Therefore, $\mathfrak{s}_M$ is an extension of $\mathfrak{s}_{(E_0 , \tau , 0)}$ to $M$.

Since $\mathfrak{s}_M$ pulls back to $\mathfrak{s}_{\widehat{E}}$, we have
\[
c(\mathfrak{s}_M)^2 = \frac{1}{2} c(\mathfrak{s}_{\widehat{E}})^2 = \frac{1}{2} ( 2c_1(\widehat{E}) - K_{\widetilde{M}})^2.
\]
Further, we have
\begin{align*}
2c_1(\widehat{E}) - K_{\widetilde{M}} &= -2D_0 + \sum_i (2\gamma_i D'_i + 2\delta_i D''_i) + (2b+2)D_0 - \sum_i ( a_i - 2)(D'_i + D''_i) \\
&= 2bD_0 + \sum_i ( (2\gamma_i + 2 - a_i)D'_i + (2\delta_i + 2 - a_i)D''_i ) \\
&= -\widehat{\partial} + \sum_i ( (2\gamma_i + 2 - a_i)D'_i + (2\delta_i + 2 - a_i)D''_i ) \\
&= \sum_i( (2\gamma_i + 1 - a_i)D'_i + (2\delta_i + 1 - a_i)D''_i ).
\end{align*}
Note that $(2\delta_i + 1-a_i) = -(2\gamma_i + 1 - a_i)$ because $\gamma_i + \delta_i = a_i - 1$. Therefore
\begin{equation}\label{equ:c2}
c(\mathfrak{s}_M)^2 = \frac{1}{2}(2c_1(\widehat{E}) - K_{\widetilde{M}})^2 = \sum_i -\frac{(2\gamma_i + 1 - a_i)^2}{a_i}.
\end{equation}

The calculation of $c(\mathfrak{s}_{\widehat{E}})$ given above shows that $c(\mathfrak{s}_{\widehat{E}})|_{\widetilde{Y}(b)} = 0$ because $D'_i|_{Y(b)} = D''_i|_{Y(b)} = 0$.

Recall that $M$ is a cobordism from $Y(b)$ to $Y$. Let $\mathfrak{s}_{Y(b)}$ denote the restriction of $\mathfrak{s}_M$ to $Y(b)$. Since $Y(b)$ is a non-singular circle bundle over $\mathbb{RP}^2$, it has four spin$^c$-structures, two of which extend to $W(b)$ (those with non-trivial fibre holonomy) and two of which do not extend (those with trivial fibre holonomy). We claim that $\mathfrak{s}_{Y(b)}$ does not extend to $W(b)$, equivalently, it has trivial fibre holonomy. To see this, note that the spin$^c$-structures with trivial fibre holonomy are those of the form $\mathfrak{s}_{u,v}$ with $v=0$ in the terminology of \cite[\textsection 2.2]{bar}. But it is easily seen that $\pi^*(c(\mathfrak{s}_{u,v})) \in H^2( \widetilde{Y}(b) ; \mathbb{Z})$ is zero if and only if $v=0$. Then since $c(\mathfrak{s}_{\widehat{E}})|_{\widetilde{Y}(b)} = 0$, it follows that $\mathfrak{s}_M|_{Y(b)}$ is of the form $\mathfrak{s}_{u,0}$, hence has trivial fibre holonomy.

Since $M$ is negative definite and $b_2(M) = 2$, the Fr\o yshov inequality \cite{fro}, \cite[Theorem 39.1.4]{km} gives:
\begin{equation}\label{equ:ineq1}
\delta( Y , \mathfrak{s}_{(E_0 , \tau , 0)} ) \ge \delta(Y(b)  ,\mathfrak{s}_{Y(b)} ) + \frac{ c(\mathfrak{s}_{M})^2 + n}{8}.
\end{equation}
The delta invariants of $Y(b)$ are computed in \cite[\textsection 2.2]{bar}. Since $\mathfrak{s}_{Y(b)}$ does not extend to $W(b)$, we have
\[
\delta(Y(b), \mathfrak{s}_{Y(b)}) = \frac{b}{8} + \frac{1}{4}\epsilon,
\]
where $\epsilon = \pm 1$ depending on the spin$^c$-structure.

Recall (eg, \cite[\textsection 2.1]{bar}) that the delta invariants of $L(a_i , 1)$ are given by
\[
\delta( L(a_i , 1) , \mathfrak{s}_{u}) = - \frac{ (2u + 2 - a_i)^2 }{8a_i} + \frac{1}{8}.
\]
Note also that $\delta( L(a_i , -1) , \mathfrak{s}_u) = -\delta( L(a_i , 1) , \mathfrak{s}_u)$. Then from Theorem \ref{thm:deltaY}, we have
\begin{align*}
& \delta(Y , \mathfrak{s}_{(E_0 , \tau , 0)}) \\
& = \frac{l}{8} - \frac{1}{2} \eta^{pin^c}(\Sigma , \mathfrak{s}_{(E_0 , \tau)}^\Sigma) + \sum_{i=1}^n \left( - \frac{(2\gamma_i + 2 - a_i)^2}{16a_i} - \frac{(2\delta_i + 2 - a_i)^2}{16a_i} + \frac{1}{8} \right) \\
&=  \frac{l}{8} - \frac{1}{2} \eta^{pin^c}(\Sigma , \mathfrak{s}_{(E_0 , \tau)}^\Sigma) + \sum_{i=1}^n \left( -\frac{(2\gamma_i + 2 - a_i)^2}{16a_i} - \frac{(2\gamma_i - a_i)^2}{16a_i} + \frac{1}{8} \right) \\
&=  \frac{l}{8} - \frac{1}{2} \eta^{pin^c}(\Sigma , \mathfrak{s}_{(E_0 , \tau)}^\Sigma) + \sum_{i=1}^n \left( -\frac{(2\gamma_i + 1 - a_i)^2}{8a_i} - \frac{1}{8a_i} + \frac{1}{8} \right) \\
&= \frac{b}{8} - \frac{1}{2} \eta^{pin^c}(\Sigma , \mathfrak{s}_{(E_0 , \tau)}^\Sigma) + \sum_{i=1}^n \left( -\frac{(2\gamma_i + 1 - a_i)^2}{8a_i} + \frac{1}{8} \right) \\
\end{align*}
where the third line was obtained using $\delta_i = a_i - 1 - \gamma_i$. Substituting into (\ref{equ:ineq1}) and also using Equation (\ref{equ:c2}) gives
\begin{align*}
&\frac{b}{8} - \frac{1}{2} \eta^{pin^c}(\Sigma , \mathfrak{s}_{(E_0 , \tau)}^\Sigma) + \sum_{i=1}^n \left( -\frac{(2\gamma_i + 1 - a_i)^2}{8a_i} + \frac{1}{8} \right) \\
& \quad \quad \ge \frac{b}{8} + \frac{1}{4}\epsilon + \frac{c(\mathfrak{s}_M)^2}{8} + \frac{n}{8} \\
& \quad \quad = \frac{b}{8} + \frac{1}{4}\epsilon + \sum_i \left( -\frac{(2\gamma_i + 1 - a_i)^2}{8a_i} \right) + \frac{n}{8} \\
& \quad \quad = \frac{b}{8} + \frac{1}{4}\epsilon + \sum_i \left( -\frac{(2\gamma_i + 1 - a_i)^2}{8a_i} + \frac{1}{8} \right).\\
\end{align*}

Hence
\begin{equation}\label{equ:pinc1}
\eta^{pin^c}(\Sigma , \mathfrak{s}_{(E_0 , \tau)}^\Sigma) \le -\frac{1}{2}\epsilon.
\end{equation}
Since $\tau$ was arbitrary, we can repeat the same argument for $-\tau$. Recall that $\mathfrak{s}_M|_Y \cong \mathfrak{s}_{(E_0 , \tau , 0)}$. Let $A \to M$ denote the flat line bundle of order $2$ corresponding to the double cover $\widetilde{M} \to M$ ($c_1(A)$ is the unique non-zero $2$-torsion class in $H^2(M ; \mathbb{Z})$). Set $\mathfrak{s}'_{M} = A \otimes \mathfrak{s}_M$. Since the double cover $\widetilde{M} \to M$ restricts to the double cover $\widetilde{Y} \to Y$ over $Y$, it follows that $\mathfrak{s}'_{M}|_{Y} \cong \mathfrak{s}_{(E_0 , -\tau , 0)}$. Similarly, since $\widetilde{M} \to M$ restricts to the non-trivial double cover $\widetilde{Y}(b) \to Y(b)$ on $Y(b)$, it follows that $\mathfrak{s}'_M|_{Y(b)} \ncong \mathfrak{s}_M|_{Y(b)}$ and therefore
\[
\delta( Y(b) , \mathfrak{s}'_M|_{Y(b)}) = \frac{b}{8} - \frac{1}{4}\epsilon.
\]
Repeating the above calculations with $\mathfrak{s}'_M$ in place of $\mathfrak{s}_M$ then gives
\[
\eta^{pin^c}(\Sigma , \mathfrak{s}_{(E_0 , -\tau)}^\Sigma) \le \frac{1}{2}\epsilon.
\]
But since $\eta^{pin^c}(\Sigma , \mathfrak{s}_{(E_0 , -\tau)}^\Sigma) = -\eta^{pin^c}(\Sigma , \mathfrak{s}_{(E_0 , \tau)}^\Sigma)$, this gives
\begin{equation}\label{equ:pinc2}
\eta^{pin^c}(\Sigma , \mathfrak{s}_{(E_0 , \tau)}^\Sigma) \ge -\frac{1}{2}\epsilon.
\end{equation}
The inequalities (\ref{equ:pinc1}), (\ref{equ:pinc2}) together give
\[
\eta^{pin^c}(\Sigma , \mathfrak{s}_{(E_0 , \tau)}^\Sigma) = -\frac{1}{2}\epsilon.
\]
In particular, $\eta^{pin^c}(\Sigma , \mathfrak{s}_{(E_0 , \tau)}^\Sigma) = \pm 1/2$.

\end{proof}

\subsection{$\mathbb{RP}^2$}\label{sec:rp2}

In the case that $\Sigma = \mathbb{RP}^2$ (no orbifold points), $Y = S(b)$ is a circle bundle over $\mathbb{RP}^2$ and $l = deg(N) = b$. There are $4$ spin$^c$-structures on $Y$, two with trivial fibre holonomy and two with non-trivial fibre holonomy. The ones with non-trivial fibre holonomy have $\delta(Y , \mathfrak{s}) = 0$ by Theorem \ref{thm:deltaY} and the ones with trivial fibre holonomy have
\[
\delta( Y , \mathfrak{s} ) = \frac{b}{8} - \frac{1}{2} \eta^{pin^c}(\mathbb{RP}^2 , \mathfrak{s}^{\mathbb{RP}^2}),
\]
where $\mathfrak{s}^{\mathbb{RP}^2}$ denotes the pin$^c$-structure on $\Sigma$ which uniquely corresponds to $\mathfrak{s}$. On the other hand the delta invariants for $Y = S(b)$ were computed in \cite[\textsection 2.2]{bar} and are given by
\[
0,0, \frac{b+2}{8}, \frac{b-2}{8}.
\]
Since $\eta^{pin^c}(\mathbb{RP}^2 , \mathfrak{s}^{\mathbb{RP}^2})$ is independent of $b$, we must have $\eta^{pin^c}( \mathbb{RP}^2 , \mathfrak{s}^{\mathbb{RP}^2}) = \pm 1/2$. More precisely, it equals $1/2$ for one pin$^c$-structure and equals $-1/2$ for the other.

\subsection{$\mathbb{RP}^2(a)$}\label{sec:rp2a}

In this section we will compute $\eta^{pin^c}$ in the case that $\Sigma = \mathbb{RP}(a)$ has a single orbifold point using different methods to Theorem \ref{thm:etapinc}.

Suppose that $\Sigma = \mathbb{RP}^2(a)$ has a single orbifold point of order $a$ and $Y = S( b ; (a , b_1))$. The fundamental group of $Y$ is isomorphic to
\[
\langle v , q , h \; | \; vhv^{-1}h, [q,h] , q^a h^{b_1} , qv^2 h^b \rangle.
\]
The relation $qv^2 h^b = 1$ can be solved for $q$, giving $q = h^{-b} v^{-2}$. Since $v^2$ and $h$ commute (by the relation $vhv^{-1} = h^{-1}$), we see that $q$ and $h$ automatically commute. Lastly, if we substitute for $q$ in the relation $q^a h^{b_1} = e$, we get $v^{-2a} = h^r$, where $r = a_1 l = ab - b_1$. Conjugating by $v$, we see that the relation $v^{-2a} = h^r$ can be replaced by the relation $v^{2a} = h^r$. Thus
\[
\pi_1(Y) \cong \langle v , h \; | \; vhv^{-1} = h^{-1}, \; v^{2a} = h^r \rangle.
\]
This is a finite metacyclic group of order $4a|r|$. Thus $Y$ is a spherical $3$-manifold, in fact a prism manifold. We will compute the delta invariants of $Y$ using its spherical geometry. Comparing with Theorem \ref{thm:deltaY}, this will allow us to determine $\eta^{pin^c}( \Sigma , \mathfrak{s}^\Sigma)$ for all pin$^c$-structures on $\Sigma$.

Prism manifolds can be constructed as quotients of $S^3$ as follows, where $S^3$ is the unit sphere in $\mathbb{C}^2$. Let $r,m$ be positive coprime integers and set $G = \langle v , h \; | vhv^{-1} = h^{-1}, \; v^{2m} = h^r \rangle$. We can regard $G$ as a subgroup of $U(2)$ by taking 
\[
v = \left[ \begin{matrix} 0 & z \\ 1 & 0 \end{matrix} \right], \quad h = \left[ \begin{matrix} w & 0 \\ 0 & w \end{matrix} \right],
\]
where $z = e^{\pi i/m}$, $w = e^{\pi i/r}$. This group acts freely on $S^3$ and the quotient $Y = S^3/G$ is a prism manifold where $a = m$, $b_1 = ab - r$. Let $\mathfrak{s}_{can}$ denote the spin$^c$-structure on $S^3$ which is the restriction of the canonical spin$^c$-structure on $\mathbb{C}^2$. Since $G$ acts on $\mathbb{C}^2$ biholomorphically, it naturally lifts to $\mathfrak{s}_{can}$ and so $\mathfrak{s}_{can}$ becomes an equivariant spin$^c$-structure on $S^3$, which descends to a spin$^c$-structure on $Y$, which we continue to denote by $\mathfrak{s}_{can}$. Any other spin$^c$-structure on $Y$ will be of the form $E \otimes \mathfrak{s}_{can}$, where $E$ is a line bundle on $Y$. All such line bundles are of the form $E = S^3 \times_G \mathbb{C}$ where $G$ acts on $\mathbb{C}$ according to some character $\phi : G \to S^1$. Such a homomorphism is determined by $\phi(v),\phi(h)$. Since $vhv^{-1} = h^{-1}$, we get $\phi(h)^2 = 1$, so $\phi(h) = (-1)^{\nu}$ for some $\nu \in \{0,1\}$. We also have $v^{2m} = h^r$, so $\phi(v)^{2m} = (-1)^{r \nu}$. Hence we can write $\phi(v) = e^{\pi i r\nu/2m} e^{\pi i u/m}$ for a unique $u \in \{ 0 , 1 , \dots , 2m-1\}$. This gives $4m$ spin$^c$-structures depending on the values of $\nu \in \{ 0,1\}$ and $u \in \{0,1, \dots , 2m-1\}$.

Let $\widetilde{G}$ be the index two subgroup of $G$ generated by $v^2$ and $h$ and set $\widetilde{Y} = S^3/\widetilde{G}$. Then $\widetilde{Y} \to Y$ is usual double cover of $Y$ obtained from the orientation double cover $\widetilde{\Sigma} \to \Sigma$. Recall that spin$^c$-structures on $\widetilde{Y}$ come in pairs $\mathfrak{s}_{(E_0 , \tau,m)}$, $\mathfrak{s}_{(E_0 , -\tau,m)}$. Since $Y$ is obtained from $\widetilde{Y}$ by quotienting by $v$, we see that if $(E_0 , \tau, m)$ corresponds to a homomorphism $\phi : G \to S^1$, then $(E_0 , -\tau , m)$ corresponds to the homomorphism $\phi' : G \to S^1$ given by $\phi'(v) = \phi(v)$, $\phi'(h) = -\phi(h)$. Equivalently, $\phi'$ corresponds to $(\nu' ,u')$ where $\nu' = \nu$ and $u' = u + m \; ({\rm mod} \; 2m)$. To a homomorphism $\phi : G \to S^1$, we let $\mathfrak{s}_\phi$ denote the corresponding spin$^c$-structure.

Give $S^3 \subset \mathbb{C}^2$ its induced metric. This descends to a spherical metric on $Y$. Let $\eta_{dir}(Y , \mathfrak{s}_\phi)$ denote the eta invariant of the spin$^c$-Dirac operator using this metric and a spin$^c$-connection $A$ covering the Levi--Civita connection with $F_A = 0$. Using the same method as \cite[\textsection 2]{bar}, the eta invariant can be computed as
\[
\eta_{dir}(Y , \mathfrak{s}_\phi) = -\frac{2}{|G|} \sum_{g \in G, g \neq 1} \frac{\phi(g)}{det(1-g^{-1})}.
\]
Every element of $G$ can uniquely be written as $g = v^i h^j$, $0 \le i \le 2m-1, 0\le j \le 2r-1$. When $i = 2k$ is even, $0 \le k \le m-1$, we find that
\[
g = (v^2)^{k} y^j = \left[ \begin{matrix} \lambda & 0 \\ 0 & \mu \end{matrix} \right],
\]
where we set $\lambda = z^k w^j$, $\mu = z^k w^{-j}$. Therefore $det(1-g^{-1}) = (1 - \lambda^{-1})(1- \mu^{-1})$. We also have $\phi(g) = z^{(2u+r\nu)k} (-1)^{\nu j}$.

When $i = 2k+1$ is odd, $0 \le k  \le m-1$, we have 
\[
g = v (v^2)^k y^j = \left[ \begin{matrix} 0 & z \\ 1 & 0 \end{matrix} \right] \left[ \begin{matrix} \lambda & 0 \\ 0 & \mu \end{matrix} \right] = \left[ \begin{matrix} 0 & z\mu \\ \lambda & 0 \end{matrix} \right]
\]
and so $det(1-g^{-1}) = 1 - z^{-1}\lambda \mu = 1 - z^{-(2k+1)}$. We also have $\phi(g) = (-1)^{\nu j } \phi(x)^{2k+1}$. Thus
\begin{align*}
\eta_{dir}(Y , \mathfrak{s}_\phi) &= -\frac{1}{2mr} \sum_{\substack{0 \le k \le m-1 \\ 0 \le j \le 2r-1 \\ (k,j) \neq (0,0)} } \frac{ (-1)^{\nu j} z^{(2u + r\nu)k} }{(1-z^{-k}w^{-j})(1-z^{-k}w^j)} \\
& \quad - \frac{1}{2mr} \sum_{\substack{0 \le k \le m-1 \\ 0 \le j \le 2r-1}} \frac{ (-1)^{\nu j} \phi(v)^{2k+1} }{(1-z^{-(2k+1)})}
\end{align*}

Notice that the first summation on the right depends on $u$ only through the term $z^{2uk}$. Replacing $u$ by $u' = u + m$ does not change the value of this term. Letting $\phi' : G \to S^1$ be the homomorphism corresponding to $\nu' = \nu$, $u' = u + m \; ({\rm mod} \; 2m)$, we get that
\[
\eta_{dir}(Y , \mathfrak{s}_\phi) - \eta_{dir}(Y , \mathfrak{s}_{\phi'}) = - \frac{1}{mr} \sum_{\substack{0 \le k \le m-1 \\ 0 \le j \le 2r-1}} \frac{ (-1)^{\nu j} \phi(v)^{2k+1} }{(1-z^{-(2k+1)})}.
\]
Now we evaluate the sum on the right. If $\nu = 1$, then the terms with $j$ even cancel with the terms with $j$ odd, giving zero. Now suppose that $\nu = 0$. Then $\phi(v) = z^u$ and the sum reduces to
\begin{align*}
& -\frac{2}{m} \sum_{0 \le k \le m-1} \frac{ z^{(2k+1)u} }{(1-z^{-(2k+1)})} \\
& = -\frac{2}{m} \left( \sum_{0 \le k \le 2m-1} \frac{ z^{ku} }{1-z^{-k}} - \sum_{0\le k \le m-1} \frac{ z^{2ku} }{(1-z^{-2ku}) } \right).
\end{align*}

Recall (eg, \cite[Lemma 2.1]{bar}) that for any integers $p,q$, $p > 0$, we have
\[
\frac{1}{p} \sum_{k=0}^{p-1} \frac{ e^{2\pi i kq/p} }{ (1- e^{-2\pi ik/p}) } = \frac{p-1}{2p} - \left\{ \frac{q}{p} \right\}. 
\]
Therefore, when $\nu = 0$, we have
\begin{align*}
& -\frac{2}{m} \left( \sum_{0 \le k \le 2m-1} \frac{ z^{ku} }{1-z^{-k}} - \sum_{0\le k \le m-1} \frac{ z^{2ku} }{(1-z^{-2ku}) } \right) \\
& \quad -2 \left( \frac{2m-1}{2m} - 2\left\{ \frac{u}{2m} \right\} - \frac{(m-1)}{2m} + \left \{ \frac{u}{m} \right\} \right) \\
& \quad = \begin{cases} -1 & 0 \le u \le m-1, \\ 1 & m \le u \le 2m-1. \end{cases}
\end{align*} 

Now because $Y$ is spherical, we have that $\delta(Y , \mathfrak{s}_\phi) = -(1/2) \eta_{dir}(Y , \mathfrak{s}_\phi) + (1/8)\eta_{sig}(Y)$, where $\eta_{sig}(Y)$ is computed using the spherical metric as in \cite[\textsection 2.5]{bar}. Therefore,
\begin{align*}
\delta(Y , \mathfrak{s}_\phi) - \delta(Y , \mathfrak{s}_{\phi'}) &= -\frac{1}{2}\left( \eta_{dir}(Y , \mathfrak{s}_\phi) - \eta_{dir}(Y , \mathfrak{s}_{\phi'}) \right) \\
&= \begin{cases} 0 & \nu = 1, \\ 1/2 & \nu = 0, \; 0 \le u \le m-1, \\ -1/2 & \nu = 0, \; m \le u \le 2m-1. \end{cases}
\end{align*}

On the other hand, from Theorem \ref{thm:deltaY}, we have that
\[
\delta(Y , \mathfrak{s}_\phi) - \delta(Y , \mathfrak{s}_{\phi'}) = \begin{cases} 0 & \text{if } \mathfrak{s}_\phi \text{ has non-trivial fibre holonomy} \\
-\eta^{pin^c}(\Sigma , \mathfrak{s}^\Sigma_\phi) & \text{if } \mathfrak{s}_\phi \text{ has trivial fibre holonomy}. \end{cases}
\]
where $\mathfrak{s}^\Sigma_\phi$ is the pin$^c$-structure on $\Sigma$ induced by $\mathfrak{s}_\phi$. Comparing these two, we see that $\delta(Y , \mathfrak{s}_\phi) - \delta(Y , \mathfrak{s}_{\phi'}) = 0$ if and only if $\nu = 1$ hence $\nu = 1$ must correspond to non-trivial fibre holonomy and $\nu = 0$ to trivial fibre holonomy. Then in the case $\nu = 0$, we have
\[
\eta^{pin^c}(\Sigma , \mathfrak{s}^\Sigma_\phi) = \begin{cases} -1/2 & 0 \le u \le m-1, \\ 1/2 & m \le u \le 2m-1. \end{cases}
\]

\subsection{$\mathbb{RP}^2(2,2)$}\label{sec:rp22}

In this section we will compute $\eta^{pin^c}$ in the case of $\Sigma = \mathbb{RP}^2(2,2)$ using methods independent of Theorem \ref{thm:etapinc}. In doing so we will also compute the delta invariants of the Hantzsche--Wendt manifold.

Since $\Sigma = \mathbb{RP}(2,2)$ has orbifold characteristic zero it admits a flat orbifold metric. Consider $Y = S( 1 ; (2,1) , (2,1) )$. Since $l = deg(N) = 1 - 1/2 - 1/2 = 0$, $N$ admits a flat connection and then $Y$ admits a flat metric. So $Y$ is a flat $3$-manifold which is also a rational homology $3$-sphere. Up to diffeomorphism there is only one such $3$-manifold, the Hantzsche--Wendt manifold.

There are $16$ spin$^c$-structures on $Y$, $8$ with trivial fibre holonomy and $8$ with non-trivial fibre holonomy. In the case of non-trivial fibre holonomy, Theorem \ref{thm:deltaY} gives
\[
\delta(Y, \mathfrak{s}_{(E_0 , \tau , -1)}) = -\sum_{i=1}^2 \delta( L(2,1) , \mathfrak{s}_{\gamma_i} ),
\]
where $[E_0] = (-1 ; \gamma_1 , \gamma_1 , \gamma_2 , \gamma_2)$ for $\gamma_1,\gamma_2 \in \mathbb{Z}_2$. We have $L(2,1) = \mathbb{RP}^3$ and the delta invariants for $L(2,1)$ are $\pm 1/8$. So the delta invariants with non-trivial fibre holonomy are:
\[
-1/4 , -1/4 , 0 , 0 , 0 , 0 , 1/4 , 1/4.
\]

In the case of trivial fibre holonomy, $\gamma_i + \delta_i = 1 \; ({\rm mod} \; 2)$, so $\mathfrak{s}_{\gamma_i} , \mathfrak{s}_{\delta_i}$ are the two spin$^c$-structures on $L(2,1)$. Hence $\delta(L(2,1) , \mathfrak{s}_{\gamma_i}) + \delta( L(2,1) , \mathfrak{s}_{\delta_i}) = 1/8 - 1/8 = 0$. Therefore, from Theorem \ref{thm:deltaY} we get
\begin{equation}\label{equ:deltaeta}
\delta(Y , \mathfrak{s}_{(E_0,\tau,0)}) = -\frac{1}{2} \eta^{pin^{c}}(\Sigma , \mathfrak{s}^{\Sigma}_{(E_0 , \tau)}).
\end{equation}

We can construct $Y$ as the quotient $Y = T^3/ \mathbb{Z}_2 \times \mathbb{Z}_2$, where $T^3 = S^1 \times S^1 \times S^1$ (where $S^1$ is the unit circle in $\mathbb{C}$) and $\mathbb{Z}_2 \times \mathbb{Z}_2 = \langle m,n \rangle$,
\[
m(x,y,z) = (-\overline{x} , \overline{y} , -z), \quad n(x,y,z) = ( -x , -\overline{y} , \overline{z} ), \quad mn(x,y,z) = (\overline{x} , -y , -\overline{z}).
\]

Let $\Lambda = \mathbb{Z}^3 = H^1(T^3 ; \mathbb{Z})$. The action of $m,n$ on $\Lambda$ is
\[
m(x,y,z) = (-x,-y,z), \quad n(x,-y,-z), \quad nm(x,y,z) = (-x,y,-z).
\]
The action on $H^2(T^3 ; \mathbb{Z}) \cong H_1(T^3 ; \mathbb{Z})$ can also be identified with $\Lambda$.

We have that $Y = \mathbb{R}^3/\Gamma$, where $\Gamma = \pi_1(Y)$ is the group generated by $m,n,e_1,e_2,e_3$, where
\[
m(x,y,z) = (-x + 1/2 , -y , z+1/2), \quad n(x,y,z) = (x+1/2 , -y+1/2 , -z)
\]
\[
e_1(x,y,z) = (x+1,y,z), \quad e_2(x,y,z) = (x,y+1,z), \quad e_3(x,y,z) = (x,y,z+1).
\]
The relations satisfied by these elements is as follows: $e_1,e_2,e_3$ commute, $me_1m^{-1} = -e_1$, $m e_2 m^{-1} = -e_2$, $me_3 m^{-1} = e_3$, $ne_1n^{-1} = e_1$, $ne_2n^{-1} = -e_2$, $ne_3 n^{-1} = -e_3$, $m^2 = e_3$, $n^2 = e_1$, $nm = e_1e_2e_3^{-1} mn$.

By taking the abelianisation one then finds $H_1(Y ; \mathbb{Z}) \cong \mathbb{Z}_4 \oplus \mathbb{Z}_4$, where $m = (1,0), n = (0,1)$.

Consider the action of $\mathbb{Z}_3$ on $\mathbb{R}^3$ given by $\tau(x,y,z) = (z,x,y)$. We have $\tau^{-1} m \tau = e_2 mn$, $\tau^{-1}n \tau = m$, $\tau^{-1}e_i \tau = e_{i-1}$. Thus $\tau$ descends to an orientation preserving $\mathbb{Z}_3$-action on $Y$. This $\mathbb{Z}_3$ action acts on $H_1(Y ; \mathbb{Z})$ according to $\tau [m] = [n]$, $\tau[n] = -[m]-[n]$.

The set of spin structures is a torsor for $H^1(Y ; \mathbb{Z}_2) = Hom( H_1(Y ; \mathbb{Z}) , \mathbb{Z}_2) \cong \mathbb{Z}_2^2$. So there are four spin structures. $\tau$ defines an orientation preserving diffeomorphism of $Y$, so it must act on the set of spin structures. Since $\tau$ has order $3$, the orbits of $\tau$ have size $1$ or $3$. Since the number of spin structures is $4$, which is not a multiple of $3$, there must be at least one spin structure fixed by $\tau$ (in fact, considering the action of $\tau$ on $H^1(Y ; \mathbb{Z}_2)$, we see that there is exactly one spin structure fixed by $\tau$). Let $\mathfrak{s}_0$ denote the underlying spin$^c$-structure of the spin structure fixed by $\tau$. So $\mathfrak{s}_0$ is also fixed by $\tau$. We can then identify the set of spin$^c$-structures with $H^2(Y ; \mathbb{Z})$ via $L \mapsto \mathfrak{s}_L = L \otimes \mathfrak{s}_0$. Under this identification, the action of $\tau$ on the set of spin$^c$-structures coincides with the action of $\tau$ on $H^2(Y ; \mathbb{Z}) \cong \mathbb{Z}_4^2$. There are six orbits, namely
\begin{align*}
& \{ (0,0) \}, \{ (1,0),(0,1),(3,3) \} , \{ (2,0) , (0,2) , (2,2) \}, \\
& \{ (3,0) , (0,3) , (1,1) \} , \{ (1,2) , (2,3) , (1,3) \} , \{ (2,1) ,  (3,1) , (3,2) \}.
\end{align*}

Call these orbits $O_1, \dots , O_6$. By diffeomorphism invariance, $\delta(Y , \mathfrak{s})$ must be a constant on each orbit. We will write $\delta(Y , O_i)$ for the value of $\delta$ on $O_i$.

Consider now the map $s\colon  \mathbb{R}^3 \to \mathbb{R}^3$ given by $s(x,y,z) = (x,y,z+1/2)$. Then $s$ is an orientation preserving diffeomorphism of $\mathbb{R}^3$, $s m s^{-1} = m$, $s n s^{-1} = e_3 n$, $se_is^{-1} = e_i$. Thus $s$ descends to an orientation preserving diffeomorphism of $Y$. The action on $H_1(Y ; \mathbb{Z})$ is given by
\[
s[m] = [m], \quad s[n] = [n] + 2[m].
\]

Consider the action of $s$ on spin$^c$-structures. Since $s$ induces an involution on $Y$ it must also induce an involution on the set of spin$^c$-structures. We have $s(\mathfrak{s}_0) = x \otimes \mathfrak{s}_0$ for some $x \in H^2(Y ; \mathbb{Z})$. Since $c(\mathfrak{s}_0) = 0$, we must also have $c( s(\mathfrak{s}_0) ) = 0$ and thus $2x = 0$. So $x = (0,0), (2,0) , (0,2)$ or $(2,2)$. Since $s$ is an involution, $\mathfrak{s}_0 = s ( s(\mathfrak{s}_0) ) = s( x \otimes \mathfrak{s}_0) = s(x) \otimes s(\mathfrak{s}_0) = (s(x) + x) \otimes \mathfrak{s}_0$. Thus $s(x)  + x = 0$. This leaves only two possibilities $x = 0$ or $x = (2,0)$. We will prove that $x = (2,0)$

Assume for sake of contradiction that $x = 0$. A short calculation shows that the action of $s$ and $\tau$ on the set of spin$^c$-structures has three orbits, of sizes $1,3$ and $12$. The delta invariants for the spin$^c$-structures with non-trivial holonomy are $\{-1/4, -1/4 , 0 , 0 , 0 , 0 , 1/4 , 1/4\}$ and from Equation (\ref{equ:deltaeta}), the delta invariant for the spin$^c$-structures with trivial holonomy are of the form $\{ \alpha_1 , -\alpha_1 , \alpha_2, -\alpha_2, \alpha_3 , -\alpha_3,$ $ \alpha_4 , -\alpha_4\}$ for some $\alpha_1, \dots , \alpha_4$. Since there are only three orbits, we must have $\alpha_i \in \{0 , 1/4 , -1/4\}$ for each $i$. Let $0 \le k \le 4$ be the number of $\alpha_i$ which are zero. Let $n(i)$ be the number of spin$^c$-structures on $Y$ for which the delta invariant equals $i$. Then $( n(1/4) , n(-1/4) , n(0) ) = ( 6-k , 6-k , 4+2k)$. But the three orbits have distinct sizes $1,3,12$, which contradicts $n(1/4) = n(-1/4)$. Therefore we must have $x = (2,0)$.

Given that $s(\mathfrak{s}_0) = (2,0) \otimes \mathfrak{s}_0$, the action of $s$ on the set of spin$^c$-structures is easily determined. One finds that the action of the group generated by $\tau$ and $s$ on the set of spin$^c$-structures has four orbits: $O_1$, $O_2 \cup O_4$, $O_3$ and $O_5 \cup O_6$. It follows that the delta invariant can take on at most four values $\delta(Y , O_1) = a , \delta(Y , O_2) = \delta(Y , O_4) = b$, $\delta(Y , O_3) = c$, $\delta(Y , O_5) = \delta(Y , O_6) = d$ for some $a,b,c,d$. 

For the spin$^c$-structures with non-trivial fibre holonomy, we get that the delta invariants are $0$, $1/4$ or $-1/4$. For the spin$^c$-structures with trivial fibre holonomy, the delta invariants are $\pm \alpha_1 , \pm \alpha_2 , \pm \alpha_3 , \pm \alpha_4 $ for some $\alpha_1, \dots , \alpha_4$. Since there are only four orbits we must have $\alpha_i \in \{ 0 , 1/4 , -1/4\}$ for each $i$. Once again, letting $k$ be the number of $\alpha_i$ which is zero we get that $(n(1/4) , n(-1/4) , n(0) = (6-k , 6-k , 4+2k)$. The four orbits have sizes $1,3,6,6$. Since $n(1/4) =n(-1/4)$, the only way this can happen is $n(1/4) = n(-1/4) = 6$. So $k=0$ and $n(0) = 4$. Hence the complete set of delta invariants for $Y$ is:
\[
0 \text{ (four times)}, \quad 1/4 \text{ (six times)}, \quad -1/4 \text{ (six times)}.
\]
In particular, the delta invariants corresponding to the spin$^c$-structures with trivial fibre holonomy are $\pm 1/4$. Comparing with Equation (\ref{equ:deltaeta}), we get that $\eta^{pin^c}(\Sigma , \mathfrak{s}^\Sigma) = \pm 1/2$, in agreement with Theorem \ref{thm:etapinc}.


\bibliographystyle{amsplain}

\end{document}